\documentclass[12pt]{article}
\usepackage{amssymb,amsmath,amsthm,ascmac}
\usepackage{enumerate}
\usepackage{array}

\newtheorem{thm}{Theorem}

\newtheorem{lem}{Lemma}[section]

\newtheorem{rem}[lem]{Remark}

\newcommand{\dis}{\displaystyle}

\newcommand{\R}{{\Bbb R}}
\newcommand{\N}{{\Bbb N}}

\newcommand{\pa}{\partial}

\makeatletter

\@addtoreset{equation}{section}
\makeatother

\usepackage{fancyhdr}
\title{\Large\sf Oscillatory behavior of solutions to the critical Fujita equation in 6D}
\author{Junichi Harada
\\{\small Faculty of Education and Human Studies, Akita University}
\\[1mm]{\small email: harada-j@math.akita-u.ac.jp}}

\begin{document}
\maketitle
\tableofcontents
\thispagestyle{empty}

 \begin{abstract}
 Long time dynamics of solutions to the 6D energy critical heat equation
 $u_t=\Delta u+|u|^{p-1}u$ on $\R^6\times(0,\infty)$ is investigated.
 It is shown that
 there exists a radially symmetric global solution $u(x,t)\in C([0,\infty);\dot H^1(\R^6))$
 of the form
 \begin{align*}
 u(x,t)
 =
 \lambda(t)^{-\frac{n-2}{2}}
 {\sf Q}(\tfrac{x}{\lambda(t)})
 +
 \text{error}
 (x,t),
 \end{align*}
 where the function \( \lambda(t) \) satisfies:
 \begin{itemize}
 \item 
 $\dis\lim_{t\to\infty}\|\text{error}(\cdot,t)\|_{\dot H_x^1(\R^6)}=0$,
 \item
 $\dis\liminf_{t\to\infty}\lambda(t)=0$,
 \item
 $\dis\limsup_{t\to\infty}\lambda(t)=\infty$.
 \end{itemize}
 The solutions constructed here
 demonstrate that the dynamical behavior in 
 \( \dot H^1(\mathbb{R}^n) \) can differ significantly 
 from the behavior in \( H^1(\mathbb{R}^n) \).  
 \end{abstract}


\section{Introduction}
 This paper deals with the Fujita equation:
 \begin{equation}\label{equation_1.1}
 \begin{cases}
 u_t = \Delta u+|u|^{p-1}u
 &
 \text{in } \R^n\times(0,\infty),\\
 u(x,0)=u_0(x)
 &
 \text{on } \R^n
 \end{cases}
 \end{equation}
 with the Sobolev critical case
 \[
 n\geq3
 \quad
 \text{and}
 \quad
 p=\tfrac{n+2}{n-2}.
 \]
 For any initial data $u_0\in\dot H^1(\R^n)$,
 equation \eqref{equation_1.1} admits a unique local in time solution $u(x,t)$ of \eqref{equation_1.1} satisfying
 \[
 u(t)\in C([0,T);\dot H^1(\R^n))\cap C((0,T);L^\infty(\R^n))
 \]
 for some $T>0$.
 If there exists $T>0$ such that $\limsup_{t\to T}\|u(t)\|_\infty=\infty$,
 we say that the solution blows up at finite time \( T \).
 In this paper, we focus on the long-time behavior of global solutions to \eqref{equation_1.1}.  
In the Sobolev critical case \( p = \frac{n+2}{n-2} \),  
the dynamics are strongly influenced by the so-called \emph{ground state},  
which is defined as a minimizer of the energy functional
\[
E[u] = \frac{1}{2} \int_{\mathbb{R}^n} |\nabla u(x)|^2 \, dx
- \frac{1}{p+1} \int_{\mathbb{R}^n} |u(x)|^{p+1} \, dx,
\]
among all functions \( u \in \dot{H}^1(\mathbb{R}^n) \setminus \{0\} \)  
that satisfy the corresponding Euler-Lagrange equation
\[
- \Delta u = |u|^{p-1} u \quad \text{in } \mathbb{R}^n.
\]
 It is well known that
 the gourd state is given by the family of Aubin-Talenti functions:
 \[
 \lambda^{-\frac{n-2}{2}}
 {\sf Q}\left(\tfrac{x-z}{\lambda}\right)
 \]
 where
 $\lambda>0$,
 $z\in\R^n$
 and
 \begin{align*}
 {\sf Q}(x)
 =
 \left(1+\tfrac{|x|^2}{n(n-2)}\right)^{-\frac{n-2}{2}}.
 \end{align*}
 Equation \eqref{equation_1.1} has an energy structure,  
 and any sufficiently regular solution \( u(x,t) \) satisfies the energy inequality
\begin{align*}
\frac{d}{dt} E[u(t)]
=
- \int_{\R^n} |u_t(x,t)|^2 \, dx
< 0.
\end{align*}
From this inequality,  
we see that if a global solution has uniformly bounded energy for all time,  
then its energy converges to a constant value as \( t \to \infty \).  
In particular, we are interested in the case
\[
\lim_{t \to \infty} E[u(t)] = E({\sf Q}).
\]
The most straightforward scenario is the convergence to a ground state profile:
there exist constants \( \lambda_\infty > 0 \) and \( z_\infty \in \mathbb{R}^n \) such that
\begin{align}
\label{equation_1.2}
\lim_{t \to \infty} u(x,t)
=
\lambda_\infty^{-\frac{n-2}{2}}
{\sf Q}\left( \tfrac{x - z_\infty}{\lambda_\infty} \right).
\end{align}
However, the Fila-King conjecture \cite{Fila-King} suggests that  
there exist more complex solutions that do not simply converge  
to a static Aubin–Talenti profile as \( t \to \infty \).  
Indeed, they formally constructed positive, radially symmetric solutions  
whose \( L^\infty \)-norms exhibit the following asymptotic behavior:
 \begin{table}[h]
 \hspace{20mm}
 \renewcommand{\arraystretch}{1.0}
 \begin{tabular}{p{3em}|p{6em}p{5em}p{3em}c}
 & \centering $\frac{n-2}{2}<\gamma<2$ & \centering $\gamma=2$ & \centering $\gamma>2$ &
 \\ \hline
 $n=3$ & \centering $t^\frac{\gamma-1}{2}$ & \centering $t^{\frac{1}{2}}(\log t)^{-1}$
 & \centering $t^{\frac{1}{2}}$ &
 \\
 $n=4$ & \centering $t^{-\frac{2-\gamma}{2}}\log t$ & \centering $1$ &
 \centering $\log t$ &
 \\
 $n=5$ & \centering $t^{-\frac{3(2-\gamma)}{2}}$ & \centering $(\log t)^{-3}$
 & \centering $1$ &
 \\ \hline
 \end{tabular}
 \\[2mm]
 \hspace{20mm}
 If $n\geq6$ and $\gamma>\frac{n-2}{2}$,
 then
 $\|u(t)\|_\infty=1$.

 \caption{Fila-King Conjecture (Growth rate of $L^\infty$-norm of the solution)}
 \label{table_1}
 \end{table}
 \\
 Here, the parameter \( \gamma \) characterizes the spatial decay of the initial data \( u_0(x) \), specifically:
 \[
 \lim_{|x| \to \infty} |x|^\gamma u_0(x) = A > 0.
 \]
 In all cases shown in Table~\ref{table_1},  
 the solution has the same asymptotic profile:
 \begin{align}
 \label{equation_1.3}
 u(x,t)
 &=
 \lambda(t)^{-\frac{n-2}{2}} {\sf Q} \left( \tfrac{x}{\lambda(t)} \right) + \text{error}(x,t), 
 \\
 \nonumber
 &
 \text{where }
 \|\text{error}(\cdot, t)\|_{\dot{H}^1(\mathbb{R}^n)} \to 0 \quad \text{as } t \to \infty.
 \end{align}
Most of the cases listed in Table~\ref{table_1}
have been affirmatively resolved in various works,
where solutions with the prescribed asymptotic profiles described in the table were constructed.
Below is a summary of the known results:
 \begin{itemize}
 \item $n=3$ ($\gamma>1$) \ by Manuel del Pino - Monica Musso - Juncheng Wei in \cite{delPino-Musso-Wei_3dim_long}
 \item $n=4$ ($\gamma>2$) \ by Juncheng Wei - Qidi Zhang - Yifu Zhou in \cite{Wei-Qidi-Yifu}
 \item $n=4$ ($1<\gamma\leq2$) \ by Juncheng Wei - Yifu Zhou in \cite{Wei-Yifu}
 \item $n=5$ ($\gamma>\frac{3}{2}$) \ by Zaizheng Li - Juncheng Wei - Qidi Zhang - Yifu Zhou in \cite{Li-Wei-Qidi-Yifu}
 \item $n=6$ ($\gamma>2$) \ by Juncheng Wei - Yifu Zhou in \cite{Wei-Yifu}
 \end{itemize}
 In \cite{Wei-Yifu},
 Juncheng Wei - Yifu Zhou did not include the full proofs for the remaining cases $n=3$ with $\frac{1}{2}<\gamma\leq1$ and
 $n\geq7$ with $\gamma>\frac{n-2}{2}$, but they provided remarks on these cases.

On the other hand, for \( n \geq 7 \),
Charles Collot - Frank Merle - Pierre Rapha\"el in \cite{Collot-Merle-Raphael}
established a complete classification of the dynamics  
of solutions near \( {\sf Q}(x) \).
(We omit the full statement of their theorems here.)
Their results imply that
if a solution \( u(x,t) \) behaves as in \eqref{equation_1.3} for all \( t > 0 \),
then the solution must converge to a static Aubin-Talenti profile
as \( t \to \infty \).
That is, \eqref{equation_1.2} holds
for some \( \lambda_\infty \in (0, \infty) \)
and \( z_\infty \in \mathbb{R}^n \).

A central goal of this paper is to construct three types of global, radially symmetric solutions of the form \eqref{equation_1.3}
in the 6D case.
The first solution is positive and exhibits vanishing behavior as $t\to\infty$;
the second is sign-changing and concentrates at a point as $t\to\infty$;
the third is sign-changing and exhibits oscillatory behavior, alternating between vanishing and concentration as $t\to\infty$.
All of these solutions are novel and are not included among the cases listed in Table \ref{table_1}.

\section{Main Result}
\label{section_2}
 Let ${\sf Q}(x)$ denote the Aubin-Talenti function,
 defined by
 \begin{align*}
 {\sf Q}(x)
 =
 ( 1+\tfrac{|x|^2}{n(n-2)} )^{-\frac{n-2}{2}},
 \end{align*}
 and let $\chi(s):[0,\infty)\to\R$ be a smooth cut off function defined by
 \[
 \chi(s)
 =
 \begin{cases}
 1 & \text{for } 0\leq s\leq1, \\
 0 & \text{for } s>2.
 \end{cases}
 \]
 Furthermore,
 let $\theta_\beta(x,t)$ denote the unique solution of the heat equation:
 \begin{align*}
 \begin{cases}
 \pa_t\theta=\Delta_x\theta
 & \text{in } \R^n\times(0,\infty),
 \\
 \theta(x,0)
 =
 \theta_0^{(\beta)}(x)
 & \text{on } \R^n,
 \end{cases}
 \end{align*}
 where the initial data $\theta_0^{(\beta)}(x)$ is a positive, radially symmetric function defined by
 \begin{align*}
 \theta_0^{(\beta)}(x)
 =
 \begin{cases}
 |x|^{-2}(\log|x|)^{-\beta} & \text{if } |x|>e,
 \\
 e^{-2} & \text{if } |x|<e.
 \end{cases}
 \end{align*}
 Note that
 \begin{align*}
 \theta_0^{(\beta)}(x)\in \dot H^1(\R^6)
 \quad
 \text{if }
 \beta>\tfrac{1}{2}. 
 \end{align*}
 Therefore,
 for the case $\beta>\frac{1}{2}$,
 it satisfies
 (see Lemma \ref{lemma_4.4})
 \begin{align}
 \label{equation_2.1}
 \lim_{t\to\infty}
 \theta_\beta(x,t)
 =0
 \quad
 \text{in }
 \dot H^1(\R^6).
 \end{align}
 We define the function space in which the solution $u(x,t)$ is constructed.
 \begin{enumerate}[\rm(c1)]
 \item
 $u(\cdot,t)\in C([0,\infty);\dot H^1(\R^6))$,
 \item
 $u(x,t)\in C(\R^6\times[0,\infty))\cap C^{2,1}(\R^6\times(0,\infty))$,
 \item
 $u(x,t)\in L^\infty(\R^6\times(0,T))$
 \quad
 for any
 $T>0$.
 \end{enumerate}
 We now present three theorems that describe different long-time behaviors of solutions,
 depending on the structure of the initial data.
 \begin{thm}[Positive $\dot H^1$-Solution]
 \label{theorem_1}
 Let $n=6$ and $p=\frac{n+2}{n-2}$.
 For any  $\beta\in(\frac{1}{2},1)$,
 there exist a constant $t_I>0$ and
 a global, positive, radially symmetric solution solution $u_\beta(x,t)$ satisfying {\rm (c1) - (c3)},
 such that
 \begin{align*}
 u_{\beta}(x,t)
 =
 \lambda_{\beta}(t)^{-\frac{n-2}{2}}
 {\sf Q}(\tfrac{x}{\lambda_\beta(t)})
 \chi_1
 +
 \theta_\beta(x,t+t_I)
 \cdot
 (1-\chi_1)
 +
 \text{\rm error}_1(x,t),
 \end{align*}
 where
 $\lambda_\beta(t)$, $\chi_1$ and $\text{\rm error}_1(x,t)$ are given as follows{\rm:}
 \begin{enumerate}[\rm(i)]
 \item
 The function $\lambda_{\beta}(t)\in C^1((0,\infty))\cap C([0,\infty))$ satisfies
 \begin{align*}
 \lambda_\beta(t)
 =
 (1+o(1))e^{\frac{5A_1}{4(1-\beta)}(\log t)^{1-\beta}}
 \quad
 \text{as }
 t\to\infty,
 \end{align*}
 where 
 $A_1$ is a positive constant{\rm;}

 \item
 The cut off function $\chi_1$ is defined by
 $\chi_1=\chi(\frac{|x|}{\lambda_\beta(t)^\kappa(\sqrt t)^{1-\kappa}})$
 for some $\kappa\in(0,\frac{1}{2})${\rm;}
 \item
 The error term $\text{\rm error}_1(x,t)$ satisfies
 $\dis\lim_{t\to\infty}\|\text{\rm error}_1(\cdot,t)\|_{\dot H_x^1(\R^6)}=0$.
 \end{enumerate}
 \end{thm}

 \begin{thm}[Sign-Changing $\dot H^1$-Solution]
 \label{theorem_2}
 Let $n=6$ and $p=\frac{n+2}{n-2}$.
 For any  $\beta\in(\frac{1}{2},1)$,
 there exist a constant $t_I>0$ and
 a global, sign changing, radially symmetric solution solution $u_\beta(x,t)$ satisfying {\rm (c1) - (c3)},
 such that
 \begin{align*}
 u_{\beta}(x,t)
 =
 \lambda_{\beta}(t)^{-\frac{n-2}{2}}
 {\sf Q}(\tfrac{x}{\lambda_\beta(t)})
 \chi_1
 -
 \theta_\beta(x,t+t_I)
 \cdot
 (1-\chi_1)
 +
 \text{\rm error}_2(x,t),
 \end{align*}
 where
 $\lambda_\beta(t)$, $\chi_1$ and $\text{\rm error}_1(x,t)$ are given as follows{\rm:}
 \begin{enumerate}[\rm(i)]
 \item
 The function $\lambda_{\beta}(t)\in C^1((0,\infty))\cap C([0,\infty))$ satisfies
 \begin{align*}
 \lambda_\beta(t)
 =
 (1+o(1))e^{-\frac{5A_1}{4(1-\beta)}(\log t)^{1-\beta}}
 \quad
 \text{as }
 t\to\infty,
 \end{align*}
 where 
 $A_1$ is the same positive constant as in {\rm(i)} of Theorem \ref{theorem_1}{\rm;}

 \item
 $\chi_1$
 is the same cut off function as in Theorem \ref{theorem_1}{\rm;}

 \item
 The error term \( \mathrm{error}_2(x,t) \) satisfies the same estimate as in {\rm(iii)} of Theorem \ref{theorem_1}.
 \end{enumerate}
 \end{thm}

 In preparation for the next result,
 we define a new function $\Theta_\beta(x,t)$,
 which will play a central role in our construction.
 Let $\{R_j\}_{j=1}^\infty$ be a sequence of positive numbers satisfying $R_j\to\infty$ as $j\to\infty$.
 We carefully  choose a radially symmetric initial function $\Theta_0^{(\beta)}(x)$
 so that it satisfies the following properties:
 \begin{align}
 \label{equation_2.2}
 \Theta_0^{(\beta)}(x)
 &=
 \begin{cases}
 e^{-\frac{n-2}{2}}
 & \text{for } |x|<e,
 \\
 |x|^{-2}
 (\log|x|)^{-\beta}
 & \text{for } e<|x|<R_1,
 \\
 -|x|^{-2}
 (\log|x|)^{-\beta}
 & \text{for } 2R_1<|x|<R_2,
 \\
 |x|^{-2}
 (\log|x|)^{-\beta}
 & \text{for } 2R_2<|x|<R_3,
 \\
 -
 |x|^{-2}
 (\log|x|)^{-\beta}
 & \text{for } 2R_3<|x|<R_4,
 \\
 \cdots
 \end{cases}
 \end{align}
 In the remaining regions of \( x\in\mathbb{R}^n \),  
 the function \( \Theta_0^{(\beta)}(x) \) is defined so as to smoothly interpolate  
 between these alternating regions.
 Let $\Theta_\beta(x,t)$ denote the unique solution of
 $\pa_t\Theta=\Delta_x\Theta$ on $(x,t)\in\R^n\times(0,\infty)$
 with the initial dat $\Theta(x,0)=\Theta_0^{(\beta)}(x)$.
 By the same reason as in the case $\theta_\beta(x,t)$ (see \eqref{equation_2.1}),
 we have
 \begin{align}
 \label{equation_2.3}
 \lim_{t\to\infty}
 \Theta_\beta(x,t)
 =0
 \quad
 \text{in }
 \dot H^1(\R^6).
 \end{align}

 \begin{thm}[Oscillatory $\dot H^1$-Solution]
 \label{theorem_3}
 Let $n=6$ and $p=\frac{n+2}{n-2}$.
  For any  $\beta\in(\frac{1}{2},1)$,
 there exist a constant $t_I>0$ and
 a global, sign changing, radially symmetric solution satisfying {\rm (c1) - (c3)},
 such that
 \begin{align*}
 u(x,t)
 =
 \lambda_\beta(t)^{-\frac{n-2}{2}}
 {\sf Q}(\tfrac{x}{\lambda_\beta(t)})
 \chi_1
 +
 \Theta_\beta(x,t+t_I)
 \cdot
 (1-\chi_1)
 +
 \text{\rm error}_3(x,t),
 \end{align*}
 where
 $\lambda_\beta(t)$, $\chi_1$ and $\text{\rm error}_3(x,t)$ are given as follows{\rm:}
 \begin{enumerate}[\rm(i)]
 \item
 The function $\lambda_\beta(t)\in C^1((0,\infty))\cap C([0,\infty))$ satisfies
 \begin{align*}
 \liminf_{t\to\infty}
 \lambda_\beta(t)=0
 \quad
 \text{ and }
 \quad
 \limsup_{t\to\infty}
 \lambda_\beta(t)=\infty.
 \end{align*}

 \item
 $\chi_1$
 is the same cut off function as in Theorem \ref{theorem_1}{\rm;}
 
 \item
 The error term \( \mathrm{error}_3(x,t) \) satisfies the same estimate as in {\rm(iii)} of Theorem \ref{theorem_1}.
 \end{enumerate}
 \end{thm}
\begin{rem}
In Theorem~\ref{theorem_3}, the qualitative behavior of the solution is determined by the dominant contribution of the initial data \( \Theta_0^{(\beta)}(x) \).
If the positive part of \( \Theta_0^{(\beta)}(x) \) dominates, the solution behaves similarly to the one described in Theorem~\ref{theorem_1}.
On the other hand, if the negative part dominates, the solution exhibits behavior as in Theorem~\ref{theorem_2}.
The initial data \( \Theta_0^{(\beta)}(x) \) is carefully constructed so that positive and negative regions alternate in space.
As a consequence, the solution exhibits oscillatory behavior in time.
\end{rem}
\begin{rem}
 We emphasize that
 the initial data used in Theorems \ref{theorem_1} {\rm-} \ref{theorem_3}
 satisfy
 $u_0(x)\in\dot H^1(\R^6)$,
 but $u_0\not\in L^q(\R^n)$ for any $q\in[1,\frac{n+2}{n-2})$.
 In {\rm\cite{Harada_6Ddynamics}},
 the author proves the following result{\rm:}
 Let $n=6$.
 If the initial data $u_0$ satisfies
 \begin{align*}
 u_0\in H^1(\R^6)
 \quad
 \text{and}
 \quad
 \|u_0-{\sf Q}\|_{\dot H^1(\R^6)}
 \ll1,
 \end{align*}
 then the corresponding global solution of \eqref{equation_1.1}
 exhibits exactly one of the following two types of behavior{\rm:}
 \begin{enumerate}
 \item $\dis \lim_{t\to\infty}u(x,t)=0$ in $\dot H^1(\R^6)${\rm;}
 \item There exist $\lambda_\infty\in(0,\infty)$ and $z_\infty\in\R^n$ such that
 $\dis\lim_{t\to\infty}u(x,t)=\lambda_\infty^{-\frac{n-2}{2}}{\sf Q}(\tfrac{x-z_\infty}{\lambda_\infty})$
 in $\dot H^1(\R^6)$.
 \end{enumerate}
 Therefore,
 solutions constructed Theorem \ref{theorem_1} {\rm-} \ref{theorem_3}
 demonstrate that the dynamical behavior in 
 \( \dot H^1(\mathbb{R}^n) \) can differ significantly 
 from the behavior in \( H^1(\mathbb{R}^n) \).  
 Such a phenomenon appears to occur only in the 
 critical 6D case.
\end{rem}
 \begin{rem}
 Oscillatory behavior of solutions to \eqref{equation_1.1} has already been observed 
 in the JL-supercritical case by Peter Pol\'a\v{c}ik - Eiji Yanagida 
 {\rm\cite{Polacik-Yanagida_2003}} {\rm(}see also {\rm\cite{Polacik-Yanagida_2009})}. 
 The essential difference between their results and ours is that 
 our solutions have finite energy, 
 whereas their solutions are not well-defined in the energy space, 
 since the corresponding energy diverges to infinity.
 \end{rem}
 \begin{rem}
It is known that the $D$-equivariant harmonic map heat flow
from \( \mathbb{R}^2 \) to \( \mathbb{S}^2 \subset \mathbb{R}^3 \)  
is closely related to the radial semilinear heat equation \eqref{equation_1.1}
in spatial dimension \( n = 2D + 2 \)
{\rm(}see Kihyun Kim and Frank Merle~{\rm\cite{Kim-Merle}}
for recent work on harmonic map heat flow{\rm)}.
 Stephen Gustafson - Kenji Nakanishi - Tai-Peng Tsai
 investigated the asymptotic stability of solutions to the $2$-equivariant
 harmonic map heat flow from \( \mathbb{R}^2 \to \mathbb{S}^2 \subset \mathbb{R}^3 \)
 near the ground state in~{\rm\cite{Gustafson-Nakanishi-Tsai}},  
 and discovered the existence of solutions that exhibit oscillatory behavior.
 Our Theorem \ref{theorem_3} can be regarded as an analogue of their result
 in the context of the semilinear heat equation.
 \end{rem}
 The analytical framework employed in this paper essentially follows the approach developed in the 5D case by
 Zaizheng Li - Juncheng Wei - Qidi Zhang - Yifu Zhou \cite{Li-Wei-Qidi-Yifu}
 (see also \cite{delPino-Musso-Wei_3dim_long, Wei-Qidi-Yifu} for the 3D and 4D cases),
 and does not introduce any new technical tools.
 We first construct an approximate solution of the form
 \begin{align*}
 u(x,t)
 =
 \lambda(t)^{-\frac{n-2}{2}}
 {\sf Q}(\tfrac{x}{\lambda(t)})
 \chi_1
 -
 \theta(x,t)
 \cdot
 (1-\chi_1)
 +
 \text{\rm error}(x,t),
 \end{align*}
 using the matched asymptotic expansion technique.
 Once this approximation is established,
 we obtain the desired solution by applying the inner-outer gluing method.
 The novelty of our work lies in identifying a new class of approximate solutions, which in turn gives rise to a new type of solution behavior.
 In particular, we carry out a detailed asymptotic analysis of the linear heat equation,
 which governs the tail behavior of the solution.

In this paper,  
we present only the proof of Theorem \ref{theorem_3},  
which also covers, as special cases, the proofs of Theorems \ref{theorem_1} and \ref{theorem_2}.

\section{Preliminary}
\label{section_3}
\subsection{Notation}
\label{section_3.1}
 Let $B(x_0,R)$ denote the open ball in $\R^n$ centered at $x_0$ with radius $R$,
 that is
 \begin{align*}
 B(x_0,R)
 =
 \{x\in\R^n;|x-x_0|<R\}.
 \end{align*}
 We also use the notation
 \begin{align*}
 B_R
 =
 \{x\in\R^n;|x|<R\}
 \end{align*}
 to denote the open ball centered at the origin with radius $R$.
 We next introduce a smooth cut off function
 $\chi(\zeta):[0,\infty)\to\R$,
 which satisfies
 \[
 \chi(\zeta)=
 \begin{cases}
 1 & \text{if } 0<\zeta<1,\\
 0 & \text{if } \zeta>2.
 \end{cases}
 \]
 The nonlinear term in \eqref{equation_1.1} is given by
 \[
 f(u)=|u|^{p-1}u,
 \quad
 p=\tfrac{n+2}{n-2}.
 \]
 We define the operator $\Lambda_y$, associated with the transformation
 $u(x)\to\lambda^\frac{1}{p-1}u(\lambda x)$,
 by
 \begin{align*}
 \Lambda_y
 =
 \tfrac{n-2}{2}
 +
 y\cdot\nabla_y.
 \end{align*}

\subsection{Linearization around the Ground State}
\label{section_3.2}
 Let us consider the linearization around the ground state ${\sf Q}(y)$.
 Define
 \[
 H_y=\Delta_y+V(y)
 \quad
 \text{with }
 V(y)
 =
 p{\sf Q}(y)^{p-1}.
 \]
 An associated eigenvalue problem is given by
 \begin{align}
 \label{EQ_3.1}
 -H_y\psi=\mu\psi \quad\text{in } \R^n,
 \end{align}
 We recall that
 the operator $H_y$ has a unique negative eigenvalue $\mu_1=-e_0$, where $e_0>0$
 (see Proposition 5.5 p. 37 \cite{Duyckaerts}).
 Let $\psi_1(r)$ denote the corresponding positive radially symmetric eigenfunction
 associated with the negative eigenvalue, normalized so that $\psi_1(0)=1$.
 It is known that
 there exists $C>0$ such that (see p. 18 \cite{Cortazar-delPino-Musso})
 \[
 0<
 \psi_1(r)
 <
 C\left( 1+r \right)^{-\frac{n-1}{2}}
 e^{-r\sqrt{e_0}}
 \quad
 \text{for }
 r>0.
 \]
 In the radially symmetric setting,
 the kernel of $H_y$ is given by
 \begin{align*}
 \text{Ker } H_y
 =
 \text{span}
 \{
 (\Lambda_y{\sf Q})
 \}
 \quad
 \text{when }
 n\geq5.
 \end{align*}
 Furthermore,
 when $n\geq5$,
 it holds that
 \begin{align*}
 \int
 (-H_y\epsilon(y))\epsilon(y)
 dy
 >0
 \end{align*}
 for any radial functions $\epsilon\in H^1(\R^n)$ satisfying
 $\int\epsilon(y)\psi_1(y)dy=\int\epsilon(y)(\Lambda_y{\sf Q})(y)dy=0$.
 Note that
 for $n \geq 7$, the same inequality holds for any radial functions
 $\epsilon\in\dot H^1(\R^n)$ satisfying the same orthogonality conditions.

\subsection{Perturbed Eigenvalue Problem Associated with the Linearized Equation}
\label{section_3.3}
 Consider the eigenvalue problem \eqref{EQ_3.1} in a bounded but very large domain.
 \begin{equation}
 \label{EQ_3.2}
 \begin{cases}
 -H_y\psi=\mu\psi & \text{in } B_R,
 \\
 \psi=0 & \text{on } \pa B_R,
 \\
 \psi \text{ is radially symmetric}.
 \end{cases}
 \end{equation}
 Let $\mu_j^{(R)}$ denote the $j$-th eigenvalue of \eqref{EQ_3.2},
 and let $\psi_j^{(R)}$ be the associated eigenfunction.
 Most of the lemmas in this subsection are proved in Section 7 of
 \cite{Cortazar-delPino-Musso} (see also Section 3.2 of \cite{Harada_5D}).
 We recall that $\mu_1 = -e_0$ is the first eigenvalue of
 the operator $-H_y$ on $\mathbb{R}^n$, where $e_0 > 0$.
 \begin{lem}
 \label{lemma_3.1}
 There exists a constant $R_1 > 1$ such that for all $R > R_1$, we have
 \[
 \mu_1^{(R)}<0<\mu_2^{(R)}\leq\mu_3^{(R)}\leq\cdots.
 \]
 Furthermore
 it holds that
 \[
 \lim_{R\to\infty}\mu_1^{(R)}=-e_0.
 \]
 \end{lem}
 \begin{lem}[Lemma 3.1-Lemma 3.2 \cite{Harada_5D} p. 312]
 \label{lemma_3.2}
 There exist constants $c > 0$ and $R_1 > 1$ such that for all $R > R_1$,
 the following estimates hold{\rm:}
 \begin{align*}
 0
 <
 \psi_1^{(R)}(r)
 &<
 c
 \left( 1+r \right)^{-\frac{n-1}{2}}
 e^{-r\sqrt{e_0}}
 \quad\text{for } r\in(0,R),
 \\
 |\psi_2^{(R)}(r)|
 &<
 c
 \left( 1+r \right)^{-(n-2)}
 \quad\text{for } r\in(0,R).
 \end{align*}
 \end{lem}
 \begin{lem}[Lemma 7.2 \cite{Cortazar-delPino-Musso}, Lemma 3.3 \cite{Harada_5D} p. 313]
 \label{lemma_3.3}
 Let $n\geq5$.
 Then we have
 \[
 \inf_{R>1}R^{n-2}\mu_2^{(R)}>0.
 \]
 \end{lem}
 \begin{lem}[Lemma 3.3 \cite{Harada_double} p. 9]
 \label{lemma_3.4}
 Let $n\geq5$.
 Then we have
 \[
 \inf_{R>1}R^{\frac{n}{2}}\mu_3^{(R)}>0.
 \]
 \end{lem}


\subsection{$L^\infty$ Bounds for Parabolic Equations}
\label{section_3.4}
 In this subsection, we present several a priori estimates for solutions to parabolic equations.
 \begin{align}
 \label{EQ_3.3}
 u_t
 &=
 \Delta_xu+{\bf b}(x,t)\cdot\nabla_xu+V(x,t)u+f(x,t)
 \\
 \nonumber
 &
 \text{for } (x,t)\in B_2\times(0,1).
 \end{align}
 For $p,q\in[1,\infty]$,
 we define
 \[
 \|f\|_{L^{p,q}(B_2\times(t_1,t_2))}
 =
 \begin{cases}
 \dis
 \left( \int_{t_1}^{t_2}\|f(t)\|_{L^p(B_2)}^qdt \right)^\frac{1}{q}
 & \text{if } q\in[1,\infty),
 \\
 \dis
 \sup_{t\in(t_1,t_2)}\|f(t)\|_{L^p(B_2)}
 & \text{if } q=\infty.
 \end{cases}
 \]

 \begin{lem}[Exercise 6.5 p. 154 \cite{Lieberman}, Theorem 8.1 p. 192 \cite{Ladyzenskaja}]
 \label{lemma_3.5}
 Assume that
 \begin{itemize}
 \item
 $n\geq3${\rm;}
 \item 
 $p,q\in(1,\infty]$ and
 $\tfrac{n}{2p}+\tfrac{1}{q}<1${\rm;}
 \item
 The functions
 ${\bf b}(x,t)\in\{L^\infty(B_2\times(0,1))\}^n$ and
 $V(x,t)\in L^\infty(B_2\times(0,1))$ satisfy
 \begin{align*}
 \|{\bf b}\|_{L^\infty(B_2\times(0,1))}+\|V\|_{L^\infty(B_2\times(0,1))}<M.
 \end{align*}
 \end{itemize}
 Then there exists a constant $c>0$,
 depending only on $p$, $q$, $n$ and $M$,
 such that
 \begin{enumerate}[{\rm (i)}]
 \item
 if $u(x,t)$ is a weak solution of \eqref{EQ_3.3},
 then
 \[
 \|u\|_{L^\infty(B_1\times(\frac{1}{2},1))}
 <
 c( \|u\|_{L^2(B_2\times(0,1))}+\|f\|_{L^{p,q}(B_2\times(0,1))} ),
 \]
 
 \item
 if $u(x,t)\in C(\overline{B_2}\times[0,1))$ is a weak solution of \eqref{EQ_3.3} with
 \[
 u(x,0)=0
 \quad
 \text{for }
 x\in B_2,
 \]
 then
 \begin{align*}
 \|u\|_{L^\infty(B_1\times(0,1))}
 <
 c( \|u\|_{L^2(B_2\times(0,1))}+\|f\|_{L^{p,q}(B_2\times(0,1))} ) ,
 \end{align*}

 \item
 if $u(x,t)\in C(\overline{B_2}\times[0,1))$ is a weak solution of \eqref{EQ_3.3} with
 \[
 u(x,0)=0
 \quad
 \text{for }
 x\in B_2,
 \]
 then
 \begin{align*}
 \|u\|_{L^\infty(B_2\times(0,1))}
 &<
 \|u\|_{L^\infty(\pa B_2\times(0,1))}
 +
 c
 \|f\|_{L^{2,2}(B_2\times(0,1))}
 \\
 &\quad
 +
 c\|f\|_{L^{p,q}(B_2\times(0,1))}.
 \end{align*}
 \end{enumerate}
 \end{lem}

 \begin{lem}[Theorem 4.8 p. 56 \cite{Lieberman}, Theorem 11.1 p. 211 \cite{Ladyzenskaja}]
 \label{lemma_3.6}
 Assume that
 \begin{itemize}
 \item
 $n\geq3${\rm;}
 \item 
 $p,q\in(1,\infty]$ and
 $\tfrac{n}{p}+\tfrac{2}{q}<1${\rm;}
 \item
 The functions ${\bf b}(x,t)\in\{C(\overline{B_2}\times(0,1))\}^n$ and
 $V(x,t)\in C(\overline{B_2}\times(0,1))$ satisfy
 \begin{align*}
 \|{\bf b}\|_{L^\infty(B_2\times(0,1))}+\|V\|_{L^\infty(B_2\times(0,1))}<M.
 \end{align*}
 \end{itemize}
 Then
 there exists a constant $c>0$,
 depending only on $p$, $q$, $n$ and $M$,
 such that
 if $u(x,t)\in C^{2,1}(B_2\times(0,1))$ is a solution of \eqref{EQ_3.3},
 then
 \[
 \|\nabla_xu\|_{L^\infty(B_1\times(\frac{1}{2},1))}
 <
 c\left( \|u\|_{L^\infty(B_2\times(0,1))}+\|f\|_{L^{p,q}(B_2\times(0,1))} \right).
 \]
 \end{lem}

 \begin{lem}
 \label{lemma_3.7}
 Let $\delta\in(0,1)$.
 Assume that
 \begin{itemize}
 \item
 $n\geq3${\rm;}
 \item 
 $p,q\in(1,\infty]$ and
 $\tfrac{n}{p}+\tfrac{2}{q}<1${\rm;}
 \item
 The functions ${\bf b}(x,t)\in\{C(\overline{B_2}\times[0,\delta])\}^n$ and
 $V(x,t)\in C(\overline{B_2}\times[0,\delta])$ satisfy
 \begin{align*}
 \|{\bf b}\|_{L^\infty(B_2\times(0,\delta))}+\|V\|_{L^\infty(B_2\times(0,\delta))}<M.
 \end{align*}
 \end{itemize}
 Then
 there exists a constant $c>0$,
 depending on $p$, $q$, $n$ and $M$ {\rm(}but not on $\delta${\rm)},
 the following holds{\rm:}
 If $u(x,t)\in C^{2,1}(B_2\times(0,\delta))\cap C(\overline{B_2}\times[0,\delta])$
 is a solution of \eqref{EQ_3.3} satisfying the initial condition
 \begin{align*}
 u(x,0)=0
 \quad
 \text{for }
 x\in B_2,
 \end{align*}
 then
 \[
 \|\nabla_xu\|_{L^\infty(B_1\times(0,\delta))}
 <
 c\left( \|u\|_{L^\infty(B_2\times(0,\delta))}+\|f\|_{L^{p,q}(B_2\times(0,\delta))} \right).
 \]
 \end{lem}
 \begin{proof}
 Since we could not find a reference that states the theorem in exactly the same form, we include a proof for completeness.
 Assume that
 \begin{itemize}
 \item
 $u(x,t)\in C^\infty(B_2\times(0,\delta))\cap C(\overline{B_2}\times[0,\delta])$,
 \item
 ${\bf b}(x,t),V(x,t),f(x,t)\in C^\infty(\R^{n+1})$,
 \end{itemize}
 and further suppose that there exists $\nu\in(0,\delta)$ such that
 \begin{itemize}
 \item
 ${\bf b}(x,t)=0$ \quad for $(x,t)\in B_2\times(0,\nu)$,
 \item
 $V(x,t)=f(x,t)=0$ \quad for $(x,t)\in B_2\times(0,\nu)$.
 \end{itemize}
 We provide a proof only in this case.
 The general case can be handled by a standard density argument and is omitted.
 Let $\chi_j$ be a cut off function satisfying $\chi_j=1$ for $t>\frac{1}{j}$ and $\chi_j=0$ for $t<\frac{1}{2j}$.
 Consider
 \begin{align}
 \label{EQ_3.4}
 &
 \begin{cases}
 w_t
 =
 \Delta_xw+{\bf b}(x,t)\cdot\nabla_xw+V(x,t)w
 +
 f(x,t)
 &
 \text{for } (x,t)\in B_2\times(0,\delta),
 \\
 w(x,t)=u(x,t)\chi_j
 &
 \text{for } (x,t)\in\pa B_2\times(0,\delta),
 \\
 w(x,0)=0
 &
 \text{for }
 x\in B_2.
 \end{cases}
 \end{align}
 This equation admits a unique classical solutions
 $w_j(x,t)\in C^\infty(B_2\times(0,\delta))\cap C(\overline{B_2}\times[0,\delta])$
 satisfying
 \[
 w_j(x,t)=0
 \quad
 \text{for }
 x\in B_2\times(0,\tfrac{1}{2j}).
 \]
 For a given function $g(x,t)$,
 we denote by $g_\text{ex}(x,t)$ its extension by zero to $t<0$.
 It is clear that
 $w_{j,\text{ex}}(x,t)\in C^\infty(B_1\times(-\infty,\delta))\cap C(\overline{B_1}\times(-\infty,\delta])$
 and
 it satisfies \eqref{EQ_3.4} with
 ${\bf b}(x,t)$, $V(x,t)$ and $f(x,t)$ replaced by
 ${\bf b}_\text{ex}(x,t)$, $V_\text{ex}(x,t)$ and $f_\text{ex}(x,t)$,
 respectively,
 for $(x,t)\in B_2\times(-\infty,\delta)$.
 \begin{itemize}
 \item
 In what follows,
 the constant $C$ is assumed to be independent of $j$ and $\delta$
 (and might depend on $n$, $p$, $q$ and $M$).
 \end{itemize}
 From Lemma \ref{lemma_3.6},
 we have
 \begin{align}
 \label{EQ_3.5}
 &\|\nabla_xw_{j,\text{ex}}\|_{L^\infty(B_1\times(\delta-1,\delta))}
 \\
 \nonumber
 &<
 c\|w_{j,\text{ex}}\|_{L^\infty(B_2\times(\delta-1,\delta))}
 +
 c\|f_\text{ex}\|_{L^{p,q}(B_2\times(\delta-1,\delta))}.
 \end{align}
 Furthermore,
 from Lemma \ref{lemma_3.5} (iii),
 we deduce that
 \begin{align}
 \nonumber
 &
 \|w_{j,\text{ex}}\|_{L^\infty(B_2\times(\delta-1,\delta))}
 \\
 \nonumber
 &<
 \|w_{j,\text{ex}}\|_{L^\infty(\pa B_2\times(\delta-1,\delta))}
 +
 C\|f_\text{ex}\|_{L^{2,2}(B_2\times(\delta-1,\delta))}
 \\
 \nonumber
 &\quad
 +
 C\|f_\text{ex}\|_{L^{p,q}(B_2\times(\delta-1,\delta))}
 \\
 \label{EQ_3.6}
 &=
 \|u(x,t)\chi_j\|_{L^\infty(\pa B_2\times(\delta-1,\delta))}
 +
 C\|f_\text{ex}\|_{L^{2,2}(B_2\times(\delta-1,\delta))}
 \\
 \nonumber
 &\quad
 +
 C\|f_\text{ex}\|_{L^{p,q}(B_2\times(\delta-1,\delta))}.
 \end{align}
 From the assumption on $p$ and $q$,
 we note that $p>2$ and $q>2$.
 Hence,
 combining \eqref{EQ_3.5} - \eqref{EQ_3.6},
 we obtain
 \begin{align*}
 &\|\nabla_xw_{j,\text{ex}}\|_{L^\infty(B_1\times(\delta-1,\delta))}
 \\
 \nonumber
 &<
 c
 \|u(x,t)\chi_j\|_{L^\infty(\pa B_2\times(\delta-1,\delta))}
 +
 c\|f_\text{ex}\|_{L^{p,q}(B_2\times(\delta-1,\delta))}.
 \end{align*}
 This inequality implies
 \begin{align}
 \label{EQ_3.7}
 &\|\nabla_xw_j\|_{L^\infty(B_1\times(0,\delta))}
 \\
 \nonumber
 &<
 c
 \|u\|_{L^\infty(\pa B_2\times(\delta-1,\delta))}
 +
 c\|f\|_{L^{p,q}(B_2\times(0,\delta))}.
 \end{align}
 Define
 \begin{align*}
 h_{kl}(x,t)
 =
 w_k(x,t)
 -
 w_l(x,t).
 \end{align*}
 Then,
 it satisfies
 \begin{align*}
 \begin{cases}
 \pa_th_{k,l}
 =
 \Delta_xh_{k,l}
 +
 {\bf b}(x,t)\cdot\nabla_xh_{k,l}
 +
 V(x,t)h_{k,l}
 &
 \text{for } (x,t)\in B_2\times(0,\delta),
 \\
 w(x,t)=u(x,t)(\chi_k-\chi_l)
 &
 \text{for } (x,t)\in\pa B_2\times(0,\delta),
 \\
 w(x,0)=0
 &
 \text{for }
 x\in B_2.
 \end{cases}
 \end{align*}
 By the same procedure as above,
 we can verify that
 \begin{align*}
 \|h_{k,l}\|_{L^\infty(B_2\times(0,\delta))}
 &<
 c
 \|u(x,t)(\chi_k-\chi_l)\|_{L^\infty(\pa B_2\times(0,\delta))},
 \\
 \|\nabla_xh_{k,l}\|_{L^\infty(B_1\times(0,\delta))}
 &<
 c
 \|u(x,t)(\chi_k-\chi_l)\|_{L^\infty(\pa B_1\times(0,\delta))}.
 \end{align*}
 Since $u(x,0)=0$ for $x\in\overline{B_2}$,
 it follows that
 \begin{align*}
 \lim_{k,l\to\infty}h_{k,l}(x,t)
 &=
 0
 \quad \text{in } C(\overline{B_2}\times[0,\delta]),
 \\
 \lim_{k,l\to\infty}\nabla_xh_{k,l}(x,t)
 &=
 0
 \quad \text{in } C(\overline{B_1}\times[0,\delta]).
 \end{align*}
 Therefore,
 there exists a function
 $w_\infty(x,t)\in C(\overline{B_2}\times[0,\delta])\cap C^{1,0}(B_1\times(0,\delta))$
 with $\nabla_xw_\infty\in\{C(\overline{B_1}\times[0,\delta])\}^n$
 such that
 \begin{align*}
 \lim_{j\to\infty}
 w_j(x,t)
 &=
 w_\infty(x,t)
 \quad \text{in } C(\overline{B_2}\times[0,\delta]),
 \\
 \lim_{j\to\infty}
 \nabla_xw_j(x,t)
 &=
 \nabla_xw_\infty(x,t)
 \quad \text{in } C(\overline{B_1}\times[0,\delta]).
 \end{align*}
 This $w_\infty(x,t)$ satisfies
 \begin{align}
 \label{EQ_3.8}
 &
 \begin{cases}
 \pa_t
 w_\infty
 =
 \Delta_xw_\infty
 +
 {\bf b}(x,t)\cdot\nabla_xw_\infty
 +
 V(x,t)w_\infty
 +
 f(x,t)
 \\
 \quad
 \text{for } (x,t)\in B_2\times(0,\delta),
 \\
 w(x,t)=u(x,t)
 \quad
 \text{for } (x,t)\in\pa B_2\times(0,\delta),
 \\
 w(x,0)=0
 \quad
 \text{for }
 x\in B_2.
 \end{cases}
 \end{align}
 Since the solution to equation \eqref{EQ_3.8} is unique within an appropriate function class,
 we conclude that $w_\infty(x,t)=u(x,t)$.
 This completes the proof.
 \end{proof}

\subsection{Gradient Decay Estimates for Solutions to Parabolic Equations}
\label{section_3.5}
 Let $1<t_1<t_2$ be real numbers,
 and let
 $\rho(t)$ be a continuous function on $[t_1,t_2]$
 satisfying the following conditions:
 \begin{itemize}
  \item
  There exists ${\sf r}>0$ such that
  $4<\rho(t)<{\sf r}\sqrt t$
  for all $t \in (t_1,t_2)${\rm;}

  \item
  there exists a constant ${\sf a}\in(0,\frac{1}{2})$ such that
  \begin{align*}
  \inf_{t'\in(\max\{(1-{\sf a})t,t_1\},t)}
  \rho(t')
  >
  \tfrac{3}{4}
  \rho(t)
  \quad
  \text{for all }
  t\in(t_1,t_2).
  \end{align*}
 \end{itemize}
 We define the time-dependent spatial domain
 \begin{align*}
 \Omega(t)
 =
 \{x\in\R^n;1<|x|<\rho(t)\},
 \end{align*}
 and set
 \begin{align*}
 \Omega_{(t_1,t_2)}
 =
 \bigcup_{t\in(t_1,t_2)}
 \Omega(t).
 \end{align*}
 Consider
 \begin{align}
 \label{EQ_3.9}
 \begin{cases}
 u_t
 =
 \Delta_xu
 +
 V(x,t)
 u
 +
 f(x,t)
 &
 \text{for }
 (x,t)\in\Omega_{(t_1,t_2)},
 \\
 u|_{t=t_1}
 =
 0
 &
 \text{for }
 x\in \Omega(t_1).
 \end{cases}
 \end{align}
 No boundary condition is imposed in this subsection.
 The following lemma is based on an argument that appears
 in the proof of Proposition 7.2 in \cite{Cortazar-delPino-Musso} (p. 340).
 We state it separately and include a proof here for clarity.
 \begin{lem}
 \label{lemma_3.8}
 Assume that there exists a constant $C_V>0$ such that
 \begin{itemize}
 \item
 $|V(x,t)|<\frac{C_V}{|x|^2}$ \quad for all $(x,t)\in\Omega_{(t_1,t_2)}$.
 \end{itemize}
 Let $u(x,t)\in C^{2,1}(\Omega_{(t_1,t_2)})\cap C(\overline{\Omega_{(t_1,t_2)}})$
 be a solution of \eqref{EQ_3.9}.
 Then,
 for any given ${\sf a}\in(0,\frac{1}{2})$,
 there exists a constant $C>0$,
 independent of $t_1$, $t_2$ and $\rho(t)$,
 such that
 if ${\sf r}^2<\frac{{\sf a}}{64}$,
 the following estimates hold{\rm:}
 \begin{enumerate}[\rm(i)]
 \item 
 If $(x_0,t_0)\in\Omega_{(t_1,t_2)}$ satisfies $1<|x_0|<\frac{\rho(t_0)}{2}$ and $t_0>t_1+\frac{|x_0|^2}{64}$,
 we have
 \begin{align*}
 \nonumber
 |\nabla_xu(x_0,t_0)|
 &<
 \tfrac{C}{|x_0|}
 \sup_{t\in(t_0-\frac{|x_0|^2}{64},t_0)}
 \sup_{x\in B(x_0,\frac{|x_0|}{4})}
 |u(x,t)|
 \\
 &\quad
 +
 C|x_0|
 \sup_{t\in(t_0-\frac{|x_0|^2}{64},t_0)}
 \sup_{x\in B(x_0,\frac{|x_0|}{4})}
 |f(x,t)|.
 \end{align*}
 
 \item
 If $(x_0,t_0)\in\Omega_{(t_1,t_2)}$ satisfies $1<|x_0|<\frac{\rho(t_0)}{2}$ and $t_0<t_1+\frac{|x_0|^2}{64}$,
 we have
  \begin{align*}
 \nonumber
 |\nabla_xu(x_0,t_0)|
 &<
 \tfrac{C}{|x_0|}
 \sup_{t\in(t_1,t_0)}
 \sup_{x\in B(x_0,\frac{|x_0|}{4})}
 |u(x,t)|
 \\
 &\quad
 +
 C|x_0|
 \sup_{t\in(t_1,t_0)}
 \sup_{x\in B(x_0,\frac{|x_0|}{4})}
 |f(x,t)|.
  \end{align*}
 \end{enumerate}
 \end{lem}
 \begin{proof}
 Let $(x_0,t_0)$ be a point in $\Omega_{(t_1,t_2)}$ such that $|x_0|<\frac{\rho(t_0)}{2}$.
 Set $d_0:=|x_0|$,
 and introduce the rescaled function
 \begin{align*}
 \tilde u(X,T)
 =
 u(x_0+\tfrac{d_0}{8}X,t_0+\tfrac{d_0^2}{64}T).
 \end{align*}
 Similarly,
 define:
 \begin{itemize}
 \item $\tilde V=(\frac{d_0}{8})^2V(x_0+\frac{d_0}{8}X,t_0+\frac{d_0^2}{64}T)$,
 \item $\tilde f=(\frac{d_0}{8})^2f(x_0+\frac{d_0}{8}X,t_0+\frac{d_0^2}{64}T)$.
 \end{itemize}
 We define the rescaled time interval \((T_1, T_2)\) as
 \[
 T_1=-\tfrac{64(t_0-t_1)}{d_0^2},
 \quad
 T_2=\tfrac{64(t_2-t_0)}{d_0^2},
 \]
 and introduce the space-time domain
 \begin{align*}
 \tilde\Omega_{(T_1,T_2)}
 &=
 \bigcup_{T\in(T_1,T_2)}
 \tilde\Omega(T),
 \end{align*}
 where $\tilde\Omega(T)=\{X\in\R^n;|x_0+\frac{d_0}{8}X|<\rho(t_0+\frac{d_0^2}{64}T)\}$.
 Then $\tilde u(X,T)$ satisfies
 \begin{align}
 \label{EQ_3.10}
 \pa_T\tilde u
 &=
 \Delta_X\tilde u
 +
 \tilde V
 \tilde u
 +
 \tilde f\quad
 \text{for }
 (X,T)\in\tilde\Omega_{(T_1,T_2)}.
 \end{align}
 From the assumption on $V(x,t)$,
 there exists $C>0$ independent of $(x_0,t_0)$ and $\rho(t)$
 such that
 \begin{align}
 \nonumber
 |\tilde V(X,T)|
 &=
 |
 d_0^2
 V(x_0+\tfrac{d_0}{8}X,t_0+\tfrac{d_0^2}{64}T)
 |
 \\
 \label{EQ_3.11}
 &<
 \tfrac{C_Vd_0^2}{|x_0+\frac{d_0}{8}X|^2}
 <
 C
 \\
 \nonumber
 & \text{for }
 (X,T)\in\tilde\Omega_{(T_1,T_2)}
 \text{ with }
 |X|<2.
 \end{align}
 We now claim:
 \begin{enumerate}[(ci)]
 \item 
 $\tilde u(X,T)$ is well-defined on the parabolic cylinder $B_2\times(-1,0)$
 if $T_1<-1$,
 which is equivalent to $t_0>t_1+\frac{|x_0|^2}{64}$,
 \item
 $\tilde u(X,T)$ is well-defined on the parabolic cylinder $B_2\times(T_1,0)$
 if $T_1\in(-1,0)$,
 which is equivalent to $t_0<t_1+\frac{|x_0|^2}{64}$.
 \end{enumerate}
 Assume first that $T_1<-1$.
 From the relation $d_0=|x_0|<\frac{\rho(t_0)}{2}$ and the first assumption on $\rho(t)$,
 we observe that
 if ${\sf r}^2<256{\sf a}$,
 then
 \begin{align*}
 t_0+\tfrac{d_0^2}{64}T
 &>
 t_0-\tfrac{d_0^2}{64}
 >
 t_0-\tfrac{\rho(t_0)^2}{256}
 >
 t_0-\tfrac{{\sf r}^2t_0}{256}
 \\
 &
 >
 (1-{\sf a})t_0
 \quad
 \text{for all } T\in(-1,0). 
 \end{align*}
 Using the second assumption on $\rho(t)$,
 we obtain
 \begin{align*}
 \rho(t_0+\tfrac{d_0^2}{64}T)
 >
 \tfrac{3}{4}
 \rho(t_0)
 \quad
 \text{for all }
 T\in(-1,0).
 \end{align*}
 Moreover,
 for $|X|<2$,
 we have
 \begin{align*}
 |x_0+\tfrac{d_0}{8}X|
 <
 \tfrac{5}{4}
 |x_0|
 <
 \tfrac{5}{8}
 \rho(t_0).
 \end{align*}
 Therefore
 it follows that
 \begin{align*}
 \rho(t_0+\tfrac{d_0^2}{16}T)>|x_0+\tfrac{d_0}{4}X|
 \quad \text{for all }
 (X,T)\in B_2\times(-1,0),
 \end{align*}
 which shows that $\tilde u(X,T)$ is indeed well-defined on $B_2 \times (-1,0)$.
 Thus the claim (ci) is proved.
 The claim (cii) can be verified in the same manner.
 In case (ci),
 it follows from Lemma \ref{lemma_3.6} that
 there exists a constant $C>0$,
 independent of $(x_0,t_0)$ and $\rho(t)$,
 such that
 \begin{align}
 \label{EQ_3.12}
 &
 \|\nabla_X\tilde u\|_{L_{X,T}^\infty(B_1\times(-\frac{1}{2},0))}
 \\
 \nonumber
 &<
 C
 \|\tilde u\|_{L_{X,T}^\infty(B_2\times(-1,0))}
 +
 C
 \|\tilde f(X,T)\|_{L_{X,T}^{p,q}(B_2\times(-1,0))}.
 \end{align}
 The last term can be written in the original variables $(x,t)$ as
 \begin{align}
 \nonumber
 &
 \|\tilde f(X,T)\|_{L_{X,T}^{p,q}(B_2\times(-1,0))}.
 \\
 \nonumber
 &<
 C
 \sup_{T\in(-1,0)}
 \sup_{|X|<2}
 |\tilde f(X,T)|
 \\
 \nonumber
 &=
 C
 \sup_{T\in(-1,0)}
 \sup_{|X|<2}
 (\tfrac{d_0}{8})^2
 |f(x_0+\tfrac{d_0}{8}X,t_0+\tfrac{d_0^2}{64}T)|
 \\
 \label{EQ_3.13}
 &=
 C
 \sup_{t\in(t_0-\frac{d_0^2}{64},t_0)}
 \sup_{x\in B(x_0,\frac{d_0}{4})}
 d_0^2
 |f(x,t)|.
 \end{align}
 Since $\nabla_X\tilde u=\frac{d_0}{8}\nabla_xu(x_0+\frac{d_0}{8}X,t_0+\frac{d_0^2}{64}T)$,
 from \eqref{EQ_3.12} - \eqref{EQ_3.13},
 we obtain
 \begin{align}
 \nonumber
 &
 \|\nabla_xu\|_{L_{x,t}^\infty(B(x_0,\frac{d_0}{8})\times(t_0-\frac{d_0}{128},t_0))}
 \\
 \label{EQ_3.14}
 &<
 \tfrac{C}{d_0}
 \sup_{t\in(t_0-\frac{d_0^2}{64},t_0)}
 \sup_{x\in B(x_0,\frac{d_0}{4})}
 |u(x,t)|
 \\
 \nonumber
 &\quad
 +
 Cd_0
 \sup_{t\in(t_0-\frac{d_0^2}{64},t_0)}
 \sup_{x\in B(x_0,\frac{d_0}{4})}
 |f(x,t)|.
 \end{align}
 Therefore,
 the statement of (i) in this lemma has been established.
 Let us now consider the alternative case (cii).
 The definitions of $\tilde u(X,T)$, $\tilde V(X,T)$, and $\tilde f(X,T)$ remain the same as in case (ci).
 The only differences between (ci) and (cii) lie in the time interval on which equation \eqref{EQ_3.10} is defined.
 In case (cii),
 we apply Lemma \ref{lemma_3.7} instead of Lemma \ref{lemma_3.6},
 which enables us to establish the estimate stated in (ii).
 \end{proof}

\subsection{Solvability of Inhomogeneous Problem $H_yT+g=0$}
\label{section_3.6}
 We provide solutions of the following linearized problem around the ground state
 under the assumption of radial symmetry.
 \begin{equation}
 \label{EQ_3.15}
 H_yT+g=0,
 \end{equation}
 where $H_y=\Delta_y+V(y)$ and $V(y)=p{\sf Q}(y)^{p-1}$.
 We recall that
 \begin{align*}
 \Lambda_y{\sf Q}
 &=
 \left( \tfrac{n-2}{2}+y\cdot\nabla_y \right){\sf Q}
 \\
 &=
 \tfrac{n-2}{2}( 1-\tfrac{|y|^2}{n(n-2)} )
 ( 1+\tfrac{|y|^2}{n(n-2)} )^{-\frac{n}{2}}
 \end{align*}
 is a radially symmetric solution to $H_yT=0$.
 Let $\Gamma$ be another linearly independent, radially symmetric solution to $H_yT=0$.
 It is known that (see Section 3.7 in \cite{Harada_6D})
 \begin{align*}
 \Gamma(r)
 &=
 -
 \tfrac{2n}{n-2}
 \{n(n-2)\}^{-\frac{n}{2}}
 +
 O(r^{-2}) \qquad (n\geq5),
 \\
 \label{3.10}
 \pa_r\Gamma(r)
 &=
 O(r^{-3}) \qquad (n\geq5).
 \end{align*}
 A solution of \eqref{EQ_3.15} can be written as
 \begin{equation}
 \label{EQ_3.16}
 T
 =
 -\Gamma
 \int_0^r
 (\Lambda_y{\sf Q})g
 r^{n-1}dr
 +
 (\Lambda_y{\sf Q})
 \int_0^r
 \Gamma gr^{n-1}dr.
 \end{equation}

 \subsection{Fundamental Estimates for Solutions to the Linear Heat Equation with Inhomogeneous Terms}
 \label{section_3.7}
 Consider
 \begin{align}
 \label{EQ_3.17}
 \begin{cases}
 u_t=
 \Delta_xu
 +
 f(x,t)
 &
 \text{for }
 (x,t)\in\R^n\times(t_0,\infty),
 \\
 u|_{t=t_0}
 =
 0
 &
 \text{for }
 x\in\R^n.
 \end{cases}
 \end{align}
 The solution $u(x,t)$ to \eqref{EQ_3.17} is represented by
 \begin{align*}
 u(x,t)
 &=
 \int_{t_0}^t
 ds
 \int_{\R_y^n}
 (4\pi(t-s))^{-\frac{n}{2}}
 e^{-\frac{|x-y|^2}{4(t-s)}}
 f(y,s)
 dy.
 \end{align*}
 The following lemmas are variants of
 results contained in Lemmas A.1 - A.2 of the paper
 by Juncheng Wei - Qidi Zhang - Yifu Zhou \cite{Wei-Qidi-Yifu}
 on p.p 119 - 124 with logarithmic corrections.
 We provide a proof for completeness.
 \begin{lem}
 \label{lemma_3.9}
 Let $\gamma\in(0,\frac{n}{2})$ and $q\in\R$.
 Assume that
 \begin{align*}
 |f(x,t)|
 <
 \begin{cases}
 t^{-\gamma}
 (\log t)^q
 &
 \text{for } |x|<K_1\sqrt t,\ t\in(t_0,\infty),
 \\
 0
 &
 \text{for } |x|>K_1\sqrt t,\ t\in(t_0,\infty)
 \end{cases}
 \end{align*}
 for some $K_1>0$ and $t_0>e$.
 Then
 for any $K_2>0$,
 there exists a constant $C>0$,
 depending only on $\gamma$, $q$, $K_1$ and $K_2$,
 but independent of $t_0$,
 such that
 \begin{align*}
 |u(x,t)|
 <
 C
 \begin{cases}
 t^{-\gamma+1}
 (\log t)^q
 & \text{\rm if } |x|<K_2\sqrt t,\ t\in(t_0,\infty),
 \\
 |x|^{-2\gamma+2}
 (\log|x|^2)^q
 & \text{\rm if } |x|>K_2\sqrt t,\ t\in(t_0,\infty).
 \end{cases}
 \end{align*}
 \end{lem}
 \begin{proof}
 We follow the argument used in the proof of Lemma A.1 in \cite{Wei-Qidi-Yifu}.
 For any $x\in\R^n$,
 we have
 \begin{align}
 \nonumber
 &
 \int_{t_0}^\frac{t}{2}
 (t-s)^{-\frac{n}{2}}
 ds
 \int_{\R^n}
 e^{-\frac{|x-y|^2}{4(t-s)}}
 s^{-\gamma}
 (\log s)^q
 {\bf 1}_{|y|<K_1\sqrt s}
 dy
 \\
 \nonumber
 &<
 C
 t^{-\frac{n}{2}}
 \int_{t_0}^\frac{t}{2}
 s^{-\gamma}
 (\log s)^q
 ds
 \int_{\R^n}
 {\bf 1}_{|y|<K_1\sqrt s}
 dy
 \\
 \label{EQ_3.18}
 &<
 C
 t^{-\frac{n}{2}}
 \int_{t_0}^\frac{t}{2}
 s^{-\gamma+\frac{n}{2}}
 (\log s)^q
 ds
 <
 C
 t^{-\gamma+1}
 (\log t)^q
 \\
 \nonumber
 &
 \text{for }
 x\in\R^n
 \end{align}
 and
 \begin{align}
 \nonumber
 &
 \int_\frac{t}{2}^t
 (t-s)^{-\frac{n}{2}}
 ds
 \int_{\R^n}
 e^{-\frac{|x-y|^2}{4(t-s)}}
 s^{-\gamma}
 (\log s)^q
 {\bf 1}_{|y|<K_1\sqrt s}
 dy
 \\
 \nonumber
 &<
 C
 t^{-\gamma}
 (\log t)^q
 \int_\frac{t}{2}^t
 (t-s)^{-\frac{n}{2}}
 ds
 \int_{\R^n}
 e^{-\frac{|x-y|^2}{4(t-s)}}
 dy
 \\
 \label{EQ_3.19}
 &<
 C
 t^{-\gamma+1}
 (\log t)^q
 \quad
 \text{for }
 x\in\R^n.
 \end{align}
 We next consider the case $|x|>2K_1\sqrt t$.
 In this case,
 it is clear that $\frac{|x|}{2}>|y|$ for $|y|<K_1\sqrt s$ and $s\in(t_0,t)$.
 Hence
 \begin{align}
 \nonumber
 &
 \int_{t_0}^\frac{t}{2}
 (t-s)^{-\frac{n}{2}}
 ds
 \int_{\R^n}
 e^{-\frac{|x-y|^2}{4(t-s)}}
 s^{-\gamma}
 (\log s)^q
 {\bf 1}_{|y|<K_1\sqrt s}
 dy
 \\
 \nonumber
 &<
 C
 t^{-\frac{n}{2}}
 \int_{t_0}^\frac{t}{2}
 s^{-\gamma}
 (\log s)^q
 ds
 \int_{\R^n}
 e^{-\frac{|x-y|^2}{4t}}
 {\bf 1}_{|y|<K_1\sqrt s}
 dy
 \\
 \nonumber
 &<
 C
 t^{-\frac{n}{2}}
 \int_{t_0}^\frac{t}{2}
 s^{-\gamma}
 (\log s)^q
 ds
 \int_{\R^n}
 e^{-\frac{|x|^2}{8t}}
 {\bf 1}_{|y|<K_1\sqrt s}
 dy
 \\
 &<
 \nonumber
 C
 t^{-\frac{n}{2}}
 e^{-\frac{|x|^2}{8t}}
 \int_{t_0}^\frac{t}{2}
 s^{-\gamma+\frac{n}{2}}
 (\log s)^q
 ds
 \\
 \nonumber
 &<
 C
 t^{-\gamma}
 e^{-\frac{|x|^2}{8t}}
 \cdot
 (\log t)^q
 \\
 \nonumber
 &<
 C
 |x|^{-2\gamma}
 \cdot
 (\log\tfrac{|x|^2}{4K_1^2})^q
 \\
 \label{EQ_3.20}
 &<
 C
 |x|^{-2\gamma}
 (\log|x|^2)^q
 \quad
 \text{for }
 |x|>2K_1\sqrt t
 \end{align}
 and
 \begin{align}
 \nonumber
 &
 \int_{\frac{t}{2}}^t
 (t-s)^{-\frac{n}{2}}
 ds
 \int_{\R^n}
 e^{-\frac{|x-y|^2}{4(t-s)}}
 s^{-\gamma}
 (\log s)^q
 {\bf 1}_{|y|<K_1\sqrt s}
 dy
 \\
 \nonumber
 &<
 C
 t^{-\gamma}
 (\log t)^q
 \int_{\frac{t}{2}}^t
 (t-s)^{-\frac{n}{2}}
 ds
 \int_{\R^n}
 e^{-\frac{|x|^2}{8(t-s)}}
 {\bf 1}_{|y|<K_1\sqrt s}
 dy
 \\
 \nonumber
 &<
 C
 t^{-\gamma}
 (\log t)^q
 \int_{\frac{t}{2}}^t
 (t-s)^{-\frac{n}{2}}
 e^{-\frac{|x|^2}{8(t-s)}}
 s^\frac{n}{2}
 ds
 \\
 \nonumber
 &<
 C
 t^{-\gamma}
 (\log t)^q
 \int_{\frac{t}{2}}^t
 |x|^{-n}
 t^\frac{n}{2} 
 ds
 \\
 \nonumber
 &<
 C
 t^{-\gamma+\frac{n}{2}+1}
 (\log t)^q
 |x|^{-n}
 \\
 \label{EQ_3.21}
 &<
 C
 |x|^{-2\gamma+2}
 (\log |x|^2)^q
 \quad
 \text{for }
 |x|>2K_1\sqrt t.
 \end{align}
 Therefore,
 inequalities \eqref{EQ_3.18} - \eqref{EQ_3.21} imply
 \begin{align*}
 |u(x,t)|
 &<
 \begin{cases}
 Ct^{-\gamma+1}(\log t)^q
 &
 \text{for }
 |x|<2K_1\sqrt t,
 \\
 C|x|^{-2\gamma+2}(\log|x|^2)^q
 &
 \text{for }
 |x|>2K_1\sqrt t.
 \end{cases}
 \end{align*}
 The proof is completed.
 \end{proof}
 \begin{lem}
 \label{lemma_3.10}
 Let $\gamma\in(0,\frac{n}{2})$ and $q\in\R$.
 Assume that
 \begin{align*}
 |f(x,t)|
 <
 \begin{cases}
 0
 &
 \text{for } |x|<K_1\sqrt t,\ t\in(t_0,\infty),
 \\
 |x|^{-2\gamma}
 (\log|x|^2)^q
 &
 \text{for } |x|>K_1\sqrt t,\ t\in(t_0,\infty)
 \end{cases}
 \end{align*}
 for some $K_1>0$ and $t_0>e$.
 Then
 for any $K_2>0$,
 there exists a constant $C>0$,
 depending only on $\gamma$, $q$, $K_1$ and $K_2$
 {\rm(}but independent of $t_0${\rm)},
 such that
 \begin{align*}
 |u(x,t)|
 <
 C
 \begin{cases}
 t^{-\gamma+1}
 (\log t)^q
 & \text{\rm if } |x|<K_2\sqrt t,\ t\in(t_0,\infty),
 \\
 |x|^{-2\gamma+2}
 (\log|x|^2)^q
 & \text{\rm if } |x|>K_2\sqrt t,\ t\in(t_0,\infty).
 \end{cases}
 \end{align*}
 \end{lem}
 \begin{proof}
 We follow the argument used in the proof of Lemma A.2 in \cite{Wei-Qidi-Yifu}.
 For any $|x|<\frac{K_1}{4}\sqrt t$,
 we have
 \begin{align}
 \nonumber
 &
 \int_{t_0}^\frac{t}{2}
 (t-s)^{-\frac{n}{2}}
 ds
 \int_{\R^n}
 e^{-\frac{|x-y|^2}{4(t-s)}}
 |y|^{-2\gamma}
 (\log |y|^2)^q
 {\bf 1}_{|y|>K_1\sqrt s}
 dy
 \\
 \nonumber
 &<
 C
 t^{-\frac{n}{2}}
 \int_{t_0}^\frac{t}{2}
 ds
 \left(
 \int_{K_1\sqrt{s}<|y|<\frac{K_1}{2}\sqrt t}
 +
 \int_{|y|>\frac{K_1}{2}\sqrt t}
 \right)
 \\
 \nonumber
 &\qquad
 \times
 e^{-\frac{|x-y|^2}{4(t-s)}}
 |y|^{-2\gamma}
 (\log |y|^2)^q
 dy
 \\
 \nonumber
 &<
 C
 t^{-\frac{n}{2}}
 \int_{t_0}^\frac{t}{2}
 ds
 \underbrace{
 \int_{K_1\sqrt s<|y|<\frac{K_1}{2}\sqrt t}
 |y|^{-2\gamma}
 (\log |y|^2)^q
 dy
 }_{<C(\frac{K_1}{2}\sqrt t)^{n-2\gamma}(\log \frac{K_1^2}{4}t)^q}
 \\
 \nonumber
 &\qquad
 +
 C
 t^{-\frac{n}{2}-2\gamma}
 (\log t)^q
 \int_{t_0}^t
 ds
 \int_{|y|>\frac{K_1}{2}\sqrt t}
 e^{-\frac{|x-y|^2}{4t}}
 dy 
 \\
 \label{EQ_3.22}
 &<
 C
 t^{-\gamma+1}
 (\log\tfrac{K_1^2}{4}t)^q
 +
 C
 t^{-\gamma+1}
 (\log t)^q
 \quad
 \text{for }
 |x|<\tfrac{K_1}{4}\sqrt t
 \end{align}
 and
 \begin{align}
 \nonumber
 &
 \int_\frac{t}{2}^t
 (t-s)^{-\frac{n}{2}}
 ds
 \int_{\R^n}
 e^{-\frac{|x-y|^2}{4(t-s)}}
 |y|^{-2\gamma}
 (\log |y|^2)^q
 {\bf 1}_{|y|>K_1\sqrt s}
 dy
 \\
 \nonumber
 &<
 \int_\frac{t}{2}^t
 (t-s)^{-\frac{n}{2}}
 ds
 \int_{\R^n}
 e^{-\frac{|x-y|^2}{4(t-s)}}
 |y|^{-2\gamma}
 (\log |y|^2)^q
 {\bf 1}_{|y|>\frac{K_1}{2}\sqrt t}
 \\
 \nonumber
 &<
 \int_\frac{t}{2}^t
 (t-s)^{-\frac{n}{2}}
 ds
 \int_{\R^n}
 e^{-\frac{|y|^2}{16(t-s)}}
 |y|^{-2\gamma}
 (\log |y|^2)^q
 {\bf 1}_{|y|>\frac{K_1}{2}\sqrt t}
 \\
 \nonumber
 &<
 C
 t^{-\gamma}
 (\log\tfrac{K_1^2}{4}t)^q
 \int_\frac{t}{2}^t
 (t-s)^{-\frac{n}{2}}
 ds
 \underbrace{
 \int_{\R^n}
 e^{-\frac{|y|^2}{16(t-s)}}
 {\bf 1}_{|y|>\frac{K_1}{2}\sqrt t}
 }_{<C(t-s)^\frac{n}{2}}
 \\
 \nonumber
 &<
 C
 t^{-\gamma}
 (\log\tfrac{K_1^2}{4}t)^q
 \int_\frac{t}{2}^t
 ds
 \\
 \label{EQ_3.23}
 &<
 C
 t^{-\gamma+1}
 (\log\tfrac{K_1^2}{4}t)^q
 \quad
 \text{for }
 |x|<\tfrac{K_1}{4}\sqrt t.
 \end{align}
 Furthermore,
 for any $|x|>\frac{K_1}{4}\sqrt t$,
 we see that
 \begin{align}
 \nonumber
 &
 \int_{t_0}^\frac{t}{2}
 (t-s)^{-\frac{n}{2}}
 ds
 \int_{\R^n}
 e^{-\frac{|x-y|^2}{4(t-s)}}
 |y|^{-2\gamma}
 (\log |y|^2)^q
 {\bf 1}_{|y|>K_1\sqrt s}
 dy
 \\
 \nonumber
 &<
 C
 t^{-\frac{n}{2}}
 \int_{t_0}^\frac{t}{2}
 ds
 \left(
 \int_{K_1\sqrt s<|y|<\frac{|x|}{2}}
 +
 \int_{|y|>\frac{|x|}{2}}
 \right)
 \\
 \nonumber
 &\qquad
 \times
 e^{-\frac{|x-y|^2}{4(t-s)}}
 |y|^{-2\gamma}
 (\log |y|^2)^q
 {\bf 1}_{|y|>K_1\sqrt s}
 dy
 \\
 \nonumber
 &<
 C
 t^{-\frac{n}{2}}
 e^{-\frac{|x|^2}{16t}}
 \int_{t_0}^\frac{t}{2}
 ds
 \int_{K_1\sqrt s<|y|<\frac{|x|}{2}}
 |y|^{-2\gamma}
 (\log |y|^2)^q
 dy
 \\
 \nonumber
 &\quad
 +
 C
 t^{-\frac{n}{2}}
 |x|^{-2\gamma}
 (\log |x|^2)^q
 \int_{t_0}^\frac{t}{2}
 ds
 \int_{|y|>\frac{|x|}{2}}
 e^{-\frac{|x-y|^2}{4t}}
 dy
 \\
 \nonumber
 &<
 C
 t^{-\frac{n}{2}+1}
 e^{-\frac{|x|^2}{16t}}
 |x|^{-2\gamma+n}
 (\log |x|^2)^q
 +
 C
 t
 |x|^{-2\gamma}
 (\log |x|^2)^q
 \\
 \label{EQ_3.24}
 &<
 C
 |x|^{-2\gamma+2}
 (\log |x|^2)^q
 \quad
 \text{for }
 |x|>\tfrac{K_1}{4}\sqrt t
 \end{align}
 and
 \begin{align}
 \nonumber
 &
 \int_{\frac{t}{2}}^t
 (t-s)^{-\frac{n}{2}}
 ds
 \int_{\R^n}
 e^{-\frac{|x-y|^2}{4(t-s)}}
 |y|^{-2\gamma}
 (\log |y|^2)^q
 {\bf 1}_{|y|>K_1\sqrt s}
 dy
 \\
 \nonumber
 &<
 \int_{\frac{t}{2}}^t
 (t-s)^{-\frac{n}{2}}
 ds
 \left(
 \int_{K_1\sqrt s<|y|<\frac{|x|}{2}}
 +
 \int_{|y|>\frac{|x|}{2}}
 \right)
 \\
 \nonumber
 &\qquad
 \times
 e^{-\frac{|x-y|^2}{4(t-s)}}
 |y|^{-2\gamma}
 (\log |y|^2)^q
 {\bf 1}_{|y|>K_1\sqrt s}
 dy
 \\
 \nonumber
 &<
 C
 \int_{\frac{t}{2}}^t
 \underbrace{
 (t-s)^{-\frac{n}{2}}
 e^{-\frac{|x|^2}{16(t-s)}}
 }_{<C|x|^{-n}}
 ds
 \underbrace{
 \int_{K_1\sqrt s<|y|<\frac{|x|}{2}}
 |y|^{-2\gamma}
 (\log |y|^2)^q
 dy
 }_{<C|x|^{n-2\gamma}(\log |x|^2)^q}
 \\
 \nonumber
 &\quad
 +
 C
 |x|^{-2\gamma}
 (\log |x|^2)^q
 \int_{\frac{t}{2}}^t
 (t-s)^{-\frac{n}{2}}
 ds
 \underbrace{
 \int_{|y|>\frac{|x|}{2}}
 e^{-\frac{|x-y|^2}{4(t-s)}}
 dy
 }_{<C(t-s)^\frac{n}{2}}
 \\
 \nonumber
 &<
 C
 |x|^{-2\gamma}
 (\log |x|^2)^q
 \int_{\frac{t}{2}}^t
 ds
 +
 C
 |x|^{-2\gamma}
 (\log |x|^2)^q
 \int_{\frac{t}{2}}^t
 ds
 \\
 \nonumber
 &<
 C
 |x|^{-2\gamma}
 (\log |x|^2)^q
 t
 \\
 \label{EQ_3.25}
 &<
 C
 |x|^{-2\gamma+2}
 (\log |x|^2)^q
 \quad
 \text{for }
 |x|>\tfrac{K_1}{4}\sqrt t.
 \end{align}
 Therefore
 from \eqref{EQ_3.22} - \eqref{EQ_3.25},
 we conclude
 \begin{align*}
 |u(x,t)|
 &<
 \begin{cases}
 Ct^{-\gamma+1}(\log t)^q
 &
 \text{for }
 |x|<\frac{K_1}{4}\sqrt t,
 \\
 C|x|^{-2\gamma+2}(\log|x|^2)^q
 &
 \text{for }
 |x|>\frac{K_1}{4}\sqrt t.
 \end{cases}
 \end{align*}
 The proof is completed.
 \end{proof}

 \section{Oscillatory Behavior of Solutions to the Linear Heat Equation}
 \label{section_4}
 \subsection{Fundamental Estimates for Solutions to the Linear Heat Equation}
 \label{section_4.1}
 Consider the linear heat equation:
 \begin{align}
 \label{e_4.1}
 \begin{cases}
 \theta_t=\Delta_x\theta
 \quad\text{in } (x,t)\in\R^n\times(0,\infty),
 \\
 \theta(x,t)|_{t=0}
 =
 \theta_0(x).
 \end{cases}
 \end{align}
 \begin{lem}
 \label{lemma_4.1}
 Let $n\geq3$, $\gamma\in(0,n)$ and $\beta>0$.
 Assume that $\theta_0(x)$ is radially symmetric and satisfies
 \begin{align}
 \label{e4.2}
 |\theta_0(x)|
 <
 \begin{cases}
 \tfrac{|x|^{-\gamma}}{
 (\log|x|)^\beta}
 &
 \text{for } |x|>e,
 \\
 e^{-\gamma}
 &
 \text{for } |x|<e.
 \\
 \end{cases}
 \end{align}
 Then
 there exist two positive constants $t_1=t_1(n,\gamma,\beta)$ and $C=C(n,\gamma,\beta)$
 such that
 \begin{align}
 \label{e4.3}
 \left|
 \theta(x,t)
 -
 \theta(0,t)
 \right|
 &<
 \tfrac{C}{t^{\frac{\gamma}{2}}(\log t)^\beta}
 \cdot
 \tfrac{|x|^2}{t}
 \\
 \nonumber
 &
 \text{for }
 |x|<\sqrt t \text{ and } t>t_1,
 \\
 \label{e4.4}
 |\nabla_x\theta(x,t)|
 &<
 \tfrac{C}{t^{\frac{\gamma}{2}}(\log t)^\beta}
 \cdot
 \tfrac{|x|}{t}
 \\
 \nonumber
 &
 \text{for }
 |x|<\sqrt t \text{ and } t>t_1.
 \end{align}
 \end{lem}
 \begin{proof}
 We introduce a rescaled spatial variable $\xi$ defined by
 \begin{align*}
 \Theta(\xi,t)
 &=
 \theta(\xi\sqrt{t},t)
 \\
 &=
 \tfrac{1}{(4\pi t)^\frac{n}{2}}
 \int_{\R_y^n}
 e^{-\frac{1}{4}|\xi-\frac{y}{\sqrt t}|^2}
 \theta_0(y)
 dy.
 \end{align*}
 Its derivative is given by
 \begin{align}
 \nonumber
 \nabla_{\xi}
 \Theta(\xi,t)
 &=
 \tfrac{1}{(4\pi t)^\frac{n}{2}}
 \int_{\R_y^n}
 (-\tfrac{1}{2})
 (\xi-\tfrac{y}{\sqrt t})
 e^{-\frac{1}{4}|\xi-\frac{y}{\sqrt t}|^2}
 \theta_0(y)
 dy
 \\
 \label{e_4.5}
 &=
 \tfrac{-\xi}{2(4\pi t)^\frac{n}{2}}
 \int_{\R_y^n}
 e^{-\frac{1}{4}|\xi-\frac{y}{\sqrt t}|^2}
 \theta_0(y)
 dy
 \\
 \nonumber
 &\quad
 +
 \tfrac{1}{2(4\pi t)^\frac{n}{2}\sqrt t}
 \int_{\R_y^n}
 y
 e^{-\frac{1}{4}|\xi-\frac{y}{\sqrt t}|^2}
 \theta_0(y)
 dy.
 \end{align}
 Since $\theta_0(y)$ is radially symmetric, the vector integral
 $\int_{\R^n} y e^{-\frac{|y|^2}{4t}} \theta_0(y) \, dy$ vanishes by symmetry.
 Hence,
 \eqref{e_4.5} can be rewritten as
 \begin{align}
 \label{e_4.6}
 \nabla_\xi\Theta(\xi,t)
 &=
 \tfrac{-\xi}{2(4\pi t)^\frac{n}{2}}
 \int_{\R_y^n}
 e^{-\frac{1}{4}|\xi-\frac{y}{\sqrt t}|^2}
 \theta_0(y)
 dy
 \\
 \nonumber
 &\quad
 +
 \tfrac{1}{2(4\pi t)^\frac{n}{2}}
 \int_{\R_y^n}
 \tfrac{y}{\sqrt t}
 \left(
 e^{-\frac{1}{4}|\frac{y}{\sqrt t}-\xi|^2}
 -
 e^{-\frac{1}{4}|\frac{y}{\sqrt t}|^2}
 \right)
 \theta_0(y)
 dy.
 \end{align}
 It is easily seen that
 \begin{align*}
 e^{-\frac{1}{4}|\xi-\frac{y}{\sqrt t}|^2}
 &<
 C_1
 e^{-\frac{|y|^2}{8t}}
 \\
 &
 \text{for all }
 |\xi|<1,\ (y,t)\in\R^n\times(0,\infty)
 \end{align*}
 and
 \begin{align*}
 \tfrac{|y|}{\sqrt t}
 \bigl|
 e^{-\frac{1}{4}|\frac{y}{\sqrt t}-\xi|^2}
 -
 e^{-\frac{1}{4}|\frac{y}{\sqrt t}|^2}
 \bigr|
 &=
 \tfrac{|y|}{\sqrt t}
 \int_0^1
 e^{-\frac{1}{4}|\frac{y}{\sqrt t}-\theta\xi|^2}
 (-\tfrac{1}{2})
 (\tfrac{y}{\sqrt t}-\theta\xi)\cdot\xi
 d\theta
 \\
 &<
 \tfrac{C|y||\xi|}{\sqrt t}
 \int_0^1
 e^{-\frac{7}{32}|\frac{y}{\sqrt t}-\theta\xi|^2}
 d\theta
 \\
 &<
 \tfrac{C|y||\xi|}{\sqrt t}
 e^{-\frac{3}{16}|\frac{y}{\sqrt t}|^2}
 <
 C_2|\xi|
 e^{-\frac{1}{8}|\frac{y}{\sqrt t}|^2}
 \\
 &
 \text{for all }
 |\xi|<1,\ (y,t)\in\R^n\times(0,\infty).
 \end{align*}
 The constants $C_1$ and $C_2$ are independent of $\xi$, $y$ and $t$.
 Hence,
 we deduce from \eqref{e_4.6} that
 \begin{align}
 \nonumber
 &
 |\nabla_\xi\Theta(\xi,t)|
 <
 \tfrac{C|\xi|}{t^\frac{n}{2}}
 \int_{\R_y^n}
 e^{-\frac{|y|^2}{8t}}
 |\theta_0(y)|
 dy
 \\
 \nonumber
 &<
 \tfrac{C|\xi|}{t^\frac{n}{2}}
 \int_{|y|<e}
 e^{-\frac{|y|^2}{8t}}
 e^{-\gamma}
 +
 \tfrac{C|\xi|}{t^\frac{n}{2}}
 \int_{|y|>e}
 e^{-\frac{|y|^2}{8t}}
 \tfrac{|y|^{-\gamma}}{(\log |y|)^\beta}
 dy
 \\
 \label{e_4.7}
 &<
 \tfrac{C|\xi|}{t^\frac{n}{2}}
 +
 \tfrac{C|\xi|}{t^\frac{n}{2}}
 \int_{|y|>e}
 e^{-\frac{|y|^2}{8t}}
 \tfrac{|y|^{-\gamma}}{(\log |y|)^\beta}
 dy
 \\
 \nonumber
 &
 \text{for all }
 |\xi|<1.
 \end{align}
The last term of \eqref{e_4.7} can be estimated as follows.
 \begin{align}
 \nonumber
 &
 \int_{e<|y|<\frac{\sqrt t}{\log\sqrt t}}
 e^{-\frac{|y|^2}{8t}}
 \tfrac{|y|^{-\gamma}}{
 (\log|y|)^{\beta}}
 dy
 <
 \int_{e<|y|<\frac{\sqrt t}{\log\sqrt t}}
 \tfrac{|y|^{-\gamma}}{
 (\log|y|)^{\beta}}
 dy
 \\
 &<
 \tfrac{C_{n,\beta,\gamma}r_1^{n-\gamma}}{(\log r_1)^\beta}
 |_{r_1=\frac{\sqrt t}{\log \sqrt t}}
 \label{e_4.8}
 <
 \tfrac{C_{n,\beta,\gamma}t^\frac{n-\gamma}{2}}{(\log\sqrt t)^{\beta+n-\gamma}},
 \end{align}
 \begin{align}
 \nonumber
 &
 \int_{|y|>\sqrt t\log\sqrt t}
 e^{-\frac{|y|^2}{8t}}
 \tfrac{|y|^{-\gamma}}{
 (\log|y|)^{\beta}}
 dy
 \\
 \nonumber
 &=
 \tfrac{t^{\frac{n-\gamma}{2}}}{(\log \sqrt{t})^\beta}
 \int_{|z|>\log\sqrt{t}}
 e^{-\frac{|z|^2}{8}}
 |z|^{-\gamma}
 ( 1+\tfrac{\log|z|}{\log\sqrt t} )^{-\beta}
 dz
 \\
 \label{e_4.9}
 &<
 \tfrac{t^{\frac{n-\gamma}{2}}}{(\log \sqrt{t})^{\beta+1}}
 \int_{\R_z^n}
 e^{-\frac{|z|^2}{8}}
 |z|^{-\gamma+1}
 dz
 \end{align}
 and
 \begin{align*}
 &
 \int_{\frac{\sqrt t}{\log\sqrt t}<|y|<\sqrt t\log\sqrt t}
 e^{-\frac{|y|^2}{8t}}
 \tfrac{|y|^{-\gamma}}{
 (\log|y|)^{\beta}}
 dy
 \qquad
 (\tfrac{y}{\sqrt{t}}=z)
 \\
 &=
 \tfrac{t^{\frac{n-\gamma}{2}}}{(\log\sqrt t)^\beta}
 \int_{\frac{1}{\log\sqrt t}<|z|<\log\sqrt t}
 e^{-\frac{|z|^2}{8}}
 |z|^{-\gamma}
 \bigl|
 1+\tfrac{\log|z|}{\log\sqrt t}
 \bigr|^{-\beta}
 dz.
 \end{align*}
 Note that
 \begin{align}
 \label{e_4.10}
 \tfrac{\log|z|}{\log\sqrt t}
 &>
 \tfrac{\log\frac{1}{\log\sqrt t}}{\log\sqrt t}
 =
 -\tfrac{\log\log\sqrt t}{\log\sqrt t}
 \quad
 \text{for all }
 |z|>\tfrac{1}{\log\sqrt t}.
 \end{align}
 Hence,
 the integral on $\frac{\sqrt t}{\log\sqrt t}<|y|<\sqrt t\log\sqrt t$ is bounded by
 \begin{align}
 \label{e_4.11}
 &
 \int_{\frac{\sqrt t}{\log\sqrt t}<|y|<\sqrt t\log\sqrt t}
 e^{-\frac{|y|^2}{8t}}
 \tfrac{|y|^{-\gamma}}{
 (\log|y|)^{\beta}}
 dy
 \\
 \nonumber
 &<
 \tfrac{2^\beta t^{\frac{n-\gamma}{2}}}{(\log\sqrt t)^\beta}
 \int_{\R_z^n}
 e^{-\frac{|z|^2}{8}}
 |z|^{-\gamma}
 dz.
 \end{align}
 Combining \eqref{e_4.8} - \eqref{e_4.9} and \eqref{e_4.11},
 we deduce from \eqref{e_4.7} that there exist constants $C=C(n,\beta,\gamma)>0$ and $t_1>0$ such that
 \begin{align}
 \nonumber
 |\nabla_{\xi}
 \Theta(\xi,t)|
 &<
 \tfrac{C|\xi|}{t^\frac{n}{2}}
 \int_{\R_y^n}
 e^{-\frac{|y|^2}{8t}}
 \tfrac{|y|^{-\gamma}}{(\log |y|)^\beta}
 dy
 \\
 \nonumber
 &<
 \tfrac{C|\xi|}{t^\frac{n}{2}}
 \Bigl(
 1
 + 
 \tfrac{t^{\frac{n-\gamma}{2}}}{(\log\sqrt{t})^{\beta+(n-\gamma)}}
 +
 \tfrac{t^{\frac{n-\gamma}{2}}}{(\log\sqrt{t})^{\beta+1}}
 +
 \tfrac{t^{\frac{n-\gamma}{2}}}{(\log\sqrt{t})^\beta}
 \Bigr)
 \\
 \label{e_4.12}
 &<
 \tfrac{C|\xi|}{t^\frac{\gamma}{2}(\log\sqrt t)^\beta}
 \quad
 \text{for all }
 |\xi|<1
 \text{ and }
 t>t_1.
 \end{align}
 From \eqref{e_4.12},
 we conclude that
 \begin{align*}
 |
 \theta(x,t)
 -
 \theta(0,t)
 |
 &=
 |
 \Theta(\xi,t)
 -
 \Theta(0,t)
 |
 \qquad
 (x=\xi\sqrt t)
 \\
 &<
 |\nabla_{\xi}\Theta(\sigma\xi,t)\cdot\xi|
 \qquad
 (\exists\sigma\in(0,1))
 \\
 &<
 \tfrac{C|\sigma\xi|\cdot|\xi|}{t^\frac{\gamma}{2}(\log\sqrt t)^\beta}
 \\
 &<
 \tfrac{C|\xi|^2}{t^\frac{\gamma}{2}(\log\sqrt t)^\beta}
 \quad
 \text{for }
 |\xi|<1
 \text{ and }
 t>t_1.
 \end{align*}
 \end{proof}

 \begin{lem}
 \label{lemma_4.2}
 Suppose that $n\geq3$, $\gamma\in(0,n)$ and $\beta>0$.
 Let $R_1$ and $R_2$ be two positive constants
 satisfying
 \begin{align*}
 e<R_1<R_2
 \quad
 \text{ and }
 \quad
 R_1\log R_1<\tfrac{R_2}{\log R_2}.
 \end{align*}
 Assume further that $\theta_0(x)$ satisfies the following{\rm:}
 \begin{itemize}
 \setlength{\itemsep}{2mm}
 \item $\theta_0(x)$ is radially symmetric{\rm;} 
 \item $\theta_0(x)=|x|^{-\gamma}(\log|x|)^{-\beta}$ \quad \text{for} $R_1<|x|<R_2${\rm;}
 \item $|\theta_0(x)|<\begin{cases}
 |x|^{-\gamma}|\log|x||^{-\beta} & \text{for } |x|>e,
 \\
 e^{-\gamma} & \text{for } |x|<e.
 \end{cases}$
 \end{itemize}
 There exist two positive constants $R_0=R_0(n,\gamma,\beta)$ and $C=C(n,\gamma,\beta)$,
 independent of $R_1$ and $R_2$,
 such that
 if $R_1>R_0$,
 then any solution $\theta(x,t)$ of \eqref{e_4.1} satisfies
 \begin{align}
 \label{equation_4.13}
 \left|
 \theta(0,t)
 -
 \tfrac{t^{-\frac{\gamma}{2}}}{(4\pi)^{\frac{n}{2}}
 (\log t)^\beta}
 \int_{\R_z^n}
 e^{-\frac{|z|^2}{4}}
 \tfrac{dz}{|z|^{\gamma}}
 \right|
 &<
 \tfrac{C}{t^{\frac{\gamma}{2}}(\log\sqrt t)^\beta}
 (\tfrac{1}{\log\sqrt t}+\tfrac{1}{(\log\sqrt t)^{n-\gamma}})
 \\
 \nonumber
 &
 \text{for all }
 \sqrt{t}\in(R_1\log R_1,\tfrac{R_2}{\log R_2}).
 \end{align}
 \end{lem}
 \begin{proof}
 We divide the integral into four parts.
 \begin{align*}
 &
 \theta(0,t)
 \\
 &=
 \tfrac{1}{(4\pi t)^{\frac{n}{2}}}
 \int_{\R_z^n}
 e^{-\frac{|y|^2}{4t}}
 \theta_0(y)
 dy
 \\
 &=
 \tfrac{1}{(4\pi t)^{\frac{n}{2}}}
 \Biggl(
 \underbrace{
 \int_{|y|<e}
 }_{=h_1}
 +
 \underbrace{
 \int_{e<|y|<\frac{2\sqrt t}{\log\sqrt t}}
 }_{=h_2}
 +
 \underbrace{
 \int_{\frac{2\sqrt t}{\log\sqrt t}<|y|<\sqrt t\log\sqrt t}
 }_{=J}
 +
 \underbrace{
 \int_{|y|>\sqrt t\log\sqrt t}
 }_{=h_3}
 \Biggr)
 \theta_0(y)
 dy
 \end{align*}
 As in \eqref{e_4.8} - \eqref{e_4.9},
 we verify that
 \begin{align}
 \label{equation_4.14}
 |h_1|
 &<
 \tfrac{C}{(4\pi t)^{\frac{n}{2}}},
 \\
 \label{equation_4.15}
 |h_2|
 &<
 \tfrac{Ct^{-\frac{\gamma}{2}}}{(\log\sqrt t)^{\beta+n-\gamma}},
 \\
 |h_3|
 \label{equation_4.16}
 &<
 \tfrac{Ct^{-\frac{\gamma}{2}}}{(4\pi)^{\frac{n}{2}}(\log\sqrt t)^{\beta+1}}.
 \end{align}
 We now evaluate the main contribution $J$ to $\theta(0,t)$.
 Note that there exists $R_0>0$,
 independent of $R_1$ and $R_2$,
 such that
 if $R_1>R_0$,
 then the following estimates hold:
 \begin{itemize}
 \item 
 If
 $\sqrt t>R_1\log R_1$,
 then
 \begin{align*}
 R_1
 <
 \tfrac{\sqrt{t}}{\log\sqrt{t}-2\log\log\sqrt{t}}
 <
 \tfrac{5}{4}
 \tfrac{\sqrt{t}}{\log\sqrt{t}}.
 \end{align*}
 
 \item
 If
 $\sqrt t<\tfrac{R_2}{\log R_2}$,
 then
 \begin{align*}
 R_2
 >
 \sqrt{t}\log\sqrt{t}
 +
 \tfrac{\sqrt{t}}{2}\log\log\sqrt{t}
 >
 \sqrt{t}\log\sqrt{t}.
 \end{align*}
 \end{itemize}
 Therefore,
 from the assumption on $\theta_0(x)$,
 we have
 \begin{align*}
 J
 &=
 \tfrac{1}{(4\pi t)^{\frac{n}{2}}}
 \int_{\frac{2\sqrt t}{\log\sqrt t}<|y|<\sqrt t\log\sqrt t}
 e^{-\frac{|y|^2}{4t}}
 \theta_0(y)
 dy
 \\
 &=
 \tfrac{1}{(4\pi t)^{\frac{n}{2}}}
 \int_{\frac{2\sqrt t}{\log\sqrt t}<|y|<\sqrt t\log\sqrt t}
 e^{-\frac{|y|^2}{4t}}
 |y|^{-\gamma}
 (\log|y|)^{-\beta}
 dy
 \\
 &=
 \tfrac{t^{-\frac{\gamma}{2}}}{(4\pi )^{\frac{n}{2}}(\log\sqrt t)^\beta}
 \int_{\frac{2}{\log\sqrt t}<|z|<\log\sqrt t}
 e^{-\frac{|z|^2}{4}}
 |z|^{-\gamma}
 \left|
 1+\tfrac{\log|z|}{\log\sqrt t}
 \right|^{-\beta}
 dz
 \\
 &=
 \tfrac{t^{-\frac{\gamma}{2}}}{(4\pi )^{\frac{n}{2}}(\log\sqrt t)^\beta}
 \int_{\R_z^n}
 e^{-\frac{|z|^2}{4}}
 |z|^{-\gamma}
 dz
 +
 h_4
 +
 h_5
 +
 h_6,
 \end{align*}
 where
 \begin{align*}
 h_4
 &=
 \tfrac{t^{-\frac{\gamma}{2}}}{(4\pi )^{\frac{n}{2}}(\log\sqrt t)^\beta}
 \int_{\frac{2}{\log\sqrt t}<|z|<\log\sqrt t}
 e^{-\frac{|z|^2}{4}}
 |z|^{-\gamma}
 \left(
 \left|
 1+\tfrac{\log|z|}{\log\sqrt t}
 \right|^{-\beta}
 -
 1
 \right)
 dz,
 \\
 h_5
 &=
 \tfrac{-t^{-\frac{\gamma}{2}}}{(4\pi )^{\frac{n}{2}}(\log\sqrt t)^\beta}
 \int_{|z|<\frac{2}{\log\sqrt t}}
 e^{-\frac{|z|^2}{4}}
 |z|^{-\gamma}
 dz,
 \\
 h_6
 &=
 \tfrac{-t^{-\frac{\gamma}{2}}}{(4\pi )^{\frac{n}{2}}(\log\sqrt t)^\beta}
 \int_{|z|>\log\sqrt t}
 e^{-\frac{|z|^2}{4}}
 |z|^{-\gamma}
 dz.
 \end{align*}
 It is clear that
 \begin{align*}
 |h_5|
 &<
 \tfrac{Ct^{-\frac{\gamma}{2}}}{(4\pi )^{\frac{n}{2}}(\log\sqrt t)^{\beta+n-\gamma}},
 \\
 |h_6|
 &<
 \tfrac{Ct^{-\frac{\gamma}{2}}}{(4\pi )^{\frac{n}{2}}(\log\sqrt t)^{\beta+1}}.
 \end{align*}
 We recall from \eqref{e_4.10} that
 $\frac{\log|z|}{\log\sqrt t}>-\frac{\log\log\sqrt t}{\log\sqrt t}$
 for $|z|>\frac{1}{\log\sqrt t}$.
 Hence,
 it follows that
 \begin{align*}
 \left|
 \left|
 1+\tfrac{\log|z|}{\log\sqrt t}
 \right|^{-\beta}
 -
 1
 \right|
 &=
 \left|
 \left(
 1+\tfrac{\log|z|}{\log\sqrt t}
 \right)^{-\beta}
 -
 1
 \right|
 \\
 &=
 \left|
 \int_0^1
 (-\beta)
 \left(
 1+\tfrac{s\log|z|}{\log\sqrt t}
 \right)^{-\beta-1}
 ds
 (\tfrac{\log|z|}{\log\sqrt t})
 \right|
 \\
 &<
 \tfrac{C|\log|z||}{\log\sqrt t}
 \quad
 \text{for all }
 |z|>\tfrac{1}{\log\sqrt t}.
 \end{align*}
 This estimate implies that
 \begin{align*}
 |h_4|
 &<
 \tfrac{t^{-\frac{\gamma}{2}}}{(4\pi )^{\frac{n}{2}}(\log\sqrt t)^\beta}
 \int_{\frac{2}{\log\sqrt t}<|z|<\log\sqrt t}
 e^{-\frac{|z|^2}{4}}
 |z|^{-\gamma}
 \tfrac{C|\log|z||}{\log\sqrt t}
 dz
 \\
 &<
 \tfrac{Ct^{-\frac{\gamma}{2}}}{(4\pi )^{\frac{n}{2}}(\log\sqrt t)^{\beta+1}}
 \int_{\R_z^n}
 e^{-\frac{|z|^2}{4}}
 \tfrac{|\log|z||}{|z|^\gamma}
 dz.
 \end{align*}
 \end{proof}

 \begin{lem}
 \label{lemma_4.3}
 Assume that $n\geq3$.
 Let $\theta(x,t)$ be a solution of {\rm \eqref{e_4.1}} with the initial data $\theta_0(x)$
 satisfying the following condition{\rm:}
 \begin{itemize}
 \item $|\theta_0(x)|<\begin{cases}
 |x|^{-\gamma}(\log|x|)^{-\beta} & \text{for } |x|>e,
 \\
 e^{-\gamma} & \text{for } |x|<e.
 \end{cases}$
 \end{itemize}
 Then,
 there exist two positive constants $t_1=t_1(n,\gamma,\beta)$ and $C=C(n,\gamma,\beta)$
 such that{\rm:}
 \begin{align}
 \label{equation_4.17}
 |\theta(0,t)|
 &<
 \tfrac{Ct^{-\frac{\gamma}{2}}}{(\log\sqrt t)^\beta}
 \quad
 \text{for }
 t>t_1,
 \\
 \label{equation_4.18}
 |\pa_t\theta(0,t)|
 &<
 \tfrac{Ct^{-\frac{\gamma}{2}-1}}{(\log\sqrt t)^\beta}
 \quad
 \text{for }
 t>t_1.
 \end{align}
 \end{lem}
 \begin{proof}
 The proof of Lemma \ref{lemma_4.2} immediately implies both estimates.
 \end{proof}

\subsection{Choice of Initial Data for Oscillatory Solutions}
\label{section_4.2}
 Fix a sequence $\{R_j\}_{j=1}^\infty$ of positive numbers satisfying the following conditions:
 \begin{align}
 \label{equation_4.19}
 e<R_j<R_{j+1},
 \quad
 (2R_j)\log(2R_j)<\tfrac{R_{j+1}}{\log R_{j+1}}
 \quad
 \text{and}
 \quad
 \lim_{j\to\infty}
 R_j=\infty.
 \end{align}
 We now choose a radially symmetric initial function
 \[
 \theta_0(x)\in C(\R^n)\cap L^\infty(\R^n)\cap \dot H^1(\R^n).
 \]
 Let $\chi_j$ be the cut off function defined by
 $\chi_j
 =
 \chi(\tfrac{|x|}{R_j})$,
 and define
 \begin{align*}
 f_0(x)
 =
 |x|^{-\frac{n-2}{2}}(\log |x|)^{-\beta}.
 \end{align*}
 We define $\theta_0(x)$ explicitly as follows:
 \begin{align}
 \label{equation_4.20}
 \theta_0(x)
 &=
 \begin{cases}
 e^{-\frac{n-2}{2}}
 & \text{for } |x|<e,
 \\
 f_0(x)
 & \text{for } e<|x|<R_1,
 \\
 f_0(x)
 \chi_1
 -
 f_0(x)
 (1-\chi_1)
 & \text{for } R_1<|x|<2R_1,
 \\
 -f_0(x)
 & \text{for } 2R_1<|x|<R_2,
 \\
 -
 f_0(x)
 \chi_2
 +
 f_0(x)
 (1-\chi_2)
 & \text{for } R_2<|x|<2R_2,
 \\
 f_0(x)
 & \text{for } 2R_2<|x|<R_3,
 \\
 f_0(x)
 \chi_3
 -
 f_0(x)
 (1-\chi_3)
 & \text{for } R_3<|x|<2R_3,
 \\
 \cdots
 \end{cases}
 \end{align}
%
%
 Let $\theta(x,t)$ be the solution to the linear heat equation starting from
 this initial data $\theta_0(x)$.
 From Lemma \ref{lemma_4.2},
 there exist constants $C>0$ and $R_0>0$ such that,
 if $R_1>R_0$,
 then the following estimates hold:
 \begin{align}
 \label{equation_4.23}
 \begin{cases}
 |\theta(0,t)-\frac{-A_1}{t(\log t)^\beta}|
 <
 \frac{C}{t(\log t)^{\beta+1}}
 &
 \text{for } \sqrt t\in(2R_1\log 2R_1,\tfrac{R_{2}}{\log R_{2}}),
 \\
 |\theta(0,t)-\frac{A_1}{t(\log t)^\beta}|
 <
 \frac{C}{t(\log t)^{\beta+1}}
 &
 \text{for } \sqrt t\in(2R_{2}\log 2R_{2},\tfrac{R_{3}}{\log R_{3}}),
 \\
 |\theta(0,t)-\frac{-A_1}{t(\log t)^\beta}|
 <
 \frac{C}{t(\log t)^{\beta+1}}
 &
 \text{for } \sqrt t\in(2R_3\log 2R_3,\tfrac{R_{4}}{\log R_{4}}),
 \\
 |\theta(0,t)-\frac{A_1}{t(\log t)^\beta}|
 <
 \frac{C}{t(\log t)^{\beta+1}}
 &
 \text{for } \sqrt t\in(2R_{4}\log 2R_{4},\tfrac{R_{5}}{\log R_{5}}),
 \\
 \cdots
 \end{cases}
 \end{align}
 Moreover,
 from Lemma \ref{lemma_4.3},
 the absolute value of \( \theta(0,t) \) is uniformly bounded by:
 \begin{align}
 \label{equation_4.24}
 |\theta(0,t)|
 <
 \tfrac{2A_1}{t(\log t)^\beta}
 \quad
 \text{for } t>t_1.
 \end{align}
 The constant $A_1$ is given by
 $A_1=
 (4\pi)^{-\frac{n}{2}}
 \int_{z\in\R^n}
 e^{-\frac{|z|^2}{4}}
 |z|^{-\frac{n-2}{2}}
 dz$.

\begin{lem}
\label{lemma_4.4}
Assume that \( n \geq 3 \) and \( \beta' > \frac{1}{2} \).  
Let \( \{R_j\}_{j=1}^\infty \) be a sequence satisfying \eqref{equation_4.19},  
and let \( \theta_0(x) \) be the initial function defined by \eqref{equation_4.20}.  
Then the solution \( \theta(x,t) \) satisfies  
\[
\lim_{t \to \infty} \theta(x,t) = 0
\quad \text{in } \dot H^1(\mathbb{R}^n).
\]
\end{lem}
 \begin{proof}
 A direct computation shows that $\theta_0(x)\in\dot H^1(\R^n)$.
 Therefore,
 by standard density arguments for the heat equation in the energy space,  
 we conclude that the solution \( \theta(x,t) \) satisfies  
 \[
 \lim_{t \to \infty} \theta(x,t) = 0 \quad \text{in } \dot H^1(\mathbb{R}^n).
 \]
 \end{proof}

 \section{Growth Speed of Oscillatory Solutions in a formal level}
 \label{section_5}
 We adopt an argument similar to those in
 Marek Fila - John R. King \cite{Fila-King} and
 Stathis Filippas - Miguel A. Herrero -  Juan J. L. Vel\'azquez \cite{FilippasHV}
 (see also Manuel del Pino - Monica Musso - Juncheng Wei \cite{delPino-Musso-Wei_3dim_long} and
 \cite{Wei-Qidi-Yifu,Li-Wei-Qidi-Yifu}).
 In this section,
 we carry out the construction of the solutions described in Theorem \ref{theorem_3} at a formal level.
 This argument also covers the cases corresponding to Theorem \ref{theorem_1} and \ref{theorem_2}.
 Let $\theta_0(x)$ be the function defined in Section \ref{section_4.2},
 and
 let $\theta(x,t)$ be the unique solution of
 \begin{align*}
 \begin{cases}
 \theta_t=\Delta_x\theta
 &
 \text{for }
 (x,t)\in\R^n\times(0,\infty),
 \\
 \theta(x,t)|_{t=0}
 =
 \theta_0(x)
 &
 \text{for }
 x\in\R^n.
 \end{cases}
 \end{align*}
 The sequence $\{R_j\}_{j=1}^\infty$ (see Section \ref{section_4.2}) will be carefully chosen later.
 We now construct a solution of the form:
 \begin{align}
 \label{e_5.1}
 u(x,t)
 &=
 \lambda(t)^{-\frac{n-2}{2}}
 \bigl\{
 {\sf Q}(\tfrac{x}{\lambda(t)})+\sigma(t)T_1(\tfrac{x}{\lambda(t)})
 \bigr\}
 \chi_1
 +
 \theta(x,t)
 (1-\chi_1)
 +
 v(x,t)
 \\
 \nonumber
 &\text{with }
 0<\lambda(t)\ll\sqrt t,
 \end{align}
 where $v(x,t)$ denotes a remainder.

\subsection{Inner Region: $|x|\sim\lambda(t)$}
\label{section_5.1}
 We seek an inner solution of the form:
 \begin{align}
 \label{e_5.2}
 u(x,t)
 =
 \lambda(t)^{-\frac{n-2}{2}}({\sf Q}(y)+\sigma(t)T_1(y)+\epsilon(y,t))
 \quad
 \text{with }
 y=\tfrac{x}{\lambda(t)},
 \end{align}
 where $\lambda(t)$ and $\sigma(t)$
 are unknown at this level.
 Here, the inner region refers to the spatial domain defined by $|x| < R\lambda(t)$, or equivalently, $|y| < R$, 
 for a sufficiently large fixed constant $R > 0$.
 Define
 \[
 \Lambda_y=\tfrac{n-2}{2}+y\cdot\nabla_y,
 \quad
 H_y
 =
 \Delta_y+f'({\sf Q}(y)).
 \]
 The function $\epsilon(y,t)$ solves
 \[
 \lambda^2\pa_t\epsilon
 +
 \lambda^2\dot\sigma T_1
 =
 H_y\epsilon
 +
 \{\lambda\dot\lambda(\Lambda_y{\sf Q})+\sigma H_yT_1\}
 +
 \sigma^2\Lambda_yT_1
 +
 \sigma\Lambda_y\epsilon
 +
 N,
 \]
 where $N$ represents a nonlinear term given by
 $N=f({\sf Q}+\epsilon)-f({\sf Q})-f'({\sf Q})\epsilon$.
 To cancel the leading-order term:
 $\lambda\dot\lambda(\Lambda_y{\sf Q})+\sigma H_yT_1$,
 we choose $\sigma(t)$ and $T_1(y)$ as follows.
 \begin{align}
 \label{e_5.3}
 \begin{cases}
 \sigma=\lambda\dot\lambda,
 \\
 H_yT_1+\Lambda_y{\sf Q}=0.
 \end{cases}
 \end{align}
 Since $\epsilon(y,t)$ is a remainder term,
 we expect that $|\epsilon(y,t)|\ll\sigma(t)T_1(y)$ in \eqref{e_5.2}.
 Therefore,
 the solution is expected to be approximated by
 \begin{equation}
 \label{e_5.4}
 u(x,t)
 =
 \lambda(t)^{-\frac{n-2}{2}}{\sf Q}(y)
 +
 \lambda(t)^{-\frac{n-2}{2}}\sigma(t)T_1(y)
 \end{equation}
in the inner region.
From (4.2) - (4.3) in \cite{Harada_6D},
we recall that
the radial function $T_1(y)$ satisfies
 \begin{align}
 \label{e_5.5}
 T_1(r)
 &=
 \tfrac{4}{5}+O(r^{-2})
 \qquad
 (n=6),
 \\
 \label{e_5.6}
 \pa_rT_1(r)
 &=
 O(r^{-3})
 \qquad
 (n=6).
 \end{align}

\subsection{Selfsimilar Region: $|x|\sim\sqrt{t}$}
\label{section_5.2}
 In the selfsimilar region:
 \begin{align*}
 |x|\sim\sqrt{t},
 \end{align*}
 we assume that
 the contribution from the nonlinear term $|u|^{p-1}u$
 in $u_t=\Delta_xu+|u|^{p-1}u$ is negligible.
 Under these circumstances,
 the solution $u(x,t)$ is expected to be well approximated by
 a solution of the linear heat equation in this region.
 Let $\theta(x,t)$ be the function defined at the beginning of Section \ref{section_5}.
 We now assume that
 \begin{align}
 \label{e_5.7}
 u(x,t)
 =
 \theta(x,t)
 \quad
 \text{for }
 |x|\sim\sqrt t.
 \end{align}

 \subsection{Matching Conditions and Choice of $\{R_j\}_{j=1}^\infty$}
 \label{section_5.3}
 We derive conditions under which the two solutions obtained in Sections \ref{section_5.1} and \ref{section_5.2}
 match at the boundary of the two regions.
 By comparing the forms of the solutions in the two regions (see \eqref{e_5.4} and \eqref{e_5.7}),
 we find that
 \begin{align*}
 \lambda^{-\frac{n-2}{2}}{\sf Q}(y)
 &+
 \lambda^{-\frac{n-2}{2}}\sigma T_1(y)
 =
 \theta(0,t)
 \\
 &
 \text{for } |y|\to\infty,\ \ |x|\ll\sqrt{t}.
 \end{align*}
 Neglecting the first term on the left-hand side,
 and using the asymptotic formula \eqref{e_5.5},
 we obtain
 \[
 \begin{cases}
 \lambda^{-2}
 \sigma
 \cdot
 \tfrac{4}{5}
 =
 \theta(0,t),
 \\
 \sigma
 =
 \lambda\dot\lambda
 \end{cases}
 \quad
 (n=6).
 \]
 From \eqref{equation_4.23},
 we expect that $\lambda(t)$ satisfies the following differential equation.
 \begin{align}
 \label{e_5.8}
 \begin{cases}
 \tfrac{\dot\lambda}{\lambda}
 =
 \tfrac{-5A_1}{4t(\log t)^\beta}
 &
 \text{for }
 \sqrt t\in(2R_1\log2R_1,\tfrac{R_2}{\log R_2}),
 \\[2mm]
 \tfrac{\dot\lambda}{\lambda}
 =
 \tfrac{5A_1}{4t(\log t)^\beta}
 &
 \text{for }
 \sqrt t\in(2R_2\log 2R_2,\tfrac{R_3}{\log R_3}),
 \\[2mm]
 \tfrac{\dot\lambda}{\lambda}
 =
 \tfrac{-5A_1}{4t(\log t)^\beta}
 &
 \text{for }
 \sqrt t\in(2R_3\log 2R_3,\tfrac{R_4}{\log R_4}),
 \\[2mm]
 \tfrac{\dot\lambda}{\lambda}
 =
 \tfrac{5A_1}{4t(\log t)^\beta}
 &
 \text{for }
 \sqrt t\in(2R_4\log 2R_4,\tfrac{R_5}{\log R_5}),
 \\
 \cdots
 \end{cases}
 \end{align}
Additionally,
the function $\lambda(t)$ satisfies the bound
 \begin{align}
 \label{e_5.9}
 |\tfrac{\dot\lambda}{\lambda}|
 <
 \tfrac{5A_1}{2t(\log t)^\beta}
 \quad
 \text{for }
 \sqrt t>R_1\log R_1.
 \end{align}
 To simplify notation,
 we introduce the notation \( t_j^{\pm} \) as follows:
 \begin{align*}
 \sqrt{t_j^{+}}
 =
 2R_j\log 2R_j,
 \quad
 \sqrt{t_j^{-}}
 =
 \tfrac{R_j}{\log R_j}.
 \end{align*}
We also define
 \begin{align*}
 q_1=\tfrac{5A_1}{4(1-\beta)}.
 \end{align*}
 The differential equation \eqref{e_5.8} - \eqref{e_5.9} can be rewritten as:
 \begin{align}
 \label{e_5.10}
 \begin{cases}
 \tfrac{\dot\lambda}{\lambda}
 =
 \tfrac{-(1-\beta)q_1}{t(\log t)^\beta}
 &
 \text{for }
 t\in(t_1^+,t_2^-),
 \\[2mm]
 \tfrac{\dot\lambda}{\lambda}
 =
 \tfrac{(1-\beta)q_1}{t(\log t)^\beta}
 &
 \text{for }
 t\in(t_2^+,t_3^-),
 \\[2mm]
 \tfrac{\dot\lambda}{\lambda}
 =
 \tfrac{-(1-\beta)q_1}{t(\log t)^\beta}
 &
 \text{for }
 t\in(t_3^+,t_4^-),
 \\[2mm]
 \tfrac{\dot\lambda}{\lambda}
 =
 \tfrac{(1-\beta)q_1}{t(\log t)^\beta}
 &
 \text{for }
 t\in(t_4^+,t_5^-),
 \\
 \cdots
 \end{cases}
 \end{align}
 with the bound
 \begin{align}
 \label{e_5.11}
 |\tfrac{\dot\lambda}{\lambda}|
 <
 \tfrac{2(1-\beta)q_1}{t(\log t)^\beta}
 \quad
 \text{for all }
 t\in(t_j^-,t_j^+).
 \end{align}

 \begin{lem}
 \label{lemma_5.1}
 Let $\beta\in(\frac{1}{2},1)$,
 and let $\bar{\sf n}\in\N$ be a large constant to be specified in the proof.
 Fix ${\sf n}_1>\bar{\sf n}$.
 Define the sequences $\{t_j\}_{j=1}^\infty$ and $\{R_j\}_{j=1}^\infty$ by
 \begin{itemize}
 \item 
 $t_j
 =
 e^{{\sf p}_j}$ with ${\sf p}_j={\sf n}_1^\frac{j}{1-\beta}${\rm;}
 \item
 $R_j=\sqrt{t_j}$.
 \end{itemize}
 Define the sequence $\{t_j^{\pm}\}_{j\in\N}$ by
 \begin{itemize}
 \item $t_j^{+}=(2R_j\log 2R_j)^2${\rm;}
 \item $t_j^{-}=(\tfrac{R_j}{\log R_j})^2$.
 \end{itemize}
 We also set
 \begin{itemize}
 \item $t_I=t_1^-$.
 \end{itemize}
 We consider the following perturbed version of the system \eqref{e_5.10} - \eqref{e_5.11}.
 \begin{align}
 \label{e_5.12}
 \begin{cases}
 \tfrac{\dot\lambda}{\lambda}
 =
 -
 \tfrac{(1-\beta)q_1}{t(\log t)^\beta}
 +
 D_1(t)
 &
 \text{for }
 t\in(t_1^+,t_2^-),
 \\[2mm]
 \tfrac{\dot\lambda}{\lambda}
 =
 +
 \tfrac{(1-\beta)q_1}{t(\log t)^\beta}
 +
 D_2(t)
 &
 \text{for }
 t\in(t_2^+,t_3^-),
 \\[2mm]
 \tfrac{\dot\lambda}{\lambda}
 =
 -
 \tfrac{(1-\beta)q_1}{t(\log t)^\beta}
 +
 D_3(t)
 &
 \text{for }
 t\in(t_3^+,t_4^-),
 \\[2mm]
 \tfrac{\dot\lambda}{\lambda}
 =
 +
 \tfrac{(1-\beta)q_1}{t(\log t)^\beta}
 +
 D_4(t)
 &
 \text{for }
 t\in(t_4^+,t_5^-),
 \\
 \cdots
 \end{cases}
 \end{align}
 with the bound
 \begin{align}
 \label{e_5.13}
 |\tfrac{\dot\lambda}{\lambda}|
 <
 \tfrac{2(1-\beta)q_1}{t(\log t)^\beta}
 \quad
 \text{for }
 t\in(t_j^-,t_j^+)
 \qquad
 (\forall j\in\N),
 \end{align}
 and the initial condition
 \begin{align}
 \label{e_5.14}
 \lambda(t_1^-)
 =
 e^{q_1(\log t_1^-)^{1-\beta}}.
 \end{align}
 Assume that
 there exist constants $\beta'>1$ and $C_1>0$, 
 independent of $j\in\N$, such that
 \begin{align*}
 |D_j(t)|<\tfrac{C_1}{t(\log t)^{\beta'}}
 \quad
 \text{\rm for all }
 t\in(t_j^+,t_{j+1}^-),\
 j\in\N. 
 \end{align*}
 Then,
 we have{\rm:}
 \begin{itemize}
 \item
 For even $j\in\N$,
 \begin{align}
 \label{e_5.15}
 \log[\tfrac{\lambda(t_j^-)}{\lambda(t_1^-)}]
 &<
 -
 \tfrac{q_1}{2}
 {\sf n}_1^j
 +
 6q_1
 {\sf n}_1^{j-1}
 +
 \tfrac{C_1}{\beta'-1}
 (\log t_1^-)^{-(\beta'-1)}.
 \end{align}
 \item
 For odd $j\in\N$,
 \begin{align}
 \label{e_5.16}
 \log[\tfrac{\lambda(t_j^-)}{\lambda(t_1^-)}]
 &>
 \tfrac{q_1}{2}
 {\sf n}_1^j
 -
 6q_1
 {\sf n}_1^{j-1}
 -
 \tfrac{C_1}{\beta'-1}
 (\log t_1^-)^{-(\beta'-1)}.
 \end{align}
 \end{itemize}
 Furthermore,
 we have
 \begin{align}
 \label{e_5.17}
 \lambda(t)
 &<
 e^{-q_1(\log t_I)^{1-\beta}}
 e^{2q_1(\log t)^{1-\beta}}
 \quad
 \text{ \rm for all }
 t\in[t_I,\infty).
 \end{align}
 \end{lem}
 \begin{proof}
 Integrating both sides of \eqref{e_5.13} over $t\in(t_j^-,t_j^+)$,
 we verify that
 for each $j\in\N$
 \begin{align}
 \label{e_5.18}
 \log\tfrac{\lambda(t)}{\lambda(t_j^-)}
 &<
 2q_1
 (\log t)^{1-\beta}
 -
 2q_1
 (\log t_j^-)^{1-\beta}
 \quad
 \text{for all }
 t\in(t_j^-,t_{j}^+),
 \\
 \label{e_5.19}
 \log\tfrac{\lambda(t)}{\lambda(t_j^-)}
 &>
 -
 2q_1
 (\log t)^{1-\beta}
 +
 2q_1
 (\log t_j^-)^{1-\beta}
 \quad
 \text{for all }
 t\in(t_j^-,t_{j}^+). 
 \end{align}
 Furthermore,
 from \eqref{e_5.12},
 we obtain the following identity:
 For even $j\in\N$,
 \begin{align}
 \label{e_5.20}
 \log\tfrac{\lambda(t)}{\lambda(t_j^+)}
 &=
 q_1(\log t)^{1-\beta}
 -
 q_1
 (\log t_j^+)^{1-\beta}
 +
 \int_{t_j^+}^t
 D_j(s)
 ds
 \\
 \nonumber
 &
 \text{for all }
 t\in(t_j^+,t_{j+1}^-).
 \end{align}
 For odd $j\in\N$,
 \begin{align}
 \label{e_5.21}
 \log\tfrac{\lambda(t)}{\lambda(t_j^+)}
 &=
 -
 q_1
 (\log t)^{1-\beta}
 +
 q_1
 (\log t_j^+)^{1-\beta}
 +
 \int_{t_j^+}^t
 D_j(s)
 ds
 \\
 \nonumber
 &
 \text{for all }
 t\in(t_j^+,t_{j+1}^-).
 \end{align}
 For simplicity,
 we define a function $D(t)$ on $t\in(t_1^-,\infty)$ by
 \begin{align*}
 D(t)
 =
 \begin{cases}
 D_j(t)
 &
 \text{if }
 t\in(t_j^+,t_{j+1}^-),
 \\[2mm]
 0
 &
 \text{if }
 t\in\bigcup_{j\in\N}(t_j^-,t_j^+).
 \end{cases}
 \end{align*}
 We now claim the following two-sided bounds \eqref{e_5.22} - \eqref{e_5.25}
 for the function $\lambda(t)$.
 For even $j\in\N$,
 \begin{align}
 \label{e_5.22}
 \log[\tfrac{\lambda(t_j^-)}{\lambda(t_1^-)}]
 &<
 -
 q_1(\log t_j^-)^{1-\beta}
 +
 3q_1(\log t_{j-1}^+)^{1-\beta}
 +
 \int_{t_1^+}^{t_j^-}
 |D(s)|ds,
 \\
 \label{e_5.23}
 \log[\tfrac{\lambda(t_j^-)}{\lambda(t_1^-)}]
 &>
 -
 q_1
 (\log t_j^-)^{1-\beta}
 -
 q_1(\log t_{j-1}^+)^{1-\beta}
 -
 \int_{t_1^+}^{t_j^-}
 |D(s)|ds.
 \end{align}
 For odd $j\in\N$,
 \begin{align}
 \label{e_5.24}
 \log[\tfrac{\lambda(t_j^-)}{\lambda(t_1^-)}]
 &<
 q_1(\log t_j^-)^{1-\beta}
 +
 q_1(\log t_{j-1}^+)^{1-\beta}
 +
 \int_{t_1^+}^{t_j^-}
 |D(s)|ds,
 \\
 \label{e_5.25}
 \log[\tfrac{\lambda(t_j^-)}{\lambda(t_1^-)}]
 &>
 q_1(\log t_j^-)^{1-\beta}
 -
 3q_1
 (\log t_{j-1}^+)^{1-\beta}
 -
 \int_{t_1^+}^{t_j^-}
 |D(s)|ds.
 \end{align}
 See Appendix \ref{section_9.2} for the proof.
 From the choice of $R_j$ and $t_j$,
 it follows that
 \[
 \log R_j
 =
 \tfrac{1}{2}\log t_j
 =
 \tfrac{{\sf p}_j}{2}.
 \]
 From the definition of $t_j^\pm$,
 we have
 \begin{align*}
 \log\sqrt{t_j^{+}}
 &=
 \log(2R_j\log 2R_j)
 =
 \log 2R_j
 +
 \log\log 2R_j
 \\
 &=
 \tfrac{{\sf p}_j}{2}
 +
 \log2
 +
 \log(\tfrac{{\sf p}_j}{2}+\log 2),
 \\
 \log\sqrt{t_j^{-}}
 &=
 \log(\tfrac{R_j}{\log R_j})
 =
 \log R_j
 -
 \log\log R_j
 \\
 &=
 \tfrac{{\sf p}_j}{2}
 -
 \log\tfrac{{\sf p}_j}{2}.
 \end{align*}
 This implies
 \begin{align*}
 \log t_j^{+}
 &=
 {\sf p}_j
 +
 2\log 2
 +
 2\log(\tfrac{{\sf p}_j}{2}+\log2),
 \\
 \log t_j^{-}
 &=
 {\sf p}_j
 -
 2\log \tfrac{{\sf p}_j}{2}.
 \end{align*}
 From \eqref{e_5.22},
 we deduce that
 for even $j\in\N$
 \begin{align*}
 &
 \log[\tfrac{\lambda(t_j^-)}{\lambda(t_1^-)}]
 \\
 &<
 -
 q_1(\log t_j^-)^{1-\beta}
 +
 3q_1(\log t_{j-1}^+)^{1-\beta}
 +
 \int_{t_1^+}^\infty
 |D(s)|
 ds 
 \\
 &<
 -
 q_1
 ({\sf p}_{j}-2\log\tfrac{{\sf p}_j}{2})^{1-\beta}
 +
 3q_1
 \Bigl(
 {\sf p}_{j-1}+2\log2
 \\
 &\quad
 +2\log(\tfrac{{\sf p}_{j-1}}{2}+\log2)
 \Bigr)^{1-\beta}
 +
 \int_{t_1^+}^\infty
 |D(s)|
 ds
 \\
 &=
 -
 q_1
 \Bigl(
 {\sf n}_1^\frac{j}{1-\beta}-2\log\tfrac{{\sf n}_1^\frac{j}{1-\beta}}{2}
 \Bigr)^{1-\beta}
 +
 3q_1
 \Bigl({\sf n}_1^\frac{j-1}{1-\beta}+2\log2
 \\
 &\quad
 +
 2\log(\tfrac{{\sf n}_1^\frac{j-1}{1-\beta}}{2}+\log2)
 \Bigr)^{1-\beta}
 +
 \int_{t_1^+}^\infty
 |D(s)|
 ds
 \\
 &<
 -
 \tfrac{q_1}{2}
 {\sf n}_1^j
 +
 6q_1
 {\sf n}_1^{j-1}
 +
 \int_{t_1^+}^\infty
 |D(s)|
 ds.
 \end{align*}
 In the last line,
 we used the fact that ${\sf n}_1$ is sufficiently large.
 The case of odd $j\in\N$ can be treated similarly.
 \begin{align*}
 &
 \log[\tfrac{\lambda(t_j^-)}{\lambda(t_1^-)}]
 \\
 &>
 q_1(\log t_j^-)^{1-\beta}
 -
 3q_1(\log t_{j-1}^+)^{1-\beta}
 -
 \int_{t_1^+}^\infty
 |D(s)|
 ds 
 \\
 &>
 q_1
 ({\sf p}_{j}-2\log\tfrac{{\sf p}_j}{2})^{1-\beta}
 -
 3q_1
 \Bigl(
 {\sf p}_{j-1}+2\log2
 \\
 &\quad
 +2\log(\tfrac{{\sf p}_{j-1}}{2}+\log2)
 \Bigr)^{1-\beta}
 -
 \int_{t_1^+}^\infty
 |D(s)|
 ds
 \\
 &=
 q_1
 ({\sf n}_1^\frac{j}{1-\beta}-2\log\tfrac{{\sf n}_1^\frac{j}{1-\beta}}{2})^{1-\beta}
 -
 3q_1
 \Bigl({\sf n}_1^\frac{j-1}{1-\beta}+2\log2
 \\
 &\quad
 +
 2\log(\tfrac{{\sf n}_1^\frac{j-1}{1-\beta}}{2}+\log2)
 \Bigr)^{1-\beta}
 -
 \int_{t_1^+}^\infty
 |D(s)|
 ds
 \\
 &>
 \tfrac{q_1}{2}
 {\sf n}_1^j
 -
 6q_1
 {\sf n}_1^{j-1}
 -
 \int_{t_1^+}^\infty
 |D(s)|
 ds.
 \end{align*}
 Thus,
 the inequalities \eqref{e_5.15} - \eqref{e_5.16} are established.
 To derive \eqref{e_5.17},
 we integrate \eqref{e_5.13} over $t\in(t_I,t)$.
 This gives
 \begin{align*}
 \log[\tfrac{\lambda(t)}{\lambda(t_I)}]
 <
 2q_1(\log t)^{1-\beta}
 -
 2q_1(\log t_I)^{1-\beta}
 \quad
 \text{for }
 t\in(t_I,\infty).
 \end{align*}
 Combining this with \eqref{e_5.14}, we obtain \eqref{e_5.17}.
\end{proof}

\section{Analytical Formulation of the Problem}
\label{section_6}
 In this section,
 we prepare the analytical framework for the proof of Theorem \ref{theorem_3}.
 Theorems \ref{theorem_1} and \ref{theorem_2} correspond to simpler cases,
 and their proofs are therefore omitted, as they are naturally included in the construction presented here.
 
 Instead of \eqref{equation_1.1},
 we consider
 \begin{align}
 \label{equation_6.1}
 \begin{cases}
 u_t=\Delta_xu+|u|^{p-1}u & \text{for } (x,t)\in\R^n\times(t_{I},\infty),
 \\
 u|_{t=t_{I}}=u_0(x)& \text{for } x\in\R^n
 \end{cases}
 \end{align}
 for some sufficiently large \( t_I > 0 \).
 In this section,
 we formulate our problem as in our previous work (see Section 5 of \cite{Harada_6D}).
 Let $\theta(x,t)$ denote the same function introduced in Section~\ref{section_5}
 and constructed in Section~\ref{section_4}.
 We recall that
 $T_1(y)$ is a solution of
 \begin{align*}
 H_yT_1(y)=-(\Lambda_y{\sf Q})(y)
 \end{align*}
 (see Section \ref{section_5.1}).
 For simplicity,
 we put
 \[
 {\sf b}(t)
 =
 \theta(0,t).
 \]

\subsection{Setting}
\label{section_6.1}
 Based on the argument in the previous section (see \eqref{e_5.1}),
 we look for solutions of the form:
 \begin{align}
 \label{equation_6.2}
 u(x,t)
 &=
 \lambda^{-\frac{n-2}{2}}
 {\sf Q}(y)
 \chi_1
 +
 \tfrac{5{\sf b}(t)}{4}
 T_1(y)
 \chi_1
 +
 \theta(x,t)
 \chi_1^c
 +
 v(x,t),
 \end{align}
 where $\lambda=\lambda(t)$ is an unknown function,
 and
 $\chi_1$ and $\chi_1^c$ are cut off functions defined by
 \begin{itemize}
 \item
 $\chi_1
 =
 \chi(\tfrac{|x|}{ \lambda(t)^\kappa(\sqrt t)^{1-\kappa} })$
 \quad
 \text{for some constant }
 $\kappa\in(0,1)$,

 \item
 $\chi_1^c
 =
 1-\chi_1$.
 \end{itemize}
 The constant $\kappa$ is determined in Lemma \ref{lemma_9.1} and Lemma \ref{lemma_9.4}.
 The derivatives of $u(x,t)$ are given by
 \begin{align*}
 u_t
 &=
 -
 \lambda^{-\frac{n}{2}}
 \dot\lambda
 (\Lambda_{y}{\sf Q})
 \chi_1
 +
 \underbrace{
 \lambda^{-\frac{n-2}{2}}
 {\sf Q}
 \dot\chi_1
 }_{=-g_1[\lambda]}
 +
 \underbrace{
 \tfrac{d}{dt}
 \{
 \tfrac{5{\sf b}T_1}{4}
 \chi_1
 \}
 }_{=-g_2[\lambda]}
 \\
 &\quad
 +
 \theta_t
 \chi_1^c
 +
 \underbrace{
 \theta
 (-\dot\chi_1)
 }_{=-g_3[\lambda]}
 +
 v_t
 \\
 &=
 -
 \lambda^{-\frac{n}{2}}
 \dot\lambda
 (\Lambda_y{\sf Q})
 \chi_1
 +
 \theta_t
 \chi_1^c
 +
 v_t
 +
 \sum_{i=1}^3
 (-g_i[\lambda])
 \end{align*}
 and
 \begin{align*}
 \Delta_xu
 &=
 \lambda^{-\frac{n+2}{2}}
 (\Delta_{y}{\sf Q})
 \chi_1
 +
 \underbrace{
 \lambda^{-\frac{n-2}{2}}
 \tfrac{2(\nabla_{y}{\sf Q}\cdot\nabla_{y}\chi_1)+{\sf Q}(\Delta_{y}\chi_1)}{\lambda^2}
 }_{=g_{4}[\lambda]}
 \\
 &\quad
 +
 \tfrac{(\Delta_{y}\frac{5{\sf b} T_1}{4})}{\lambda^2}
 \chi_1
 +
 \underbrace{
 \tfrac{2\nabla_{y}\frac{5{\sf b} T_1}{4}\cdot\nabla_{y}\chi_1}{\lambda^2}
 }_{=g_{5}[\lambda]}
 +
 \underbrace{
 (
 \tfrac{5{\sf b} T_1}{4}
 -
 \theta
 )
 \tfrac{(\Delta_{y}\chi_1)}{\lambda^2}
 }_{=g_{6}[\lambda]}
 \\
 &\quad
 +
 (\Delta_x\theta)
 \chi_1^c
 +
 \underbrace{
 \tfrac{2\nabla_x\theta\cdot(-\nabla_{y}\chi_1)}{\lambda}
 }_{=g_{7}[\lambda]}
 +
 \Delta_xv
 \\
 &=
 \lambda^{-\frac{n+2}{2}}
 (\Delta_{y}{\sf Q})
 \chi_1
 +
 \tfrac{(\Delta_{y}\frac{5{\sf b} T_1}{4})}{\lambda^2}
 \chi_1
 \\
 &\quad
 +
 (\Delta_x\theta)
 \chi_1^c
 +
 \Delta_xv
 +
 \sum_{i=4}^{7}
 g_i[\lambda].
 \end{align*}
 The nonlinearity can be written as:
 \begin{align}
 \nonumber
 &
 f(u)
 =
 f
 (
 \lambda^{-\frac{n-2}{2}}
 {\sf Q}
 \chi_1
 +
 \tfrac{5{\sf b} T_1}{4}
 \chi_1
 +
 \theta
 \chi_1^c
 +
 v
 )
 \\
 \label{equation_6.3}
 &=
 f(
 \lambda^{-\frac{n-2}{2}}
 {\sf Q}
 \chi_1
 )
 +
 f'(
 \lambda^{-\frac{n-2}{2}}
 {\sf Q}
 \chi_1
 )
 \cdot
 (
 u-\lambda^{-\frac{n-2}{2}}
 {\sf Q}
 \chi_1
 )
 \\
 \nonumber
 &\quad
 +
 f(\theta\chi_1^c)
 +
 f'(\theta\chi_1^c)
 \cdot
 (
 u-\theta\chi_1^c
 )
 +
 N.
 \end{align}
 Here,
 $N$ is defined by:
 \begin{align}
 \label{equation_6.4}
 N
 &=
 f
 (
 \lambda^{-\frac{n-2}{2}}
 {\sf Q}
 \chi_1
 +
 \tfrac{5{\sf b} T_1}{4}
 \chi_1
 +
 \theta
 \chi_1^c
 +
 v
 )
 -
 f(
 \lambda^{-\frac{n-2}{2}}
 {\sf Q}
 \chi_1
 )
 \\
 \nonumber
 &
 -
 f'(
 \lambda^{-\frac{n-2}{2}}
 {\sf Q}
 \chi_1
 )
 \cdot
 (
 u-\lambda^{-\frac{n-2}{2}}{\sf Q}\chi_1
 )
 \\
 \nonumber
 &
 -
 f(\theta\chi_1^c)
 -
 f'(\theta\chi_1^c)
 \cdot
 (
 u-\theta\chi_1^c
 ).
 \end{align}
 Note that
 \begin{align*}
 f(u)=|u|u,
 \quad
 f'(u)=2|u|
 \qquad
 (n=6).
 \end{align*}
 Using the expressions,
 we further simplify the terms on the right-hand side of \eqref{equation_6.3}.
 \begin{align*}
 f(
 \lambda^{-\frac{n-2}{2}}
 {\sf Q}
 \chi_1
 )
 &=
 \lambda^{-\frac{n+2}{2}}
 f({\sf Q})
 \chi_1^2
 \\
 &=
 \lambda^{-\frac{n+2}{2}}
 f({\sf Q})
 \chi_1
 +
 \underbrace{
 \lambda^{-\frac{n+2}{2}}
 f({\sf Q})
 (\chi_1^2-\chi_1)
 }_{=g_{8}[\lambda]},
 \\
 f(\theta\chi_1^c)
 &=
 \underbrace{
 f(\theta)
 (1-\chi_1)^2
 }_{=g_{9}[\lambda]},
 \\
 f'(
 \lambda^{-\frac{n-2}{2}}
 {\sf Q}
 \chi_1
 )
 &\cdot
 (
 u
 -
 \lambda^{-\frac{n-2}{2}}
 {\sf Q}
 \chi_1
 )
 \\
 &=
 \tfrac{f'({\sf Q})\chi_1}{\lambda^2}
 \cdot
 (
 \tfrac{5{\sf b} T_1}{4}
 +
 \tfrac{5{\sf b} T_1}{4}
 (
 \chi_1-1
 )
 +
 \theta
 \chi_1^c
 +
 v
 )
 \\
 &=
 \tfrac{f'({\sf Q})}{\lambda^2}
 \tfrac{5{\sf b} T_1}{4}
 \chi_1
 +
 \underbrace{
 \tfrac{f'({\sf Q})}{\lambda^2}
 (\tfrac{5{\sf b} T_1}{4}-\theta)
 \chi_1(\chi_1-1)
 }_{=g_{10}[\lambda]}
 +
 \tfrac{f'({\sf Q})}{\lambda^2}
 v
 \chi_1
 \\
 &=
 \tfrac{f'({\sf Q})}{\lambda^2}
 \tfrac{5{\sf b} T_1}{4}
 \chi_1
 +
 \tfrac{f'({\sf Q})}{\lambda^2}
 v
 \chi_1
 +
 g_{10}[\lambda],
 \end{align*}
 \begin{align*}
 &
 f'(\theta\chi_1^c)
 \cdot
 (u-\theta\chi_1^c)
 \\
 &=
 f'(\theta)
 \chi_1^c
 \cdot
 (
 \lambda^{-\frac{n-2}{2}}
 {\sf Q}
 \chi_1
 +
 \tfrac{5{\sf b} T_1}{4}
 \chi_1
 +
 v
 )
 \\
 &=
 \underbrace{
 f'(\theta)
 \lambda^{-\frac{n-2}{2}}
 {\sf Q}
 \chi_1(1-\chi_1)
 }_{=g_{11}[\lambda]}
 +
 \underbrace{
 f'(\theta)
 \tfrac{5{\sf b} T_1}{4}
 \chi_1(1-\chi_1)
 }_{=g_{12}[\lambda]}
 +
 f'(\theta)
 v
 \chi_1^c
 \\
 &=
 f'(\theta)
 v
 \chi_1^c
 +
 \sum_{i=11}^{12}
 g_i[\lambda].
 \end{align*}
 From the above relations,
 it follows that
 \begin{align}
 \nonumber
 v_t
 &=
 \lambda^{-\frac{n}{2}}
 \dot\lambda
 (\Lambda_{y}{\sf Q})
 \chi_1
 -
 \theta_t
 \chi_1^c
 +
 \lambda^{-\frac{n+2}{2}}
 (\Delta_{y}{\sf Q})
 \chi_1
 \\
 \nonumber
 &\quad
 +
 \tfrac{(\Delta_{y}\tfrac{5{\sf b} T_1}{4})}{\lambda^2}
 \chi_1
 +
 (\Delta_x
 \theta)
 \chi_1^c
 +
 \Delta_xv
 \\
 \nonumber
 &\quad
 +
 \lambda^{-\frac{n+2}{2}}
 f(
 {\sf Q}
 )
 \chi_1
 +
 \tfrac{f'({\sf Q})}{\lambda^2}
 \tfrac{5{\sf b} T_1}{4}
 \chi_1
 +
 \tfrac{f'({\sf Q})}{\lambda^2}
 v
 \chi_1
 \\
 \nonumber
 &\quad
 +
 f'(\theta)
 v
 \chi_1^c
 +
 N
 +
 \sum_{i=1}^{12}g_i[\lambda]
 \\
 \nonumber
 &=
 \lambda^{-\frac{n}{2}}
 \dot\lambda
 (\Lambda_{y}{\sf Q})
 \chi_1
 +
 \tfrac{(H_{y}\tfrac{5{\sf b} T_1}{4})}{\lambda^2}
 \chi_1
 +
 \Delta_xv
 \\
 \nonumber
 &\quad
 +
 \tfrac{f'({\sf Q})}{\lambda^2}
 v
 \chi_1
 +
 f'(\theta)
 v
 \chi_1^c
 +
 N
 +
 \sum_{i=1}^{12}g_i[\lambda]
 \\
 \label{equation_6.5}
 &=
 \lambda^{-2}
 (
 \lambda^{-\frac{n-4}{2}}
 \dot\lambda
 -
 \tfrac{5{\sf b}}{4}
 )
 (\Lambda_{y}{\sf Q})
 \chi_1
 +
 \Delta_xv
 +
 \tfrac{f'({\sf Q})}{\lambda^2}
 v
 \chi_1
 \\
 \nonumber
 &\quad
 +
 f'(\theta)
 v
 \chi_1^c
 +
 N
 +
 \sum_{i=1}^{12}g_i[\lambda].
 \end{align}
 Define a cut of function $\chi_{\text{in}}$ by
 \begin{align*}
 \chi_{\text{in}}
 =
 \chi(\tfrac{4|x|}{R\lambda}).
 \end{align*}
 The constant $R$ will be chosen sufficiently large in Section \ref{section_7} - Section \ref{section_8}.
 We now rewrite $v(x,t)$ as
 \begin{align*}
 v(x,t)
 =
 \lambda(t)^{-\frac{n-2}{2}}
 \epsilon(y,t)
 \chi_{\text{in}}
 +
 w(x,t),
 \end{align*}
 where $\epsilon(y,t)$ and $w(x,t)$ are unknown functions.
 Note that
 \begin{align*}
 v_t
 &=
 \lambda^{-\frac{n-2}{2}}
 \epsilon_t
 \chi_{\text{in}}
 +
 \underbrace{
 \lambda^{-\frac{n}{2}}
 \dot
 \lambda
 (-\Lambda_{y}\epsilon)
 \chi_{\text{in}}
 }_{=-h_1[\lambda,\epsilon]}
 +
 \underbrace{
 \lambda^{-\frac{n-2}{2}}
 \epsilon
 \dot\chi_{\text{in}}
 }_{=-h_2[\lambda,\epsilon]}
 +
 w_t,
 \\
 \Delta_xv
 &=
 \lambda^{-\frac{n+2}{2}}
 (\Delta_{y}\epsilon)
 \chi_{\text{in}}
 +
 \underbrace{
 2\lambda^{-\frac{n+2}{2}}
 (\nabla_{y}\epsilon\cdot\nabla_{y}\chi_{\text{in}})
 }_{=h_3[\lambda,\epsilon]}
 \\
 &\quad
 +
 \underbrace{
 \lambda^{-\frac{n+2}{2}}
 \epsilon
 (\Delta_{y}\chi_{\text{in}})
 }_{=h_4[\lambda,\epsilon]}
 +
 \Delta_xw
 \end{align*}
 and
 \begin{align*}
 \tfrac{f'({\sf Q})}{\lambda^2}
 v
 \chi_1
 &=
 \tfrac{f'({\sf Q})}{\lambda^2}
 (
 \lambda^{-\frac{n-2}{2}}\epsilon\chi_{\text{in}}
 +
 w
 )
 \chi_1
 \\
 &=
 \lambda^{-\frac{n+2}{2}}
 f'({\sf Q})
 \epsilon
 \chi_{\text{in}}
 +
 \tfrac{f'({\sf Q})}{\lambda^2}
 w
 \chi_1,
 \\
 f'(\theta)
 v
 \chi_1^c
 &=
 f'(\theta)
 (
 \lambda^{-\frac{n-2}{2}}\epsilon\chi_{\text{in}}
 +
 w
 )
 \chi_1^c
 =
 \underbrace{
 f'(\theta)
 w
 (1-\chi_1)
 }_{=k_1[w]}.
 \end{align*}
 Plugging the above expressions into \eqref{equation_6.4},
 we arrive at
 \begin{align}
 \label{equation_6.6}
 &
 \lambda^{-\frac{n-2}{2}}
 \epsilon_t
 +
 w_t
 =
 \lambda^{-2}
 (
 \lambda^{-\frac{n-4}{2}}
 \dot\lambda
 -
 \tfrac{5{\sf b}}{4}
 )
 (\Lambda_{y}{\sf Q})
 \chi_1
 +
 \lambda^{-\frac{n+2}{2}}
 (H_y\epsilon)
 \chi_{\text{in}}
 \\
 \nonumber
 &\quad
 +
 \Delta_xw
 +
 \tfrac{f'({\sf Q})}{\lambda^2}
 w
 \chi_1
 +
 N
 +
 \underbrace{
 \sum_{i=1}^{12}g_i[\lambda]
 +
 \sum_{i=1}^{5}h_i[\lambda,\epsilon]
 +
 k_1[w]
 }_{=g[\lambda]+h[\lambda,\epsilon]+k[w]}.
 \end{align}
 We now introduce a parabolic system for $(\lambda(t),\epsilon(y,t),w(x,t))$.
 \begin{align}
 \label{equation_6.7}
 &
 \begin{cases}
 \lambda^2
 \epsilon_t
 =
 H_{y}
 \epsilon
 +
 \lambda^\frac{n-2}{2}
 (\lambda^{-\frac{n-4}{2}}
 \dot\lambda
 -
 \tfrac{5{\sf b}}{4}
 )
 (\Lambda_{y}{\sf Q})
 +
 \lambda^\frac{n-2}{2}
 f'({\sf Q})
 w
 \\
 \qquad
 \text{for } (y,t)\in B_{R}\times(t_{I},\infty),
 \\
 \epsilon
 =
 0
 \quad
 \text{for } (y,t)\in \pa B_{R}\times(t_{I},\infty),
 \\
 \epsilon|_{t=t_{I}}
 =
 \psi
 \quad
 \text{for } y\in B_{R},
 \end{cases}
 \\
 &
 \label{equation_6.8}
 \begin{cases}
 w_t
 =
 \Delta_xw
 +
 \lambda^{-2}
 (\lambda^{-\frac{n-4}{2}}
 \dot\lambda
 -
 \tfrac{5{\sf b}}{4}
 )
 (\Lambda_{y}{\sf Q})
 \cdot
 (\chi_1-\chi_{\text{in}})
 \\
 \qquad
 +
 \tfrac{f'({\sf Q})}{\lambda^2}
 w
 \cdot
 (\chi_1-\chi_{\text{in}})
 +
 N+g[\lambda]+h[\lambda,\epsilon]+k[w]
 \\
 \qquad
 \text{for } (x,t)\in\R^n\times(t_{I},\infty),
 \\
 w|_{t=t_{I}}=0
 \quad
 \text{for }
 x\in\R^n.
 \end{cases}
 \end{align}
 Once a solution $(\lambda(t),\epsilon(y,t),w(x,t))$ of \eqref{equation_6.7} - \eqref{equation_6.8} is obtained,
 a function $u(x,t)$ defined by
 \begin{align*}
 u(x,t)
 &=
 \{
 \lambda(t)^{-\frac{n-2}{2}}
 {\sf Q}(y)
 +
 \tfrac{5{\sf b}}{4}
 T_1(y)
 \}
 \chi_1
 +
 \theta(x,t)
 \chi_1^c
 \\
 \nonumber
 &\quad
 +
 \lambda(t)^{-\frac{n-2}{2}}
 \epsilon(y,t)
 \chi_\text{in}
 +
 w(x,t)
 \quad
 \text{with }
 y=\tfrac{x}{\lambda(t)}
 \end{align*}
 gives a solution of \eqref{equation_6.1}.
 Therefore,
 our problem is reduced to constructing a solution to the parabolic system
 \eqref{equation_6.7} - \eqref{equation_6.8}
 within an appropriate class.
 In equation \eqref{equation_6.7},
 although the boundary condition for $\epsilon(y,t)$ is not explicitly stated,
 we impose the Dirichlet zero boundary condition to ensure the space-time decay of $\epsilon(y,t)$
 (see Section \ref{section_7}).
 For clarity,
 we summarize the definitions of the error terms $g_i$, $h_i$ and $k_i$ below,
 which have appeared in the preceding computations.
 \begin{enumerate}[(g1)]
 \item 
 $g_1[\lambda]
 =
 \lambda^{-\frac{n-2}{2}}
 (-{\sf Q})
 \dot\chi_1$
 
 \item
 $g_2[\lambda]
 =
 \tfrac{d}{dt}
 \{
 \tfrac{-5{\sf b}T_1}{4}
 \chi_1
 \}$
 
 \item
 $g_3[\lambda]
 =
 \theta
 \dot\chi_1$
 
 \item
 $g_4[\lambda]
 =
 \lambda^{-\frac{n-2}{2}}
 \tfrac{2(\nabla_{y}{\sf Q}\cdot\nabla_{y}\chi_1)+{\sf Q}(\Delta_{y}\chi_1)}{\lambda^2}$

 \item
 $g_5[\lambda]
 =
 \tfrac{2\nabla_{y}\frac{5{\sf b} T_1}{4}\cdot\nabla_{y}\chi_1}{\lambda^2}$

 \item
 $g_6[\lambda]
 =
 (
 \tfrac{5{\sf b} T_1}{4}
 -
 \theta
 )
 \tfrac{(\Delta_{y}\chi_1)}{\lambda^2}$
 
 \item
 $g_7[\lambda]
 =
 \tfrac{2\nabla_x\theta\cdot(-\nabla_{y}\chi_1)}{\lambda}$

 \item
 $g_8[\lambda]
 =
 \lambda^{-\frac{n+2}{2}}
 f({\sf Q})
 (\chi_1^2-\chi_1)$
 
 \item
 $g_9[\lambda]
 =
 f(\theta)
 (1-\chi_1)^2$

 \item
 $g_{10}[\lambda]
 =
 \tfrac{f'({\sf Q})}{\lambda^2}
 (\tfrac{5{\sf b} T_1}{4}-\theta)
 \chi_1
 (\chi_1-1)$

 \item
 $g_{11}[\lambda]
 =
 f'(\theta)
 \lambda^{-\frac{n-2}{2}}
 {\sf Q}
 \chi_1(1-\chi_1)$
 
 \item
 $g_{12}[\lambda]
 =
 f'(\theta)
 \tfrac{5{\sf b} T_1}{4}
 \chi_1(1-\chi_1)$
 \end{enumerate}
 \begin{enumerate}[(h1)]
 \item
 $h_1[\lambda,\epsilon]
 =
 \lambda^{-\frac{n}{2}}
 \dot
 \lambda
 (\Lambda_{y}\epsilon)
 \chi_{\text{in}}$
 
 \item
 $h_2[\lambda,\epsilon]
 =
 \lambda^{-\frac{n-2}{2}}
 (-\epsilon)
 \dot\chi_{\text{in}}$
 
 \item
 $h_3[\lambda,\epsilon]
 =
 2\lambda^{-\frac{n+2}{2}}
 (\nabla_{y}\epsilon\cdot\nabla_{y}\chi_{\text{in}})$

 \item
 $h_4[\lambda,\epsilon]
 =
 \lambda^{-\frac{n+2}{2}}
 \epsilon
 (\Delta_{y}\chi_{\text{in}})$
 \end{enumerate}
 \begin{enumerate}[(k1)]
 \item
 $k_1[w]
 =
 f'(\theta)
 w
 (1-\chi_1)$
 \end{enumerate}

\subsection{Fixed Point Argument}
\label{section_6.2}
 In this subsection,
 we present a strategy for solving the coupled parabolic system \eqref{equation_6.7} - \eqref{equation_6.8}
 by applying a fixed point argument to the function \(w(x,t)\).
 First,
 we introduce several notations and preparatory definitions.
 Let $\beta\in(\frac{1}{2},1)$ be as in Theorem \ref{theorem_1} - \ref{theorem_3}
 (see Section \ref{section_4} for its role).
 We fix another parameter $\beta'$ such that
 \begin{align}
 \label{equation_6.9}
 \beta'
 \in
 (1,\beta+1),
 \end{align}
 and introduce an auxiliary function:
 \begin{align}
 \label{equation_6.10}
 {\cal W}(x,t)
 &=
 \begin{cases}
 \dis
 t^{-1}
 (\log t)^{-\beta'}
 & \text{for } |x|<\sqrt{t},
 \\
 |x|^{-2}
 (\log|x|^2)^{-\beta'}
 & \text{for } |x|> \sqrt{t}.
 \end{cases}
 \end{align}
 Let $\{t_j\}_{j\in\N}$, $\{R_j\}_{j\in\N}$, $\{t_j^+\}_{j\in\N}$ and $\{t_j^-\}_{j\in\N}$ be sequences given in Lemma \ref{lemma_5.1}.
 We rewrite them as
 \begin{itemize}
 \item 
 $t_j
 =
 e^{{\sf p}_j}$ with ${\sf p}_j={\sf n}_1^\frac{j}{1-\beta}$,
 \item
 $R_j=\sqrt{t_j}$,
 \item
 $t_j^+=(R_j\log R_j)^2$,
 \item
 $t_j^-=(\frac{R_j}{\log R_j})^2$.
 \end{itemize}
 Recall that $t_I$ is the initial time of our equation \eqref{equation_6.1}.
 We fix the initial time $t_I$ by setting
 \begin{align}
 \label{equation_6.11}
 t_I=t_1^-.
 \end{align}
 Let $\tau\in\R$ be a large positive constant.
 Denote by $X_\tau$ the subset of $C(\R^n\times[t_I,\tau])$ consisting of functions ${\sf w}(x,t)$ satisfying
 \begin{align*}
 |{\sf w}(x,t)|
 \leq
 {\cal W}(x,t)
 \quad
 &
 \text{for }
 (x,t)\in\R^n\times[t_I,\tau],
 \\
 {\sf w}(x,t_I)=0
 \quad
 &
 \text{for }
 x\in\R^n.
 \end{align*}
 We define the norm $\|\cdot\|$ on $C(\R^n\times[t_I,\tau])$ by
 \begin{align*}
 \|{\sf w}\|
 &=
 \sup_{(x,t)\in\R^n\times[t_I,\tau]}
 |{\sf w}(x,t)|.
 \end{align*}
 Clearly, the subset $X_\tau$ is closed and convex
 in the Banach space $(C(\R^n\times[t_I,\tau]),\|\cdot\|)$.
 Define the extension map
 \[
 {\sf w}(x,t)\in X_\tau\mapsto {\sf w}_{\text{ex}}(x,t)\in C(\R^n\times[t_I,\infty))
 \]
 by
 \begin{align}
 \label{equation_6.12}
 {\sf w}_{\text{ex}}(x,t)
 =
 \begin{cases}
 {\sf w}(x,t) & \text{for } x\in\R^n,\ t\in[t_I,\tau],
 \\
 \mathcal W(x,t) & \text{for } x\in\R^n,\ t>\tau\, \text{ and if } {\sf w}(x,\tau)>\mathcal W(x,t),
 \\
 {\sf w}(x,\tau) & \text{for } x\in\R^n,\ t>\tau\, \text{ and if } |{\sf w}(x,\tau)|<\mathcal W(x,t),
 \\
 -
 \mathcal W(x,t) & \text{for } x\in\R^n,\ t>\tau\, \text{ and if } {\sf w}(x,\tau)<-\mathcal W(x,t).
 \end{cases}
 \end{align}
 To formulate our problem \eqref{equation_6.7} - \eqref{equation_6.8} as a fixed point problem,
 we define the solution map $\mathcal{T}_\tau:X_\tau\to X_\tau$ as follows.
 For given ${\sf w}\in X_\tau$,
 we first solve
 \begin{align}
 \label{equation_6.13}
 &
 \begin{cases}
 \lambda^2
 \epsilon_t
 =
 H_{y}
 \epsilon
 +
 \lambda^\frac{n-2}{2}
 (
 \lambda^{-\frac{n-4}{2}}
 \dot\lambda
 -
 \tfrac{5{\sf b}}{4}
 )
 (\Lambda_{y}{\sf Q})
 +
 \lambda^\frac{n-2}{2}
 f'({\sf Q})
 {\sf w}_\text{ex}
 \\
 \quad
 \text{for } (y,t)\in B_{R}\times(t_{I},\infty),
 \\
 \epsilon
 =
 0
 \quad
 \text{for } (y,t)\in \pa B_{R}\times(t_{I},\infty),
 \\
 \epsilon|_{t=t_{I}}
 =
 \psi
 \quad
 \text{for } y\in B_{R}.
 \end{cases}
 \end{align}
 This equation is a slight modification of \eqref{equation_6.7},
 and it involves two unknown functions: $\lambda(t)$ and $\epsilon(y,t)$.
 In this step,
 we can determine $\lambda(t)$ and $\epsilon(y,t)$ from the given ${\sf w}_\text{ex}$
 so that $\epsilon(y,t)$ becomes small
 (see Section \ref{section_7}).
 We denote them by $\lambda_{[{\sf w}]}(t)$ and $\epsilon_{[{\sf w}]}(y,t)$.
 Using the pair $(\lambda_{[{\sf w}]}(t),\epsilon_{[{\sf w}]}(y,t))$,
 we next consider
 \begin{align}
 \label{equation_6.14}
 \begin{cases}
 w_t
 =
 \Delta_xw
 +
 \lambda_{[{\sf w}]}^{-2}
 (\lambda_{[{\sf w}]}^{-\frac{n-4}{2}}
 \dot\lambda_{[{\sf w}]}
 -
 \tfrac{5{\sf b}}{4}
 )
 (\Lambda_{y}{\sf Q})
 \cdot
 (\chi_1-\chi_{\text{in}})
 \\
 \qquad
 +
 \tfrac{f'({\sf Q})}{\lambda_{\sf a}^2}
 {\sf w}_\text{ex}
 \cdot
 (\chi_1-\chi_{\text{in}})
 +
 N[\lambda_{[{\sf w}]},\epsilon_{[{\sf w}]},{\sf w}_\text{ex}]
 +
 g[\lambda_{[{\sf w}]}]
 \\
 \qquad
 +
 h[\lambda_{[{\sf w}]},\epsilon_{[{\sf w}]}]
 +
 k[{\sf w}_\text{ex}]
 \\
 \quad
 \text{for } (x,t)\in\R^n\times(t_{I},\infty),
 \\
 w|_{t=t_{I}}=0
 \quad
 \text{for }
 x\in\R^n.
 \end{cases}
 \end{align}
 This corresponds to equation \eqref{equation_6.8}.
 Equation \eqref{equation_6.14} admits a unique solution,
 which is denoted by $\tilde w_{[{\sf w}]}(x,t)$.
 Throughout this procedure,
 we can define the solution map ${\cal T}_\tau:X_\tau\to C(\R^n\times[t_I,\infty))$ by
 \[
 {\cal T}_\tau
 {\sf w}
 =
 \tilde w_{[{\sf w}]}.
 \]
 To ensure that ${\cal T}_\tau$ has a fixed point in $X_\tau$,
 it suffices to prove the following three properties.
 \begin{enumerate}[(i)]
 \item ${\cal T}_\tau:X_\tau\to X_\tau$ is well defined,
 \item ${\cal T}_\tau:(X_\tau,\|\cdot\|)\to(X_\tau,\|\cdot\|)$ is continuous,
 \item ${\cal T}_\tau:(X_\tau,\|\cdot\|)\to(X_\tau,\|\cdot\|)$ is compact.
 \end{enumerate}
 We do not give precise proofs of (ii) - (iii),
 since those are standard.
 Indeed,
 the construction of ${\cal T}_\tau:{\sf w}(x,t) \mapsto \tilde w_{[{\sf w}]}(x,t)$
 is uniquely determined at each step, which ensures property (ii).
 Furthermore,
 we note that
 the H\"older continuity of solutions is obtained
 by the $L^\infty$ bound of solutions
 from standard parabolic estimates (see Theorem 6.29 in \cite{Lieberman} p. 131).
 The H\"older continuity of solutions immediately implies (iii).
 Therefore,
 for any $\tau>0$,
 there exists a fixed point ${\sf w}^{(\tau)}(x,t)\in X_\tau$ such that ${\cal T}_\tau {\sf w}^{(\tau)}={\sf w}^{(\tau)}$.
 We fix a sequence $\{\tau_j\}_{j\in\N}$ so that $\lim_{j\to\infty}\tau_j=\infty$,
 and denote ${\sf w}_j(x,t):={\sf w}^{(\tau_j)}(x,t)$.
 We finally take a limit $(\lambda_{[{\sf w}_j]}(t),\epsilon_{[{\sf w}_j]}(y,t),{\sf w}_j(x,t))
 \to(\lambda_{\infty}(t),\epsilon_{\infty}(y,t),w_{\infty}(x,t))$
 as $j\to\infty$.
 A triple in this limit $(\lambda_{\infty}(t),\epsilon_{\infty}(y,t),w_{\infty}(x,t))$
 provides the solution stated in Theorem \ref{theorem_1} - Theorem \ref{theorem_3}.
 In the remainder of this paper, we focus on establishing the property (i).

\section{Analysis of the Inner Solution}
\label{section_7}
 In this section,
 we construct a pair of solutions $(\lambda(t),\epsilon(y,t))$
 in the same procedure as in Harada \cite{Harada_double} (see also \cite{Harada_6D}).
 The original idea can be traced back to
 Carmen Cort\'azar - Manuel del Pino - Monica Musso (Proposition 7.1 of \cite{Cortazar-delPino-Musso})
 and
 Juan D\'avila - Manuel del Pino -  Juncheng Wei (Section 6 - Section 7 of \cite{Davile-delPino-Wei}).
 Throughout Section \ref{section_7},
 we fix the following notations and assumptions.
 \begin{itemize}
 \item 
 The function ${\sf w}(x,t)\in X_\tau\subset C(\R^n\times[t_I,\tau])$ is assumed to be a given function.
 \item 
 The function ${\sf w}_\text{ex}(x,t)\in C(\R^n\times[t_I,\infty))$ denotes the extension of ${\sf w}(x,t)$,
 as defined in Section \ref{section_6.2}.
 \item
 We choose a constant $R$ such that $R=\log \log t_I$.
 \item
 The symbol $C$ denotes a generic positive constant independent of $R$, and $t_I$.
 \item
 The constant $t_I$ will be chosen sufficiently large, independent of $R$ and $\tau$.
 \end{itemize}

 \subsection{Choice of $\lambda(t)$}
 \label{section_7.1}
 In this subsection,
 we determine $\lambda(t)$
 as explained in Section \ref{section_6.2}.
 Let
 $\{t_j\}_{j\in\N}$, $\{R_j\}_{j\in\N}$, $\{t_j^+\}_{j\in\N}$ and $\{t_j^-\}_{j\in\N}$
 be sequences defined in Section \ref{section_6.2}.
 Until now,
 $n$ has been used to describe the general setting.
 From this section, we consider
 the specific case $n=6$ to proceed with a concrete analysis.
 For convenience,
 we rewrite \eqref{equation_6.13}:
 \begin{align}
 \label{EQ_7.1}
 &
 \begin{cases}
 \lambda^2
 \epsilon_t
 =
 H_{y}
 \epsilon
 +
 \lambda^2
 (\tfrac{\dot\lambda}{\lambda}
 -
 \tfrac{5{\sf b}}{4}
 )
 (\Lambda_{y}{\sf Q})
 +
 \lambda^2
 f'({\sf Q})
 {\sf w}_\text{ex}(\lambda y,t)
 \\
 \qquad
 \text{for } (y,t)\in B_{R}\times(t_{I},\infty),
 \\
 \epsilon
 =
 0
 \quad
 \text{for } (y,t)\in \pa B_{R}\times(t_{I},\infty),
 \\
 \epsilon|_{t=t_{I}}
 =
 \epsilon_0(y)
 \quad
 \text{for } y\in B_{R}.
 \end{cases}
 \end{align}
 In Section \ref{section_7.2},
 the initial data $\epsilon_0(y)$ will be chosen to ensure that $\epsilon(y,t)$ decays sufficiently as
 $y\to\infty$ and $t\to\infty$.
 \begin{itemize}
 \item
 For simplicity,
 we denote the inner product on $L_{y}^2(B_{R})$ and $L_y^2(B_R)$ norm by
 $\langle\cdot,\cdot\rangle$ and $\|\cdot\|_{L_y^2}$, respectively.
 \end{itemize}
 Recall that
 $(\mu_j^{(R)},\psi_j^{(R)}(y))$ is the $j$-th eigenvalue and the associated eigenfunction of the linearized problem on $B_R$
 (see Section \ref{section_3.3}).
 We define $\lambda(t)$ as the unique solution of
 \begin{align}
 \label{EQ_7.2}
 \begin{cases}
 (
 \tfrac{\dot\lambda}{\lambda}
 -
 \tfrac{5{\sf b}}{4}
 )
 =
 -
 \tfrac{\langle
 f'({\sf Q})
 {\sf w}_\text{ex}(\lambda y,t),
 \psi_2^{(R)}
 \rangle}{
 \langle
 (\Lambda_{y}{\sf Q}),
 \psi_2^{(R)}
 \rangle}
 \quad
 \text{for } t\in(t_I,\infty),
 \\[2mm]
 \dis
 \lambda(t)
 =
 e^{\frac{5A_1}{4(1-\beta)}(\log t_I)^{1-\beta}}
 \quad
 \text{for }
 t=t_I.
 \end{cases}
 \end{align}
 Put
 \begin{align}
 \label{EQ_7.3}
 \varrho(t)
 =
 \log
 \left(
 \lambda(t)
 e^{-\frac{5A_1}{4(1-\beta)}(\log t_I)^{1-\beta}}
 e^{-\int_{t_I}^t\frac{5{\sf b}(\zeta)}{4}d\zeta}
 \right).
 \end{align}
 From \eqref{EQ_7.2},
 $\varrho(t)$ solves
 \begin{align}
 \label{EQ_7.4}
 \begin{cases}
 \tfrac{d}{dt}
 \varrho
 =
 -
 \tfrac{\langle f'({\sf Q})
 {\sf w}_\text{ex}(\tilde\lambda[\varrho]y,t),\psi_2^{(R)}\rangle}{\langle(\Lambda_{y}{\sf Q}),\psi_2^{(R)}\rangle}
 \quad
 \text{for } t\in(t_I,\infty),
 \\[2mm]
 \varrho(t)
 =
 0
 \quad
 \text{for }
 t=t_I,
 \end{cases}
 \end{align}
 where
 we use the notation
 \begin{align*}
 \tilde\lambda[\varrho]
 =
 e^{\frac{5A_1}{4(1-\beta)}(\log t_I)^{1-\beta}}
 e^{\int_{t_I}^t\frac{5{\sf b}(\zeta)}{4}d\zeta}
 e^{\varrho}.
 \end{align*}
 From \eqref{equation_6.12},
 we note that
 there exists $\beta'>1$ such that
 \begin{align*}
 |{\sf w}_\text{ex}(x,t)|
 &<
 {\cal W}(x,t)
 \quad
 \text{for }
 (x,t)\in\R^n\times(t_I,\infty),
 \\
 {\cal W}(x,t) 
 &<
 t^{-1}
 (\log t)^{-\beta'} 
 \quad
 \text{for }
 |x|<\sqrt t.
 \end{align*}
 Under the assumption that $\tilde\lambda[\varrho(t)] R < \sqrt{t}$ for $t \in (t_I, \infty)$,
 the estimate on ${\sf w}_\text{ex}(x,t)$ implies that the right-hand side of \eqref{EQ_7.4} is bounded by
 \begin{align*}
 |
 \tfrac{\langle f'({\sf Q})
 {\sf w}_\text{ex}(\tilde\lambda[\varrho]y,t),\psi_2^{(R)}\rangle}{\langle(\Lambda_{y}{\sf Q}),\psi_2^{(R)}\rangle}
 |
 &<
 Ct^{-1}(\log t)^{-\beta'}
 \langle
 f'({\sf Q}),|\psi_2^{(R)}|
 \rangle
 \\
 &<
 Ct^{-1}(\log t)^{-\beta'}
 \quad
 \text{for }t\in(t_I,\infty).
 \end{align*}
 Following the same approach as in Section 6.1 of Harada \cite{Harada_double}
 (see also Section 6.2 in Harada \cite{Harada_6D}),
 we can verify that for any given ${\sf w}(x,t)\in X_\tau$,
 there exists a unique solution $\varrho(t)$ of \eqref{EQ_7.4},
 which satisfies both
 \begin{align*}
 |\varrho(t)|
 &<
 C(\log t_I)^{1-\beta'}
 \quad
 \text{for } t\in[t_I,\infty)
 \text{ and}
 \\
 \tilde\lambda[\varrho(t)]R
 &<
 \sqrt{t}
 \quad
 \text{for } t\in[t_I,\infty).
 \end{align*}
 Recall that $\beta'>1$ from \eqref{equation_6.9}.
 Therefore
 from the definition of $\varrho(t)$ (see \eqref{EQ_7.3}),
 it follows that
 \begin{align}
 \label{EQ_7.5}
 &|
 \lambda(t)
 -
 e^{\frac{5A_1}{4(1-\beta)}(\log t_I)^{1-\beta}}
 e^{\int_{t_I}^{t}\frac{5{\sf b}(\zeta)}{4}d\zeta}
 |
 \\
 \nonumber
 &
 <
 C(\log t_I)^{1-\beta'}
  e^{\frac{5A_1}{4(1-\beta)}(\log t_I)^{1-\beta}}
 e^{\int_{t_I}^{t}\frac{5{\sf b}(\zeta)}{4}d\zeta}
 \\
 \nonumber
 &
 \text{for }
 t\in[t_I,\infty).
 \end{align}
 From the construction of $\lambda(t)$,
 the constant $C$ in \eqref{EQ_7.5} is independent of the choice of $t_I$.
 By choosing $t_I$ sufficiently large,
 we can assume that
 \begin{align}
 \label{EQ_7.6}
 &
 \tfrac{3}{4}
 e^{\frac{5A_1}{4(1-\beta)}(\log t_I)^{1-\beta}}
 e^{\int_{t_I}^{t}\frac{5{\sf b}(\zeta)}{4}d\zeta}
 < 
 \lambda(t)
 \\
 \nonumber
 &\quad
 <
 \tfrac{5}{4}
 e^{\frac{5A_1}{4(1-\beta)}(\log t_I)^{1-\beta}}
 e^{\int_{t_I}^{t}\frac{5{\sf b}(\zeta)}{4}d\zeta}
 \quad
 \text{for }
 t\in[t_I,\infty).
 \end{align}
 Furthermore,
 we can verify that the function $\lambda(t)$,
 constructed above,
 satisfies all the assumptions of Lemma \ref{lemma_5.1}.
 Therefore,
 by  Lemma \ref{lemma_5.1},
 we have
 \begin{align}
 \label{EQ_7.7}
 \log[\tfrac{\lambda(t_j^-)}{\lambda(t_1^-)}]
 &<
 -
 \tfrac{q_1}{2}
 {\sf n}_1^j
 +
 6q_1
 {\sf n}_1^{j-1}
 +
 \tfrac{C_1}{\beta'-1}
 (\log t_1^-)^{-(\beta'-1)}
 \\
 \nonumber
 &
 \text{ \rm for even } j\in\N,
 \\
 \label{EQ_7.8}
 \log[\tfrac{\lambda(t_j^-)}{\lambda(t_1^-)}]
 &>
 \tfrac{q_1}{2}
 {\sf n}_1^j
 -
 6q_1
 {\sf n}_1^{j-1}
 -
 \tfrac{C_1}{\beta'-1}
 (\log t_1^-)^{-(\beta'-1)}
 \\
 \nonumber
 &
 \text{ \rm for odd } j\in\N
 \end{align}
 and
 \begin{align}
 \label{EQ_7.9}
 \lambda(t)
 &<
 e^{-q_1(\log t_I)^{1-\beta}}
 e^{2q_1(\log t)^{1-\beta}}
 \quad
 \text{ \rm for }
 t\in[t_I,\infty),
 \end{align}
 where $q_1 = \frac{5A_1}{4(1 - \beta)}$, as defined in Section \ref{section_5.3}.

 \subsection{Construction of $\epsilon(y,t)$}
 \label{section_7.2}
 We now construct a solution $\epsilon(y,t)$ to \eqref{EQ_7.1}.
 Let $\lambda_{[{\sf w}]}(t)$ be the solution of \eqref{EQ_7.2} constructed in Section \ref{section_7.1}.
 From \eqref{EQ_7.2},
 it holds that
 \begin{align}
 \label{EQ_7.10}
 |
 \tfrac{\dot\lambda_{[{\sf w}]}}{\lambda_{[{\sf w}]}}
 -
 \tfrac{5{\sf b}(t)}{4}
 |
 &<
 C
 t^{-1}
 (\log t)^{-\beta'}
 \quad
 \text{for }
 t\in(t_I,\infty).
 \end{align}
 Note that ${\sf b}(t)=\theta(0,t)$ satisfies \eqref{equation_4.23} - \eqref{equation_4.24}.
 We choose ${\sf r}_0>0$ sufficiently large so that
 \begin{align}
 \nonumber
 -H_{y}
 |y|^{-(n-4)}
 &=
 -
 \{
 \Delta_{y}+f'({\sf Q}(y))
 \}
 |y|^{-(n-4)}
 \\
 \nonumber
 &=
 2(n-4)|y|^{-(n-2)}
 -
 f'({\sf Q}(y))
 |y|^{-(n-4)}
 \\
 \label{EQ_7.11}
 &>
 (n-4)
 |y|^{-(n-2)}
 \quad
 \text{for }
 |y|>{\sf r}_0.
 \end{align}
 We introduce some notation.
 \begin{itemize}
 \item 
 $\chi_{{\sf r}_0}=\chi(\frac{|y|}{{\sf r}_0})$,
 \item
 $B_{R\setminus{\sf r}_0}=B_R\setminus B_{{\sf r}_0}$,
 \item
 $\pa B_{R\setminus{\sf r}_0}=\pa B_R\cup\pa B_{{\sf r}_0}$.
 \end{itemize}
 Consider
 \begin{align}
 \label{EQ_7.12}
 &
 \begin{cases}
 \dis
 \lambda_{[{\sf w}]}^2
 \pa_t
 \epsilon_\text{{\sf e}1}
 =
 H_{y}
 \epsilon_\text{{\sf e}1}
 +
 \lambda_{[{\sf w}]}^2
 (
 \tfrac{\dot\lambda_{[{\sf w}]}}{\lambda_{[{\sf w}]}}
 -
 \tfrac{5{\sf b}}{4}
 )
 (\Lambda_{y}{\sf Q})
 &
 \text{for }
 y\in{B}_{R\setminus{\sf r}_0},\ t\in(t_I,\infty),
 \\
 \epsilon_\text{{\sf e}1}
 =0
 &
 \text{for } y\in\pa{B}_{R\setminus{\sf r}_0},\ t\in(t_I,\infty),
 \\
 \epsilon_\text{{\sf e}1}|_{t=t_I}
 =0
 &
 \text{for } y\in {B}_{R\setminus{\sf r}_0},
 \end{cases}
 \\
 \label{EQ_7.13}
 &
 \begin{cases}
 \dis
 \lambda_{[{\sf w}]}^2
 \pa_t
 \epsilon_\text{{\sf e}2}
 =
 H_{y}
 \epsilon_\text{{\sf e}2}
 +
 \lambda_{[{\sf w}]}^2
 f'({\sf Q})
 {\sf w}_\text{ex}(\lambda_{[{\sf w}]}y,t)
 &
 \text{for }
 y\in{B}_{R\setminus{\sf r}_0},\ t\in(t_I,\infty),
 \\
 \epsilon_\text{{\sf e}2}
 =0
 &
 \text{for }
 y\in\pa{B}_{R\setminus{\sf r}_0},\ t\in(t_I,\infty),
 \\
 \epsilon_\text{{\sf e}2}|_{t=t_I}
 =0
 &
 \text{for } y\in{B}_{R\setminus{\sf r}_0},
 \end{cases}
 \\
 \label{EQ_7.14}
 &
 \begin{cases}
 \dis
 \lambda_{[{\sf w}]}^2
 \pa_t
 \epsilon_{{\sf in}}
 =
 H_{y}
 \epsilon_{{\sf in}}
 +
 \lambda_{[{\sf w}]}^2
 (
 \tfrac{\dot\lambda_{[{\sf w}]}}{\lambda_{[{\sf w}]}}
 -
 \tfrac{5{\sf b}}{4}
 )
 (\Lambda_{y}{\sf Q})
 \chi_{{\sf r}_0}
 +
 \lambda_{[{\sf w}]}^2
 f'({\sf Q})
 {\sf w}_\text{ex}(\lambda_{[{\sf w}]}y,t)
 \chi_{{\sf r}_0}
 \\
 \hspace{15mm}
 +
 \underbrace{
 2\nabla_{y}
 (\epsilon_\text{{\sf e}1}+\epsilon_\text{{\sf e}2})
 \cdot
 (-\nabla_{y}\chi_{{\sf r}_0})
 +
 (\epsilon_\text{{\sf e}1}+\epsilon_\text{{\sf e}2})
 (-\Delta_{y}\chi_{{\sf r}_0})
 }_{=l[\epsilon_\text{{\sf e}1},\epsilon_\text{{\sf e}2}]}
 \\
 \quad
 \text{for } y\in B_R,\ t\in(t_I,\infty),
 \\
 \epsilon_\text{{\sf in}}
 =0
 \quad
 \text{for } y\in B_R,\ t\in(t_I,\infty),
 \\
 \epsilon_\text{{\sf in}}|_{t=t_I}
 =
 \epsilon_0(y)
 \quad
 \text{for } y\in B_{R}.
 \end{cases}
 \end{align}
 Once
 the triple $(\epsilon_{{\sf e}1}(y,t),\epsilon_{{\sf e}2}$ $(y,t),\epsilon_\text{\sf in}(y,t))$
 is constructed,
 a function $\epsilon(y,t)$ defined by
 \[
 \epsilon
 =
 \epsilon_{{\sf e}1}
 \cdot
 (1-\chi_{{\sf r}_0})
 +
 \epsilon_{{\sf e}2}
 \cdot
 (1-\chi_{{\sf r}_0})
 +
 \epsilon_{{\sf in}}
 \]
 provides a solution of \eqref{EQ_7.1}.
 Put
 $\bar\epsilon_{{\sf e}1}(y,t)
 =\bar C_1\lambda_{[{\sf w}]}(t)^2t^{-1}(\log t)^{-\beta'}|y|^{-2}$.
 From \eqref{EQ_7.10} - \eqref{EQ_7.11},
 we easily see that
 \begin{align*}
 \lambda_{[{\sf w}]}^2
 \pa_t
 \bar\epsilon_{{\sf e}1}
 -
 H_{y}
 \bar\epsilon_{{\sf e}1}
 &>
 \left(
 \lambda_{[{\sf w}]}^2
 (
 \tfrac{2\dot\lambda_{[{\sf w}]}}{\lambda_{[{\sf w}]}}
 -
 \tfrac{1}{t}
 -
 \tfrac{\beta'}{t(\log t)}
 )
 +
 \tfrac{2}{|y|^2}
 \right)
 \bar\epsilon_{{\sf e}1}
 \\
 &>
 \left(
 \lambda_{[{\sf w}]}^2
 (
 \tfrac{5{\sf b}}{2}
 -
 \tfrac{2C}{t(\log t)^{\beta'}}
 -
 \tfrac{1}{t}
 -
 \tfrac{\beta'}{t(\log t)}
 )
 +
 \tfrac{2}{|y|^2}
 \right)
 \bar\epsilon_{{\sf e}1}
 \\
 &
 \text{for }
 y\in{B}_{R\setminus{\sf r}_0},\
 t\in(t_I,\infty).
 \end{align*}
 Since ${\sf b}(t)=\theta(0,t)$ satisfies \eqref{equation_4.24},
 we deduce that
 \begin{align*}
 \lambda_{[{\sf w}]}^2
 \pa_t
 \bar\epsilon_{{\sf e}1}
 -
 H_{y}
 \bar\epsilon_{{\sf e}1}
 &>
 \left(
 -
 \tfrac{2\lambda_{[{\sf w}]}^2}{t}
 +
 \tfrac{2}{|y|^2}
 \right)
 \bar\epsilon_{{\sf e}1}
 \\
 &
 \text{for }
 y\in{B}_{R\setminus{\sf r}_0},\
 t\in(t_I,\infty),
 \end{align*}
 if $t_I$ is sufficiently large.
 Combining the facts that $\lambda_{[{\sf w}]}(t)<t^\frac{1}{8}$ (see \eqref{EQ_7.9}) and $R=\log\log t_I$,
 we obtain
 \begin{align*}
 \lambda_{[{\sf w}]}^2
 \pa_t
 \bar\epsilon_{{\sf e}1}
 -
 H_{y}
 \bar\epsilon_{{\sf e}1}
 &>
 \tfrac{\bar\epsilon_{{\sf e}1}}{|y|^2}
 =
 \tfrac{\bar{C}_1\lambda_{[{\sf w}]}^2}{t(\log t)^{\beta'}|y|^4}
 \\
 &
 \text{for }
 y\in {B}_{R\setminus{\sf r}_0},\
 t\in(t_I,\infty). 
 \end{align*}
 Furthermore,
 from \eqref{EQ_7.10},
 the right-hand side of \eqref{EQ_7.12} can be estimated as
 \begin{align*}
 |
 \lambda_{[{\sf w}]}^2
 (
 \tfrac{\dot\lambda_{[{\sf w}]}}{\lambda_{[{\sf w}]}}
 -
 \tfrac{5{\sf b}}{4}
 )
 (\Lambda_{y}{\sf Q})
 |
 &<
 \tfrac{C\lambda_{[{\sf w}]}^2}{t(\log t)^{\beta'}|y|^4}
 \quad
 \text{for }
 y\in {B}_{R\setminus{\sf r}_0},\
 t\in(t_I,\infty). 
 \end{align*}
 Hence,
 by choosing \(\bar C_1 > 0\) sufficiently large (independent of
 \(t_I\) and \(\tau\)), the function \(\bar\epsilon_{{\sf e}1}(y,t)\) can be used as a comparison function for \eqref{EQ_7.12}.
 Therefore,
 we have
 \begin{align}
 \label{EQ_7.15}
 |\epsilon_{{\sf e}1}(y,t)|
 &<
 \bar\epsilon_{{\sf e}1}(y,t)
 <
 \tfrac{\bar C_1\lambda_{[{\sf w}]}^2}{t(\log t)^{\beta'}|y|^2}
 \quad
 \text{for } y\in {B}_{R\setminus{\sf r}_0},\ t\in(t_I,\infty).
 \end{align}
 We next derive estimates for $\epsilon_{{\sf e}2}(y,t)$.
 From the definition of ${\sf w}\in X_\tau$,
 the second term on the right-hand side of \eqref{EQ_7.13}
 is bounded by
 \begin{align*}
 |\lambda_{[{\sf w}]}^2
 f'({\sf Q})
 {\sf w}_\text{ex}(\lambda_{[{\sf w}]}y,t)
 |
 &<
 \tfrac{2\lambda_{[{\sf w}]}^2{\sf Q}}{t(\log t)^{\beta'}}
 <
 \tfrac{C\lambda_{[{\sf w}]}^2}{t(\log t)^{\beta'}|y|^4}
 \\
 &
 \text{for } y\in {B}_{R\setminus{\sf r}_0},\ t\in(t_I,\infty).
 \end{align*}
 Therefore,
 in exactly the same way as in the case of $\epsilon_{{\sf e}1}(y,t)$,
 we obtain
 \begin{align}
 \label{EQ_7.16}
 |\epsilon_{{\sf e}2}(y,t)|
 &<
 \tfrac{
 \bar C_2
 \lambda_{[{\sf w}]}^2
 }{
 t(\log t)^{\beta'}
 |y|^2
 }
 \quad
 \text{for } y\in{B}_{R\setminus{\sf r}_0},\ t\in(t_I,\infty).
 \end{align}
 We introduce a new time variable $s$.
 \begin{align}
 \label{EQ_7.17}
 s
 :=
 s(t)
 =
 \int_{t_I}^t\tfrac{dt'}{\lambda_{[{\sf w}]}(t')^2}.
 \end{align}
 In terms of the new variable $s$,
 equation \eqref{EQ_7.12} becomes
 \begin{align}
 \label{EQ_7.18}
 &
 \begin{cases}
 \dis
 \pa_s
 \epsilon_\text{{\sf e}1}
 =
 H_{y}
 \epsilon_\text{{\sf e}1}
 +
 \lambda_{[{\sf w}]}^2
 (
 \tfrac{\dot\lambda_{[{\sf w}]}}{\lambda_{[{\sf w}]}}
 -
 \tfrac{5{\sf b}}{4}
 )
 (\Lambda_{y}{\sf Q})
 &
 \text{for }
 y\in{B}_{R\setminus{\sf r}_0},\ s\in(0,\infty),
 \\
 \epsilon_\text{{\sf e}1}
 =0
 &
 \text{for } y\in\pa{B}_{R\setminus{\sf r}_0},\ s\in(0,\infty),
 \\
 \epsilon_\text{{\sf e}1}|_{s=0}
 =0
 &
 \text{for } y\in {B}_{R\setminus{\sf r}_0}.
 \end{cases}
 \end{align}
 Let $\bar{\mathtt{t}}(s)$ be the inverse function of $s(t)$ defined in \eqref{EQ_7.17}.
 By standard gradient estimates for parabolic equations,
 together with \eqref{EQ_7.15},
 it holds that
 for $s\in(1,\infty)$
 \begin{align}
 \nonumber
 &
 \sup_{y\in B_{R\setminus{\sf r}_0}}
 |\nabla_y\epsilon_{{\sf e}1}(y,s)|
 \\
 \nonumber
 &<
 C
 \sup_{s'\in(s-1,s)}
 \|\bar\epsilon_{{\sf e}1}(s)\|_{L_y^\infty(B_{R\setminus{\sf r}_0})}
 \\
 \nonumber
 &\quad
 +
 C
 \sup_{s'\in(s-1,s)}
 \|
 \lambda_{[{\sf w}]}^2
 (
 \tfrac{\dot\lambda_{[{\sf w}]}}{\lambda_{[{\sf w}]}}
 -
 \tfrac{5{\sf b}}{4}
 )
 (\Lambda_{y}{\sf Q})
 \|_{L_y^\infty(B_{R\setminus{\sf r}_0})}
 \\
 \label{EQ_7.19}
 &<
 \bar C_1'
 \sup_{s'\in(s-1,s)}
 \tfrac{\lambda_{[{\sf w}]}(\bar{\mathtt{t}}(s'))^2}{\bar{\mathtt{t}}(s')(\log\bar{\mathtt{t}}(s'))^{\beta'}},
 \end{align}
 and for $s\in(0,1)$
 \begin{align}
 \label{EQ_7.20}
 \sup_{y\in B_{R\setminus{\sf r}_0}}
 |\nabla_y\epsilon_{{\sf e}1}(y,s)|
 &<
 \bar C_1'
 \sup_{s'\in(0,s)}
 \tfrac{\lambda_{[{\sf w}]}(\bar{\mathtt{t}}(s'))^2}{\bar{\mathtt{t}}(s')(\log\bar{\mathtt{t}}(s'))^{\beta'}}.
 \end{align}
 We now define $\lambda(s)$
 to simplify notation in the new variable $s$.
 \[
 \lambda(s) := \lambda_{[{\sf w}]}(\bar{\mathtt t}(s)).
 \]
 Then,
 we claim that
 \begin{align}
 \label{EQ_7.21}
 \log\tfrac{\lambda(s+\Delta s)}{\lambda(s)}
 &<
 \tfrac{C\Delta s}{t_I^\frac{3}{4}(\log t_I)^\beta}
 \quad
 \text{for }
 \Delta s\in(-s,\infty),
 \\
 \label{EQ_7.22}
 \log\tfrac{\bar{\mathtt t}(s+\Delta s)}{\bar{\mathtt t}(s)}
 &<
 \tfrac{C\Delta s}{t_I^\frac{3}{4}}
 \quad
 \text{for }
 \Delta s\in(-s,\infty).
 \end{align}
 From the definition of $s=s(t)$ (see \eqref{EQ_7.17}) and the estimate \eqref{EQ_7.10},
 we observe that
 for $\Delta s\in(-s,\infty)$
 \begin{align*}
 &
 \log\tfrac{\lambda(s+\Delta s)}{\lambda(s)}
 =
 \int_s^{s+\Delta s}
 \tfrac{d}{ds}
 \log\lambda(s)
 ds
 \\
 &=
 \int_s^{s+\Delta s}
 \tfrac{1}{\lambda(s)}
 \tfrac{d\bar{\mathtt t}}{ds}
 \tfrac{d\lambda}{dt}
 ds
 <
 C
 \int_s^{s+\Delta s}
 \lambda(s)
 |{\sf b}(\bar{\mathtt t}(s))|
 ds
 \\
 &<
 C
 \int_s^{s+\Delta s}
 \tfrac{\lambda(s)ds}{\bar{\mathtt t}(s)(\log(\bar{\mathtt t}(s)))^\beta},
 \end{align*}
 and for $\Delta s\in(-s,\infty)$
 \begin{align*}
 &
 \log \tfrac{\bar{\mathtt t}(s-\Delta s)}{\bar{\mathtt t}(s)}
 =
 \int_s^{s+\Delta s}
 \tfrac{d}{ds}
 \log \bar{\mathtt t}(s) ds
 \\
 &=
 \int_s^{s+\Delta s}
 \tfrac{1}{\bar{\mathtt t}(s)}
 \tfrac{d\bar{\mathtt t}}{ds}
 ds
 =
 \int_s^{s+\Delta s}
 \tfrac{\lambda(s)^2}{\bar{\mathtt t}(s)}
 ds.
 \end{align*}
 Note from \eqref{EQ_7.9} that $\lambda_{[{\sf w}]}(t)<Ct^{\frac{1}{8}}$
 for $t\in(t_I,\infty)$.
 Therefore,
 we have
 for $\Delta s\in(-s,\infty)$
 \begin{align*}
 \log\tfrac{\lambda(s-\Delta s)}{\lambda(s)}
 &<
 C
 \int_s^{s+\Delta s}
 \tfrac{ds}{\bar{\mathtt t}(s)^\frac{3}{4}(\log(\bar{\mathtt t}(s)))^\beta}
 <
 \tfrac{C\Delta s}{t_I^\frac{3}{4}(\log t_I)^\beta},
 \end{align*}
 and for $\Delta s\in(-s,\infty)$
 \begin{align*}
 &
 \log \tfrac{\bar{\mathtt t}(s-\Delta s)}{\bar{\mathtt t}(s)}
 <
 \tfrac{C\Delta s}{t_I^\frac{3}{4}}.
 \end{align*}
 This completes the proof of estimates \eqref{EQ_7.21} - \eqref{EQ_7.22}.
 Combining the estimates \eqref{EQ_7.19} - \eqref{EQ_7.22},
 we conclude
 \begin{align*}
 \sup_{y\in B_{R\setminus{\sf r}_0}}
 |\nabla_y\epsilon_{{\sf e}1}(y,s)|
 &<
 \bar C_1''
 \tfrac{\lambda_{[{\sf w}]}(\bar{\mathtt{t}}(s))^2}{\bar{\mathtt{t}}(s)(\log\bar{\mathtt{t}}(s))^{\beta'}}
 \quad
 \text{for }
 s\in(0,\infty).
 \end{align*}
 In terms of the original variable $t$,
 this estimate becomes:
 \begin{align}
 \label{EQ_7.23}
 \sup_{y\in B_{R\setminus{\sf r}_0}}
 |\nabla_y\epsilon_{{\sf e}1}(y,t)|
 &<
 \bar C_1''
 \tfrac{\lambda_{[{\sf w}]}(t)^2}{t(\log t)^{\beta'}}
 \quad
 \text{for }
 t\in(t_I,\infty).
 \end{align}
 Using the estimate \eqref{EQ_7.16} and repeating the same approach as in the case of 
 $\epsilon_{{\sf e}1}(y,t)$,
 we similarly obtain the gradient estimate for $\epsilon_{{\sf e}2}(y,t)$:
 \begin{align}
 \label{EQ_7.24}
 \sup_{y\in B_{R\setminus{\sf r}_0}}
 |\nabla_y\epsilon_{{\sf e}2}(y,t)|
 &<
 \bar C_2''
 \tfrac{\lambda_{[{\sf w}]}(t)^2}{t(\log t)^{\beta'}}
 \quad
 \text{for }
 t\in(t_I,\infty).
 \end{align}
 We finally construct $\epsilon_{{\sf in}}(y,t)$.
 Recall that $(\mu_k^{(R)},\psi_k^{(R)}(y))$ denotes the $k$-th eigenpair of the following eigenvalue problem:
 \begin{align*}
 \begin{cases}
 -H_y\psi
 =
 \mu\psi
 & \text{for } y\in B_{R},
 \\
 \psi=0
 & \text{for } y\in\pa B_{R},
 \\
 \psi \text{ is radial}.
 \end{cases}
 \end{align*}
 Here,
 $\mu_k^{(R)}$ is the $k$-th eigenvalue and $\psi_k^{(R)}$
 is the associated radially symmetric eigenfunction normalized as $\psi_k^{(R)}(0) = 1$ (see Section \ref{section_3.3}).
 The functions
 $\epsilon_{{\sf e}1}(y,t)$ and $\epsilon_{{\sf e}2}(y,t)$
 denote the solutions to \eqref{EQ_7.12} and \eqref{EQ_7.13},
 respectively,
 constructed above.
 We decompose $\epsilon_{\sf in}(y,t)$ into components in the directions of 
 $\psi_1^{(R)}$ and $\psi_2^{(R)}$, and a remainder orthogonal to both in $L_y^2(B_R)$:
 \begin{align*}
 \epsilon_{\sf in}(y,t)
 =
 \langle \epsilon_{\sf in}, \psi_1^{(R)} \rangle
 \tfrac{\psi_1^{(R)}(y)}{\|\psi_1^{(R)}\|_{L_y^2}^2}
 +
 \langle \epsilon_{\sf in}, \psi_2^{(R)} \rangle
 \tfrac{\psi_2^{(R)}(y)}{\|\psi_2^{(R)}\|_{L_y^2}^2}
 +
 \epsilon_{\sf in}^\perp(y,t).
 \end{align*}
 We first estimate the projection of $\epsilon_{\sf in}$ along $\psi_1^{(R)}$.
 We continue to use the time variable $s$, as defined in \eqref{EQ_7.17}.
 From \eqref{EQ_7.14},
 we have
 \begin{align*}
 &
 \tfrac{d}{ds}
 \langle
 \epsilon_{\sf in},\psi_1^{(R)}
 \rangle
 \\
 &=
 -
 \mu_1^{(R)}
 \langle
 \epsilon_{\sf in},\psi_1^{(R)}
 \rangle
 +
 \underbrace{
 \langle
 \lambda_{[{\sf w}]}^2
 (
 \tfrac{\dot\lambda_{[{\sf w}]}}{\lambda_{[{\sf w}]}}
 -
 \tfrac{5{\sf b}}{4}
 )
 (\Lambda_{y}{\sf Q})
 \chi_{{\sf r}_0},
 \psi_1^{(R)}
 \rangle
 }_{=\mathcal{R}_1(s)}
 \\
 &
 +
 \underbrace{
 \langle
 \lambda_{[{\sf w}]}^2
 f'({\sf Q})
 {\sf w}_\text{ex}(\lambda_{[{\sf w}]}y,t)
 \chi_{{\sf r}_0},
 \psi_1^{(R)}
 \rangle
 }_{=\mathcal{R}_2(s)}
 +
 \underbrace{
 \langle
 l[\epsilon_{{\sf e}1},\epsilon_{{\sf e}2}],
 \psi_1^{(R)}
 \rangle
 }_{=\mathcal{R}_3(s)}.
 \end{align*}
 For notational convenience, let
 \[
 \omega
 =
 \mu_1^{(R)}.
 \]
 Integrating the differential equation,
 we get
 \begin{align}
 \label{EQ_7.25}
 &
 e^{\omega s}
 \langle
 \epsilon_{\text{\sf in}}(s),\psi_1^{(R)}
 \rangle
 -
 \langle
 \epsilon_{\text{\sf in}}(s)|_{s=0},\psi_1^{(R)}
 \rangle
 \\
 \nonumber
 &=
 \int_0^s
 e^{\omega\zeta}
 \{
 \mathcal{R}_1(\zeta)+\mathcal{R}_2(\zeta)+\mathcal{R}_3(\zeta)
 \}
 d\zeta.
 \end{align}
 We now take the initial data $\epsilon_{{\sf in}}|_{t=t_I}=\epsilon_0(y)$ such that
 \begin{align}
 \label{EQ_7.26}
 \epsilon_0(y)
 &=
 \alpha_I
 \psi_1^{(R)}(y)
 \end{align}
 with
 \begin{align}
 \label{EQ_7.27}
 \alpha_I
 =
 -
 \int_0^\infty
 e^{\omega\zeta}
 (\mathcal{R}_1+\mathcal{R}_2+\mathcal{R}_3)
 d\zeta.
 \end{align}
 The integral in \eqref{EQ_7.27} is finite since $\omega=\mu_1^{(R)}<0$.
 Combining \eqref{EQ_7.25} - \eqref{EQ_7.27},
 we observe that
 \begin{align}
 \label{EQ_7.28}
 \langle
 \epsilon_{{\sf in}}(s),\psi_1^{(R)}
 \rangle
 =
 -
 e^{-\omega s}
 \int_s^\infty
 e^{\omega\zeta}
 (
 \mathcal{R}_1+\mathcal{R}_2+\mathcal{R}_3
 )
 d\zeta.
 \end{align}
 From \eqref{EQ_7.10},
 together with the above estimates for $\epsilon_{{\sf e}1}(y,t)$ and $\epsilon_{{\sf e}2}(y,t)$,
 and Lemma \ref{lemma_3.2},
 we can verify that there exists a constant $C>0$ such that
 \begin{align}
 \label{EQ_7.29}
 |\mathcal{R}_1(\zeta)|+|\mathcal{R}_2(\zeta)|+|\mathcal{R}_3(\zeta)|
 &<
 \tfrac{C\lambda_{[{\sf w}]}(\bar{\mathtt{t}}(\zeta))^2}{\bar{\mathtt t}(\zeta)(\log\bar{\mathtt t}(\zeta))^{\beta'}}
 \quad
 \text{for }
 \zeta\in(0,\infty).
 \end{align}
 Note that $\bar{\mathtt t}(s)|_{s=0}=t_I$.
 Since $\omega=\mu_1^{(R)}<-\frac{e_0}{2}$, where $e_0>0$ (see Lemma \ref{lemma_3.1}),
 applying integration by parts,
 we see that
 \begin{align*}
 \int_0^\infty
 \tfrac{
 e^{\omega\zeta}\lambda_{[{\sf w}]}(\bar{\mathtt t}(\zeta))^2d\zeta
 }{
 \bar{\mathtt t}(\zeta)(\log\bar{\mathtt t}(\zeta))^{\beta'}}
 =
 \tfrac{
 \lambda_{[{\sf w}]}(t_I)^2
 }{
 (-\omega)t_I(\log t_I)^{\beta'}}
 -
 \int_0^\infty
 \tfrac{e^{\omega\zeta}}{\omega}
 \tfrac{d}{d\zeta}
 (
 \tfrac{
 \lambda_{[{\sf w}]}(\bar{\mathtt t}(\zeta))^2
 }{
 \bar{\mathtt t}(\zeta)(\log\bar{\mathtt t}(\zeta))^{\beta'}}
 )
 d\zeta.
 \end{align*}
 Using \eqref{EQ_7.10} and \eqref{EQ_7.17}, we compute the last term as follows:
 \begin{align}
 \label{EQ_7.30}
 &
 \tfrac{d}{d\zeta}
 (
 \tfrac{
 \lambda_{[{\sf w}]}(\bar{\mathtt t})^2
 }{
 \bar{\mathtt t}(\log\bar{\mathtt t})^{\beta'}}
 )
 =
 \tfrac{dt}{d\zeta}
 \cdot
 \tfrac{d}{dt}
 (
 \tfrac{
 \lambda_{[{\sf w}]}(t)^2
 }{
 t(\log t)^{\beta'}}
 )
 |_{t=\bar{\mathtt t}}
 \\
 \nonumber
 &=
 \lambda_{[{\sf w}]}^2
 \cdot
 (
 \tfrac{2\dot\lambda_{[{\sf w}]}}{\lambda_{[{\sf w}]}}
 -
 \tfrac{1}{\bar{\mathtt t}}
 -
 \tfrac{\beta'}{\bar{\mathtt t}(\log \bar{\mathtt t})}
 )
 \tfrac{
 \lambda_{[{\sf w}]}^2
 }{
 \bar{\mathtt t}(\log \bar{\mathtt t})^{\beta'}}
 \\
 \nonumber
 &=
 \lambda_{[{\sf w}]}^2
 \cdot
 (
 \tfrac{5{\sf b}}{4}
 +
 \mathtt h
 -
 \tfrac{1}{\bar{\mathtt t}}
 -
 \tfrac{\beta'}{\bar{\mathtt t}(\log\bar{\mathtt t})}
 )
 \tfrac{
 \lambda_{[{\sf w}]}^2
 }{\bar{\mathtt t}(\log\bar{\mathtt t})^{\beta'}},
 \end{align}
 where
 \begin{itemize}
 \item 
 $|{\mathtt h}|<\tfrac{C}{ \bar{\mathtt t}(\log\bar{\mathtt t})^{\beta'} }$,
 \item
 $|{\sf b}(t)|<\tfrac{2A_1}{t(\log t)^\beta}$ (see \eqref{equation_4.24}) and
 \item
 $\lambda_{[{\sf w}]}(t)<
 \tfrac{5}{4}
 e^{\frac{5A_1}{4(1-\beta)}(\log t_I)^{1-\beta}}
 e^{\int_{t_I}^{t}\frac{5{\sf b}(\zeta)}{4}d\zeta}$
 (see \eqref{EQ_7.6}).
 \end{itemize}
 Therefore,
 since
 $\lim_{R\to\infty}\omega=-e_0$ (see Lemma \ref{lemma_3.1}),
 we can choose $t_I$ sufficiently large so that
 \begin{align}
 \label{EQ_7.31}
 \int_0^\infty
 \tfrac{
 e^{\omega\zeta}\lambda_{[{\sf w}]}(\bar{\mathtt t}(\zeta))^2d\zeta
 }{
 \bar{\mathtt t}(\zeta)(\log\bar{\mathtt t}(\zeta))^{\beta'}}
 &<
 \tfrac{
 4\lambda_{[{\sf w}]}(t_I)^2
 }{
 e_0t_I(\log t_I)^{\beta'}}.
 \end{align}
 As a consequence,
 from \eqref{EQ_7.27},
 we obtain a bound for $\alpha_I$.
 \begin{align}
 \label{EQ_7.32}
 |\alpha_I|
 &<
 \tfrac{
 C\lambda_{[{\sf w}]}(t_I)^2
 }{
 e_0t_I(\log t_I)^{\beta'}}.
 \end{align}
 In exactly the same manner as in \eqref{EQ_7.31},
 we can verify that
 \[
 \int_s^\infty
 \tfrac{
 e^{\omega\zeta}\lambda_{[{\sf w}]}(\bar{\mathtt t}(\zeta))^2d\zeta
 }{
 \bar{\mathtt t}(\zeta)(\log\bar{\mathtt t}(\zeta))^{\beta'}}
 <
 \tfrac{
 4e^{\omega s}\lambda_{[{\sf w}]}(\bar{\mathtt t}(s))^2
 }{
 e_0\bar{\mathtt t}(s)(\log \bar{\mathtt t}(s))^{\beta'}}
 \quad
 \text{for }
 s\in(0,\infty).
 \]
 Therefore,
 from \eqref{EQ_7.28} - \eqref{EQ_7.29},
 we obtain
 \begin{align*}
 \nonumber
 |\langle\epsilon_{{\sf in}}(s),\psi_1^{(R)}\rangle|
 &=
 |
 -
 e^{-\omega s}
 \int_s^\infty
 e^{\omega\zeta}
 (\mathcal{R}_1+\mathcal{R}_2+\mathcal{R}_3)
 d\zeta
 |
 \\
 \nonumber
 &
 <
 e^{-\omega s}
 \int_s^\infty
 \tfrac{
 e^{\omega\zeta}\lambda_{[{\sf w}]}(\bar{\mathtt t})^2d\zeta
 }{
 \bar{\mathtt t}(\log\bar{\mathtt t})^{\beta'}}
 d\zeta
 \\
 &
 <
 \tfrac{
 \lambda_{[{\sf w}]}(\bar{\mathtt t}(s))^2
 }{\bar{\mathtt t}(s)(\log\bar{\mathtt t}(s))^{\beta'}}
 \quad
 \text{for }
 s\in(0,\infty).
 \end{align*}
 This gives
 \begin{align}
 \label{EQ_7.33}
 |\langle\epsilon_{{\sf in}}(t),\psi_1^{(R)}\rangle|
 <
 \tfrac{\lambda_{[{\sf w}]}(t)^2
 }{t(\log t)^{\beta'}}
 \quad
 \text{for }
 t\in(t_I,\infty).
 \end{align}
 We next consider the projection onto the direction of $\psi_2^{(R)}$.
 From the defining condition for $\lambda_{[{\sf w}]}(t)$
 (see Section \ref{section_7.1}),
 we recall that
 \[
 \langle\epsilon(t),\psi_{2}^{(R)}\rangle=0
 \quad
 \text{for }
 t\in[t_I,\infty).
 \]
 Since $\epsilon(y,t)$ is given by
 $\epsilon
 =
 \epsilon_{{\sf e}1}
 \cdot
 (1-\chi_{{\sf r}_0})
 +
 \epsilon_{{\sf e}2}
 \cdot
 (1-\chi_{{\sf r}_0})
 +
 \epsilon_{{\sf in}}$,
 it follows from \eqref{EQ_7.15} - \eqref{EQ_7.16} that
 \begin{align*}
 |\langle\epsilon_{\sf in},\psi_2^{(R)}\rangle|
 &=
 |
 \langle
 \epsilon_{{\sf e}1}\cdot(1-\chi_{{\sf r}_0})+\epsilon_{{\sf e}2}\cdot(1-\chi_{{\sf r}_0}),\psi_2^{(R)}
 \rangle
 |
 \\
 &<
 \tfrac{C\lambda_{[{\sf w}]}(t)^2}{t(\log t)^{\beta'}}
 \int_{y\in B_{R\setminus{\sf r}_0}}
 |y|^{-2}
 |\psi_2^{(R)}(y)|
 dy.
 \end{align*}
 Recall that, when $n=6$,
 we have $|\psi_2^{(R)}(y)|<c(1+|y|)^{-4}$ for $y\in B_{R}$
 (see Lemma \ref{lemma_3.2}).
 Therefore,
 we have
 \begin{align}
 \label{EQ_7.34}
 |\langle \epsilon_{{\sf in}},\psi_2^{(R)} \rangle|
 &<
 \tfrac{C\lambda_{[{\sf w}]}(t)^2}{t(\log t)^{\beta'}}
 \int_{y\in B_{R\setminus{\sf r}_0}}
 |y|^{-6}
 dy
 <
 \tfrac{C\lambda_{[{\sf w}]}(t)^2\log R}{t(\log t)^{\beta'}}
 \\
 \nonumber
 &
 \text{for }
 t\in[t_I,\infty).
 \end{align}
 This provides a bound on the projection onto $\psi_2^{(R)}$.
 Finally,
 we turn to the estimate of the orthogonal component $\epsilon_{\sf in}^\perp$.
 Let
 $\omega_3
 :=
 \mu_3^{(R)}$.
 We multiply \eqref{EQ_7.14} by $\epsilon_{{\sf in}}^\perp$ and integrate by parts.
 Then we have
 \begin{align*}
 \nonumber
 &
 \tfrac{\lambda_{[{\sf w}]}^2}{2}
 \tfrac{d}{dt}
 \|\epsilon_{{\sf in}}^{\perp}\|_{L_{y}^2}^2
 \\
 &=
 \underbrace{
 \langle
 \epsilon_{{\sf in}}^{\perp},H_y\epsilon_{{\sf in}}^{\perp}
 \rangle
 }_{<-\omega_3\|\epsilon_{{\sf in}}^{\perp}\|_{L_{y}^2}^2}
 +
 \underbrace{
 \langle
 \lambda_{[{\sf w}]}^2
 (
 \tfrac{\dot\lambda_{[{\sf w}]}}{\lambda_{[{\sf w}]}}
 -
 \tfrac{5{\sf b}}{4}
 )
 (\Lambda_{y}{\sf Q})
 \chi_{{\sf r}_0},
 \epsilon_{{\sf in}}^{\perp}
 \rangle
 }_{=\mathcal{U}_1(t)}
 \\
 &\quad
 +
 \underbrace{
 \langle
 \lambda_{[{\sf w}]}^2
 f'({\sf Q})
 {\sf w}_\text{ex}(\lambda_{[{\sf w}]}y,t)
 \chi_{{\sf r}_0},
 \epsilon_{{\sf in}}^{\perp}
 \rangle
 }_{=\mathcal{U}_2(t)}
 +
 \underbrace{
 \langle
 l[\epsilon_{{\sf e}1},\epsilon_{{\sf e}2}],
 \epsilon_{{\sf in}}^{\perp}
 \rangle
 }_{=\mathcal{U}_3(t)}.
 \end{align*}
 Following the same argument as in \eqref{EQ_7.29},
 we see that
 \begin{align*}
 |\mathcal{U}_1|
 +
 |\mathcal{U}_2|
 +
 |\mathcal{U}_3|
 &<
 \tfrac{C\lambda_{[{\sf w}]}^2}{t(\log t)^{\beta'}}
 \|\epsilon_{{\sf in}}^\perp\|_{L_{y}^2}
 \\
 &<
 \tfrac{C}{2\omega_3}
 (
 \tfrac{\lambda_{[{\sf w}]}^2}{t(\log t)^{\beta'}}
 )^2
 +
 \tfrac{\omega_3}{2}
 \|\epsilon_{{\sf in}}^\perp\|_{L_{y}^2}^2
 \quad
 \text{for }
 t\in[t_I,\infty).
 \end{align*}
 Therefore,
 integrating the equation over time $s\in(0,s)$,
 we obtain
 \begin{align}
 \label{EQ_7.35}
 &e^{\omega_3 s}
 \|\epsilon_{{\sf in}}^{\perp}(s)\|_{L_{y}^2}^2
 -
 \underbrace{
 \|\epsilon_{{\sf in}}^{\perp}(s)|_{s=0}\|_{L_{y}^2}^2
 }_{=0}
 <
 \int_0^s
 e^{\omega_3\zeta}
 \tfrac{C}{\omega_3}
 (
 \tfrac{\lambda_{[{\sf w}]}(\bar{\mathtt t})^2}{\bar{\mathtt t}(\log\bar{\mathtt t})^{\beta'}}
 )^2
 d\zeta.
 \end{align}
 Since we took the initial data as
 $\epsilon_{{\sf in}}|_{s=0}=\epsilon_0(y)=\alpha_I\psi_1^{(R)}(y)$ (see \eqref{EQ_7.26}),
 it is clear that
 $\epsilon_{{\sf in}}^\perp|_{s=0}=0$.
 We now compute the integral on the right-hand side of \eqref{EQ_7.35}.
 Note from \eqref{EQ_7.30} that
 \begin{align*}
 \tfrac{d}{d\zeta}
 (
 \tfrac{
 \lambda_{[{\sf w}]}(\bar{\mathtt t})^2
 }{
 \bar{\mathtt t}(\log\bar{\mathtt t})^{\beta'}}
 )^2
 &=
 2
 \lambda_{[{\sf w}]}(\bar{\mathtt t})^2
 \cdot
 (
 \tfrac{5|{\sf b}|}{4}
 +
 {\mathtt h}
 -
 \tfrac{1}{\bar{\mathtt t}}
 -
 \tfrac{\beta'}{\bar{\mathtt t}(\log\bar{\mathtt t})}
 )
 \cdot
 (
 \tfrac{
 \lambda_{[{\sf w}]}(\bar{\mathtt t})^2
 }{\bar{\mathtt t}(\log\bar{\mathtt t})^{\beta'}}
 )^2
 \\
 &<
 2
 \lambda_{[{\sf w}]}(\bar{\mathtt t})^2
 \cdot
 (2{\sf b})
 \cdot
 (
 \tfrac{
 \lambda_{[{\sf w}]}(\bar{\mathtt t})^2
 }{\bar{\mathtt t}(\log\bar{\mathtt t})^{\beta'}}
 )^2
 \\
 &<
 \tfrac{8A_1\lambda_{[{\sf w}]}(\bar{\mathtt t})^2}{\bar{\mathtt t}(\log\bar{\mathtt t})^{\beta}}
 \cdot
 (
 \tfrac{
 \lambda_{[{\sf w}]}(\bar{\mathtt t})^2
 }{\bar{\mathtt t}(\log\bar{\mathtt t})^{\beta'}}
 )^2.
 \end{align*}
 Hence,
 the integral in \eqref{EQ_7.35} is bounded by
 \begin{align*}
 &
 \int_0^s
 e^{\omega_3\zeta}
 (
 \tfrac{\lambda_{[{\sf w}]}(\bar{\mathtt t})^2}{\bar{\mathtt t}(\log\bar{\mathtt t})^{\beta'}}
 )^2
 d\zeta
 -
 \left[
 \tfrac{e^{\omega_3\zeta}}{\omega_3}
 (
 \tfrac{\lambda_{[{\sf w}]}(\bar{\mathtt t})^2}{\bar{\mathtt t}(\log\bar{\mathtt t})^{\beta'}}
 )^2
 \right]_{s=0}^{s=s}
 \\
 \nonumber
 &
 =
 \int_0^s
 \tfrac{e^{\omega_3\zeta}}{\omega_3}
 \tfrac{d}{ds}
 (
 \tfrac{\lambda_{[{\sf w}]}(\bar{\mathtt t})^2}{\bar{\mathtt t}(\log\bar{\mathtt t})^{\beta'}}
 )^2
 d\zeta
 \\
 &<
 \int_0^s
 \tfrac{8A_1\lambda_{[{\sf w}]}(\bar{\mathtt t})^2}{\bar{\mathtt t}(\log\bar{\mathtt t})^{\beta}}
 \cdot
 \tfrac{e^{\omega_3\zeta}}{\omega_3}
 (
 \tfrac{
 \lambda_{[{\sf w}]}(\bar{\mathtt t})^2
 }{\bar{\mathtt t}(\log\bar{\mathtt t})}
 )^2
 d\zeta.
 \end{align*}
 Since we may assume
 $\tfrac{8A_1\lambda_{[{\sf w}]}(\bar{\mathtt t}(s))^2}{\bar{\mathtt t}(s)(\log\bar{\mathtt t}(s))^{\beta}}<\frac{1}{2}$
 for $s\in(0,\infty)$ (see \eqref{EQ_7.6} for the estimate of $\lambda_{[{\sf w}]}(t)$),
 we get a bound for the above integral:
 \begin{align}
 \nonumber
 \int_0^s
 e^{\omega_3\zeta}
 (
 \tfrac{\lambda_{[{\sf w}]}(\bar{\mathtt t})^2}{\bar{\mathtt t}(\log\bar{\mathtt t})^{\beta'}}
 )^2
 d\zeta
 &<
 2
 \left[
 \tfrac{e^{\omega_3\zeta}}{\omega_3}
 (
 \tfrac{\lambda_{[{\sf w}]}(\bar{\mathtt t})^2}{\bar{\mathtt t}(\log\bar{\mathtt t})^{\beta'}}
 )^2
 \right]_{\zeta=0}^{\zeta=s}
 \\
 \label{EQ_7.36} 
 &<
 \tfrac{2e^{\omega_3s}}{\omega_3}
 (
 \tfrac{\lambda_{[{\sf w}]}(\bar{\mathtt t}(s))^2}{\bar{\mathtt t}(s)(\log\bar{\mathtt t}(s))^{\beta'}}
 )^2
 \quad
 \text{for }
 s\in(0,\infty).
 \end{align}
 By substituting \eqref{EQ_7.36} into \eqref{EQ_7.35} and using the lower bound
 $\omega_3=\mu_3^{(R)}>k_3R^{-3}$ (see Lemma \ref{lemma_3.4}),
 we obtain
 \begin{align}
 \label{EQ_7.37}
 \|\epsilon_{{\sf in}}^\perp(s)\|_{L_{y}^2}
 &<
 \tfrac{C}{\omega_3}
 \tfrac{\lambda_{[{\sf w}]}(\bar{\mathtt t}(s))^2}{\bar{\mathtt t}(s)(\log\bar{\mathtt t}(s))^{\beta'}}
 <
 \tfrac{2CR^3}{k_3}
 \tfrac{\lambda_{[{\sf w}]}(\bar{\mathtt t}(s))^2}{\bar{\mathtt t}(s)(\log\bar{\mathtt t}(s))^{\beta'}}
 \quad
 \text{for }
 s\in(0,\infty).
 \end{align}
 Combining \eqref{EQ_7.33} - \eqref{EQ_7.34} and \eqref{EQ_7.37},
 and returning to the original time variable $t$,
 we conclude
 \begin{align}
 \label{EQ_7.38}
 \|\epsilon_{{\sf in}}(t)\|_{L_{y}^2}
 &<
 \tfrac{C(\log R+R^3)\lambda_{[{\sf w}]}(t)^2}{t(\log t)^{\beta'}}
 <
 \tfrac{CR^3\lambda_{[{\sf w}]}(t)^2}{t(\log t)^{\beta'}}
 \quad
 \text{for }
 t\in(t_I,\infty).
 \end{align}
 We now derive an $L^\infty$ bound for $\epsilon_{{\sf in}}(y,t)$,
 using the time variable $s$ introduced in \eqref{EQ_7.17}.
 Applying standard parabolic estimates to solutions of \eqref{EQ_7.14},
 we deduce the following $L^\infty$ bound:
 \begin{align}
 \nonumber
 &
 \|\epsilon_{{\sf in}}(y,s)\|_{L_y^\infty(B_R)}
 \\
 \nonumber
 &<
 C\sup_{s'\in(s-1,s)}
 \|\epsilon_{{\sf in}}(s')\|_{L_{y}^2(B_R)}
 \\
 \nonumber
 &\quad
 +
 C\sup_{s'\in(s-1,s)}
 \|(\text{inhomogeneous term in \eqref{EQ_7.14}})\|_{L_y^\infty(B_R)}
 \\
 \label{EQ_7.39}
 &<
 C\sup_{s'\in(s-1,s)}
 \tfrac{R^3\lambda_{[{\sf w}]}(\bar{\mathtt t}(s'))^2}{\bar{\mathtt t}(s')(\log\bar{\mathtt t}(s'))^{\beta'}}
 \quad
 \text{for }
 s\in(1,\infty)
 \end{align}
 and
 \begin{align}
 \nonumber
 &
 \|\epsilon_{{\sf in}}(y,s)\|_{L_y^\infty(B_R)}
 \\
 \nonumber
 &<
 C\sup_{s'\in(0,s)}
 \|\epsilon_{{\sf in}}(s')\|_{L_{y}^2(B_R)}
 \\
 \nonumber
 &\quad
 +
 C\sup_{s'\in(0,s)}
 \|(\text{inhomogeneous term in \eqref{EQ_7.14}})\|_{L_y^\infty(B_R)}
 \\
 \label{EQ_7.40}
 &<
 C\sup_{s'\in(0,s)}
 \tfrac{R^3\lambda_{[{\sf w}]}(\bar{\mathtt t}(s'))^2}{\bar{\mathtt t}(s')(\log\bar{\mathtt t}(s'))^{\beta'}}
 \quad
 \text{for }
 s\in(0,1).
 \end{align}
 Applying \eqref{EQ_7.21} - \eqref{EQ_7.22} to
 the estimates \eqref{EQ_7.39} - \eqref{EQ_7.40},
 we obtain
 \begin{align*}
 \|\epsilon_{{\sf in}}(y,s)\|_{L_y^\infty(B_R)}
 <
 C
 \tfrac{R^3\lambda_{[{\sf w}]}(\bar{\mathtt t}(s))^2}{\bar{\mathtt t}(s)(\log\bar{\mathtt t}(s))^{\beta'}}
 \quad
 \text{for }
 s\in(0,\infty).
 \end{align*}
 Rewriting this estimate in terms of the original time variable $t$, 
 we have
 \begin{align}
 \label{EQ_7.41}
 \|\epsilon_{{\sf in}}(y,t)\|_{L_y^\infty(B_R)}
 <
 C
 \tfrac{R^3\lambda_{[{\sf w}]}(t)^2}{t(\log t)^{\beta'}}
 \quad
 \text{for }
 t\in(t_I,\infty).
 \end{align}
 We recall from \eqref{EQ_7.32} and Lemma \ref{lemma_3.2} that
 \begin{align}
 \label{EQ_7.42}
 |\epsilon_0(y)|
 &=
 |
 \alpha_I
 \psi_1^{(R)}(y)
 |
 <
 \tfrac{C\lambda_{[{\sf w}]}(t_I)^2}{t_I(\log t_I)^{\beta'}}
 (1+|y|)^{-\frac{n-1}{2}}
 e^{-|y|\sqrt{e_0}}
 \\
 \nonumber
 &
 \text{for }
 y\in B_{R}. 
 \end{align}
 Note that $\epsilon_{{\sf in}}(y,t)$ satisfies
 \begin{align*}
 \lambda_{[{\sf w}]}^2\pa_t\epsilon_{{\sf in}}=H_{y}\epsilon_{{\sf in}}
 \quad
 \text{for }
 y\in {B}_{R\setminus2{\sf r}_0},\ t\in(t_I,\infty). 
 \end{align*}
 We employ a comparison argument, as in \eqref{EQ_7.15},
 taking
 \begin{align*}
 \bar\epsilon_{{\sf in}}(y,t)=\bar C
 R^3\lambda_{[{\sf w}]}^2t^{-1}(\log t)^{-\beta'}(2{\sf r}_0|y|^{-1})^{\frac{7}{2}}
 \end{align*}
 as a comparison function.
 Using the boundary estimate from \eqref{EQ_7.41} on $\pa{B}_{R\setminus2{\sf r}_0}$,
 together with the initial time estimate \eqref{EQ_7.42},
 we deduce that
 \begin{align}
 \label{EQ_7.43}
 |\epsilon_{{\sf in}}(y,t)|
 &<
 \tfrac{
 \bar C_\text{\sf in}(2{\sf r}_0)^\frac{7}{2}R^3\lambda_{[{\sf w}]}(t)^2}
 {t(\log t)^{\beta'}}
 |y|^{-\frac{7}{2}}
 \quad
 \text{for }
 y\in {B}_{R\setminus2{\sf r}_0},\
 t\in(t_I,\infty).
 \end{align}
 We now return to equation \eqref{EQ_7.1}, using the time variable $s$ defined in \eqref{EQ_7.17}.
 Recall that
 \[
 \epsilon
 =
 \epsilon_{{\sf e}1}
 \cdot
 (1-\chi_{{\sf r}_0})
 +
 \epsilon_{{\sf e}2}
 \cdot
 (1-\chi_{{\sf r}_0})
 +
 \epsilon_{{\sf in}}
 \]
 gives a solution of \eqref{EQ_7.1},
 and satisfies (see \eqref{EQ_7.15}, \eqref{EQ_7.16} and \eqref{EQ_7.43})
 \begin{align}
 \label{EQ_7.44}
 |\epsilon(y,t)|
 &<
 \tfrac{
 CR^3\lambda_{[{\sf w}]}(t)^2}
 {t(\log t)^{\beta'}(1+|y|^2)^{\frac{7}{4}}}
 +
 \tfrac{
 C
 \lambda_{[{\sf w}]}(t)^2
 }{
 t(\log t)^{\beta'}
 |y|^2
 }
 \quad
 \text{for }
 y\in {B}_{R},\
 t\in(t_I,\infty).
 \end{align}
 We now apply Lemma \ref{lemma_3.8},
 which provides gradient estimates for parabolic equations.
 Throughout this argument, we work in the rescaled time variable $s$.
 Note that $V(y)=p{\sf Q}(y)^{p-1}$ satisfies the assumptions on $V$ in Lemma \ref{lemma_3.8}.
 Hence,
 applying Lemma \ref{lemma_3.8} together with \eqref{EQ_7.44},
 we deduce that
 if
 $(y_0,s_0)$ satisfies
 $1<|y_0|<\frac{R}{2}$ and
 $\tfrac{|y_0|^2}{64}<s_0<\infty$,
 then
 \begin{align*}
 \nonumber
 &
 |\nabla_y\epsilon(y_0,s_0)|
 \\
 &<
 \tfrac{C}{|y_0|}
 \sup_{s\in(s_0-\frac{|y_0|^2}{64},s_0)}
 \sup_{y\in B(y_0,\frac{|y_0|}{4})}
 |\epsilon(y,s)|
 \\
 &\quad
 +
 \tfrac{C}{|y_0|}
 \sup_{s\in(s_0-\frac{|y_0|^2}{64},s_0)}
 \sup_{y\in B(y_0,\frac{|y_0|}{4})}
 |y_0|^2
 |(\text{inhomogeneous term in \eqref{EQ_7.1}})(y,s)|
 \\
 &<
 \tfrac{C}{|y_0|}
 \sup_{s\in(s_0-\frac{|y_0|^2}{64},s_0)}
 \left(
 \tfrac{
 R^3\lambda_{[{\sf w}]}(\bar{\mathtt t}(s))^2}
 {\bar{\mathtt t}(s)(\log\bar{\mathtt t}(s))^{\beta'}(1+|y_0|^2)^{\frac{7}{4}}}
 +
 \tfrac{
 \lambda_{[{\sf w}]}(\bar{\mathtt t}(s))^2
 }{
 \bar{\mathtt t}(s)(\log\bar{\mathtt t}(s))^{\beta'}
 (1+|y_0|^2)
 }
 \right)
 \\
 &\quad
 +
 \tfrac{C}{|y_0|}
 \sup_{s\in(s_0-\frac{|y_0|^2}{64},s_0)}
 \tfrac{
 \lambda_{[{\sf w}]}(\bar{\mathtt t}(s))^2
 }{
 \bar{\mathtt t}(s)(\log\bar{\mathtt t}(s))^{\beta'}
 (1+|y_0|^2)
 }.
 \end{align*}
 Using \eqref{EQ_7.21} - \eqref{EQ_7.22} and the fact that $|y_0|<R=\log\log t_I$,
 we get
 \begin{align}
 \label{EQ_7.45}
 &
 |\nabla_y\epsilon(y_0,s_0)|
 \\
 \nonumber
 &<
 \tfrac{C}{|y_0|}
 \left(
 \tfrac{
 R^3\lambda_{[{\sf w}]}(\bar{\mathtt t}(s_0))^2}
 {\bar{\mathtt t}(s_0)(\log\bar{\mathtt t}(s_0))^{\beta'}(1+|y_0|^2)^{\frac{7}{4}}}
 +
 \tfrac{
 \lambda_{[{\sf w}]}(\bar{\mathtt t}(s_0))^2
 }{
 \bar{\mathtt t}(s_0)(\log\bar{\mathtt t}(s_0))^{\beta'}
 (1+|y_0|^2)
 }
 \right)
 \\
 \nonumber
 &
 \text{for }
 1<|y_0|<\tfrac{R}{2}
 \text{ and }
 \tfrac{|y_0|^2}{64}<s_0<\infty.
 \end{align}
 To estimate $\nabla_y\epsilon(y,s)$ for
 $1<|y_0|<\frac{R}{2}$ and $0<s_0<\tfrac{|y_0|^2}{64}$,
 we introduce
 \begin{align*}
 \tilde \epsilon(y,s)
 &=
 \epsilon(y,s)
 -
 \epsilon_0(y)
 =
 \epsilon(y,s)
 -
 \alpha_I
 \psi_1^{(R)}(y).
 \end{align*}
 Note that $\tilde\epsilon(y,s)|_{s=0}=0$ for $y\in B_R$
 and that
 $\tilde\epsilon(y,s)$ satisfies equation \eqref{EQ_7.1} with an additional inhomogeneous term
 $H_y\epsilon_0(y)$ $(=\alpha_IH_y\psi_1^{(R)})$.
 We now apply the second case of Lemma \ref{lemma_3.8}.
 Proceeding as in the previous case,
 we get
 \begin{align}
 \nonumber
 &
 |\nabla_y\epsilon(y_0,s_0)|
 \\
 \nonumber
 &<
 \tfrac{C}{|y_0|}
 \sup_{s\in(0,s_0)}
 \sup_{y\in B(y_0,\frac{|y_0|}{4})}
 |\epsilon(y,s)|
 \\
 \nonumber
 &\quad
 +
 \tfrac{C}{|y_0|}
 \sup_{s\in(0,s_0)}
 \sup_{y\in B(y_0,\frac{|y_0|}{4})}
 |y_0|^2
 |(\text{inhomogeneous term in \eqref{EQ_7.1}})(y,s)|
 \\
 \nonumber
 &\quad
 +
 \tfrac{C}{|y_0|}
 \sup_{y\in B(y_0,\frac{|y_0|}{4})}
 |y_0|^2
 |\alpha_IH_y\psi_1^{(R)}(y)|
 \\
 \nonumber
 &<
 \tfrac{C}{|y_0|}
 \sup_{s\in(0,s_0)}
 \left(
 \tfrac{
 R^3\lambda_{[{\sf w}]}(\bar{\mathtt t}(s))^2}
 {\bar{\mathtt t}(s)(\log\bar{\mathtt t}(s))^{\beta'}(1+|y_0|^2)^{\frac{7}{4}}}
 +
 \tfrac{
 \lambda_{[{\sf w}]}(\bar{\mathtt t}(s))^2
 }{
 \bar{\mathtt t}(s)(\log\bar{\mathtt t}(s))^{\beta'}
 (1+|y_0|^2)
 }
 \right)
 \\
 \nonumber
 &\quad
 +
 \tfrac{C}{|y_0|}
 \sup_{s\in(0,s_0)}
 \tfrac{
 \lambda_{[{\sf w}]}(\bar{\mathtt t}(s))^2
 }{
 \bar{\mathtt t}(s)(\log\bar{\mathtt t}(s))^{\beta'}
 (1+|y_0|^2)
 }
 +
 \tfrac{C}{|y_0|}
 \tfrac{
 \lambda_{[{\sf w}]}(t_I)^2
 }{
 t_I(\log t_I)^{\beta'}
 }
 \tfrac{
 |y_0|^2e^{-|y_0|\sqrt{e_0}}
 }{
 (1+|y_0|^2)^{\frac{n-1}{4}}
 }
 \\
 \label{EQ_7.46}
 &<
 \tfrac{C}{|y_0|}
 \left(
 \tfrac{
 R^3\lambda_{[{\sf w}]}(\bar{\mathtt t}(s_0))^2}
 {\bar{\mathtt t}(s_0)(\log\bar{\mathtt t}(s_0))^{\beta'}(1+|y_0|^2)^{\frac{7}{4}}}
 +
 \tfrac{
 \lambda_{[{\sf w}]}(\bar{\mathtt t}(s_0))^2
 }{
 \bar{\mathtt t}(s_0)(\log\bar{\mathtt t}(s_0))^{\beta'}
 (1+|y_0|^2)
 }
 \right)
 \\
 \nonumber
 &
 \text{for }
 1<|y_0|<\tfrac{R}{2}
 \text{ and }
 0<s_0<\tfrac{|y_0|^2}{64}.
 \end{align}
 Therefore,
 combining \eqref{EQ_7.45} - \eqref{EQ_7.46},
 and rewriting the result in terms of the original time variable $t$,
 we arrive at
 \begin{align}
 \label{EQ_7.47}
 |\nabla_y\epsilon(y,t)|
 &<
 \tfrac{
 CR^3\lambda_{[{\sf w}]}(t)^2}{t(\log t)^{\beta'}|y|^{\frac{9}{2}}}
 +
 \tfrac{
 C\lambda_{[{\sf w}]}(t)^2}{t(\log t)^{\beta'}|y|^3}
 \\
 \nonumber
 &
 \text{for }
 1<|y|<\tfrac{R}{2}
 \text{ and }
 t\in(t_I,\infty).
 \end{align}

\section{Analysis of the Outer Solution}
\label{section_8}
 In this section,
 we solve the outer problem \eqref{equation_6.14}.
 For convenience,
 we restate it below:
 \begin{align}
 \label{equation_8.1}
 \begin{cases}
 w_t
 =
 \Delta_xw
 +
 \lambda_{[{\sf w}]}^{-2}
 (
 \tfrac{\dot\lambda_{[{\sf w}]}}{\lambda_{[{\sf w}]}}
 -
 \tfrac{5{\sf b}}{4}
 )
 (\Lambda_{y}{\sf Q})
 \cdot
 (\chi_1-\chi_{\text{in}})
 \\
 \qquad
 +
 \tfrac{f'({\sf Q})}{\lambda_{[{\sf w}]}^2}
 {\sf w}_\text{ex}
 \cdot
 (\chi_1-\chi_{\text{in}})
 +
 N[\lambda_{[{\sf w}]},\epsilon_{[{\sf w}]},{\sf w}_\text{ex}]
 +
 g[\lambda_{[{\sf w}]}]
 \\
 \qquad
 +
 h[\lambda_{[{\sf w}]},\epsilon_{[{\sf w}]}]
 +
 k[{\sf w}_\text{ex}]
 \\
 \quad
 \text{for } (x,t)\in\R^n\times(t_{I},\infty),
 \\
 w|_{t=t_{I}}=0
 \quad
 \text{for }
 x\in\R^n.
 \end{cases}
 \end{align}
 Note that equation \eqref{equation_8.1} contains only one unknown function $w(x,t)$,
 while the functions
 ${\sf w}_\text{ex}(x,t)$, $\lambda_{[{\sf w}]}(t)$ and $\epsilon_{[{\sf w}]}(y,t)$ are all given.
 Let $\tilde w_{[{\sf w}]}(x,t)$ denote the unique solution to \eqref{equation_8.1}.
 As stated in Section \ref{section_6.2},
 to complete the proof of Theorem~\ref{theorem_3}, it suffices to verify that
 \[
 \tilde w_{[{\sf w}]}\in X_\tau
 \quad
 \text{for any given }
 {\sf w}\in X_\tau.
 \]
 Using the estimates for $\epsilon_{[\sf w]}(y,t)$ and $\lambda_{[{\sf w}]}(t)$
 established in Section~\ref{section_7},
 together with the asymptotic behavior of $\theta(x,t)$ and ${\sf b}(t)=\theta(0,t)$ described in
 in Section~\ref{section_4},
 and the definition of $X_\tau$ given in Section~\ref{section_6.2},
 we find that all the functions
 $\theta(x,t)$, ${\sf b}(t)$, ${\sf w}_\text{ex}(x,t)$, $\lambda_{[{\sf w}]}(t)$,
 and $\epsilon_{[{\sf w}]}(y,t)$ satisfy the assumptions (A1) - (A8), (B1) - (B3), and (D1),
 as listed in the Appendix.
 Therefore,
 all the estimates stated in Lemmas~\ref{lemma_9.1} - \ref{lemma_9.4} are applicable.
 In particular,
 there exist $\delta_1>0$, $\delta_2>0$, $\delta_3>0$ and $C>0$
 such that
 \begin{align}
 \label{equation_8.2}
 \left|
 \sum_{j=1}^{12}g_j[\lambda_{[{\sf w}]}]
 \right|
 &<
 \tfrac{1}{\lambda_{[{\sf w}]}^2}
 \tfrac{C}{t^{1+\delta_1}(\log t)^{\beta}}
 {\bf 1}_{|y|<1}
 +
 \tfrac{1}{\lambda_{[{\sf w}]}^2}
 \tfrac{C|y|^{-(2+\delta_2)}}{t^{1+\delta_1}(\log t)^\beta}
 {\bf 1}_{1<|y|<2(\frac{\sqrt t}{\lambda_{[{\sf w}]}})^{1-\kappa}}
 \\
 \nonumber
 &\quad
 +
 \tfrac{C}{t^2(\log t)^{2\beta}} 
 {\bf 1}_{|x|<\sqrt t}
 +
 \tfrac{C}{|x|^4(\log|x|^2)^{2\beta}}
 {\bf 1}_{|x|>\sqrt t}
 \\
 \nonumber
 &\text{for all }
 (y,t)\in\R^n\times(t_I,\infty),
 \\
 \label{equation_8.3}
 \left|
 \sum_{j=1}^4
 h_j[\epsilon_{[{\sf w}]}]
 \right|
 &<
 \tfrac{1}{\lambda_{[{\sf w}]}^2}
 \tfrac{C}{t^{1+\delta_1}(\log t)^{\beta'}}
 {\bf 1}_{|y|<1}
 +
 \tfrac{1}{\lambda_{[{\sf w}]}^2}
 \tfrac{C|y|^{-(2+\delta_2)}}{t^{1+\delta_1}(\log t)^{\beta'}}
 {\bf 1}_{1<|y|<\frac{R}{2}}
 \\
 \nonumber
 &\quad
 +
 \tfrac{1}{\lambda_{[{\sf w}]}^2}
 \tfrac{CR^{-\delta_3}|y|^{-(2+\delta_2)}}{t(\log t)^{\beta'}}
 {\bf 1}_{1<|y|<\frac{R}{2}}
 \\
 \nonumber
 &\text{for all }
 (y,t)\in\R^n\times(t_I,\infty),
 \\
 \label{equation_8.4}
 |k_1[{\sf w}_\text{ex}]|
 &<
 \tfrac{C}{t^2(\log t)^{\beta+\beta'}}
 {\bf 1}_{|x|<\sqrt t}
 +
 |x|^{-4}
 (\log|x|^2)^{-(\beta+\beta')}
 {\bf 1}_{|x|>\sqrt t}
 \\
 \nonumber
 &\text{for all }
 (x,t)\in\R^n\times(t_I,\infty)
 \end{align}
 and
 \begin{align}
 \label{equation_8.5}
 \left|
 N[\lambda_{[{\sf w}]},\epsilon_{[{\sf w}]},{\sf w}_\text{ex}]
 \right|
 &<
 \tfrac{1}{\lambda_{[{\sf w}]}^2}
 \tfrac{C|y|^{-(2+\delta_2)}}{t^{1+\delta_1}(\log t)^\beta}
 {\bf 1}_{(\frac{\sqrt t}{\lambda_{[{\sf w}]}})^{1-\kappa}<|y|<2(\frac{\sqrt t}{\lambda_{[{\sf w}]}})^{1-\kappa}}
 \\
 \nonumber
 &\quad
 +
 \tfrac{C}{t^2(\log t)^{2\beta}}
 {\bf 1}_{|x|<\sqrt t}
 +
 \tfrac{C}{|x|^4(\log|x|^2)^{2\beta}}
 {\bf 1}_{|x|>\sqrt t}
 \\
 \nonumber
 &\text{for all }
 (x,t)\in\R^n\times(t_I,\infty).
 \end{align}
 We emphasize that all the constants $\delta_1$, $\delta_2$, $\delta_3$ and $C$
 are independent of $R$ and $t_I$.
 Furthermore,
 the second and third terms on the right-hand side of \eqref{equation_8.1}
 can be estimated as
 \begin{align}
 \nonumber
 &
 |
 \lambda_{[{\sf w}]}^{-2}
 (\tfrac{\dot\lambda_{[{\sf w}]}}{\lambda_{[{\sf w}]}}
 -
 \tfrac{5{\sf b}}{4}
 )
 \cdot
 (\Lambda_{y}{\sf Q})
 \cdot
 (\chi_1-\chi_{\text{in}})
 |
 \\
 \nonumber
 &<
 \tfrac{1}{\lambda_{[{\sf w}]}^2}
 \tfrac{C|y|^{-4}}{t(\log t)^{\beta'}}
 {\bf 1}_{\frac{R}{4}<|y|<(\frac{\sqrt t}{\lambda_{[{\sf w}]}})^{1-\kappa}}
 \\
 \label{equation_8.6}
 &<
 \tfrac{1}{\lambda_{[{\sf w}]}^2}
 \tfrac{CR^{-1}|y|^{-3}}{t(\log t)^{\beta'}}
 {\bf 1}_{\frac{R}{4}<|y|<(\frac{\sqrt t}{\lambda_{[{\sf w}]}})^{1-\kappa}}
 \end{align}
 and
 \begin{align}
 \nonumber
 |
 \tfrac{f'({\sf Q})}{\lambda_{[{\sf w}]}^2}
 {\sf w}_\text{ex}
 \cdot
 (\chi_1-\chi_{\text{in}})
 |
 &<
 \tfrac{1}{\lambda_{[{\sf w}]}^2}
 \tfrac{C|y|^{-4}}{t(\log t)^{\beta'}}
 {\bf 1}_{\frac{R}{4}<|y|<(\frac{\sqrt t}{\lambda_{[{\sf w}]}})^{1-\kappa}}
 \\
 \label{equation_8.7}
 &<
 \tfrac{1}{\lambda_{[{\sf w}]}^2}
 \tfrac{CR^{-1}|y|^{-3}}{t(\log t)^{\beta'}}
 {\bf 1}_{\frac{R}{4}<|y|<(\frac{\sqrt t}{\lambda_{[{\sf w}]}})^{1-\kappa}}.
 \end{align}
 Combining \eqref{equation_8.1} - \eqref{equation_8.7},
 we can rewrite \eqref{equation_8.1} in the form:
 \begin{align}
 \label{equation_8.8}
 \begin{cases}
 \dis
 w_t
 =
 \Delta_xw
 +
 \sum_{j=1}^3
 F_j(x,t)
 &
 \text{for } (x,t)\in\R^n\times(t_{I},\infty),
 \\
 w|_{t=t_{I}}=0
 &
 \text{for }
 x\in\R^n.
 \end{cases}
 \end{align}
 The inhomogeneous terms
 $F_j(x,t)$ ($j=1,2,3$) are functions satisfying
 \begin{align}
 \label{EQ_8.9}
 |F_1(x,t)|
 &<
 \tfrac{1}{\lambda_{[{\sf w}]}^2}
 \tfrac{C}{t^{1+\delta_1}(\log t)^{\beta}}
 {\bf 1}_{|y|<1}
 +
 \tfrac{1}{\lambda_{[{\sf w}]}^2}
 \tfrac{C|y|^{-(2+\delta_2)}}{t^{1+\delta_1}(\log t)^\beta}
 {\bf 1}_{1<|y|<2(\frac{\sqrt t}{\lambda_{[{\sf w}]}})^{1-\kappa}}
 \\
 \nonumber
 &\quad
 +
 \tfrac{1}{\lambda_{[{\sf w}]}^2}
 \tfrac{CR^{-\delta_3}|y|^{-(2+\delta_2)}}{t(\log t)^{\beta'}}
 {\bf 1}_{1<|y|<\frac{R}{2}},
 \\
 \label{EQ_8.10}
 |F_2(x,t)|
 &<
 \tfrac{C}{t^2(\log t)^{2\beta}} 
 {\bf 1}_{|x|<\sqrt t},
 \\
 \label{EQ_8.11}
 |F_3(x,t)|
 &<
 \tfrac{C}{|x|^4(\log|x|^2)^{2\beta}}
 {\bf 1}_{|x|>\sqrt t}.
 \end{align}
 To derive \eqref{EQ_8.9} - \eqref{EQ_8.10},
 we used the fact that $\beta'>\beta$.
 Let $\bar {\sf r}_1$ be a small positive constant.
 We first consider the localized problem:
 \begin{align}
 \label{EQ_8.12}
 \begin{cases}
 \dis
 \pa_tw_1
 =
 \Delta_xw_1
 +
 F_1(x,t)
 &
 \text{for } x\in B_{\bar{\sf r}_1\sqrt t},\ t\in(t_{I},\infty),
 \\
 w_1=0
 &
 \text{for }
 x\in\pa B_{\bar{\sf r}_1\sqrt t},\
 t\in(t_I,\infty),
 \\
 w_1|_{t=t_{I}}=0
 &
 \text{for }
 x\in B_{\bar{\sf r}_1\sqrt t}.
 \end{cases}
 \end{align}
 Let us define
 \begin{align*}
 \bar w_1(x,t)
 &=
 \tfrac{\bar C_1(1+|y|^2)^{-\frac{\delta_2}{2}}}{t^{1+\delta_1}(\log t)^\beta}
 +
 \tfrac{\bar C_1R^{-\delta_3}(1+|y|^2)^{-\frac{\delta_2}{2}}}{t(\log t)^{\beta'}}.
 \end{align*}
 One can verify that
 there exist positive constants $D_j>0$ ($j=1,2,3$)
 depending only on $\delta_1$, $\delta_2$, $\beta$ and $\beta'$
 such that
 \begin{align*}
 \pa_t\bar w_1-\Delta_x\bar w_1
 &>
 (
 -
 \tfrac{D_1}{t}
 -
 D_2
 |
 \tfrac{\dot\lambda_{[{\sf w}]}}{\lambda_{[{\sf w}]}}
 |
 +
 \tfrac{D_3}{\lambda_{[{\sf w}]^2}}
 \tfrac{1}{1+|y|^2})
 \bar w_1
 \\
 &>
 \begin{cases}
 \tfrac{1}{\lambda_{[{\sf w}]}^2}
 (
 -
 \tfrac{D_1\lambda_{[{\sf w}]}^2}{t}
 -
 D_2
 |
 \lambda_{[{\sf w}]}
 \dot\lambda_{[{\sf w}]}
 |
 +
 \tfrac{D_3}{2}
 )
 \bar w_1
 &
 \text{for }
 |y|<1,
 \\
 (
 -
 \tfrac{D_1}{t}
 -
 D_2
 |
 \tfrac{\dot\lambda_{[{\sf w}]}}{\lambda_{[{\sf w}]}}
 |
 +
 \tfrac{D_3}{2|x|^2})
 \bar w_1
 &
 \text{for }
 |y|>1.
 \end{cases}
 \end{align*}
 Recall from \eqref{EQ_7.10} and \eqref{EQ_7.9} that 
 $|\frac{\dot\lambda_{[{\sf w}]}}{\lambda_{[{\sf w}]}}|<\frac{C}{t(\log t)^\beta}$
 and 
 $\lambda_{[{\sf w}]} < C t^{1/8}$.
 Therefore,
 if $\bar{\sf r}_1<\sqrt{\frac{D_3}{4D_1}}$,
 then
 \begin{align}
 \label{EQ_8.13}
 \pa_t\bar w_1-\Delta_x\bar w_1
 &>
 \begin{cases}
 \tfrac{D_3}{2\lambda_{[{\sf w}]}^2}
 \bar w_1
 &
 \text{for }
 |y|<1
 \\
 \tfrac{D_3}{4|x|^2}
 \bar w_1
 &
 \text{for }
 |y|>1,\
 |x|<\bar{\sf r}_1\sqrt t
 \end{cases}
 \quad
 \text{for }
 t\in(t_I,\infty).
 \end{align}
 As a consequence,
 there exists a constant $\bar C_1^*>0$, depending on $D_j$ ($j=1,2,3$) in \eqref{EQ_8.13}
 and on $C$ in \eqref{EQ_8.9},
 such that
 if $\bar C_1>\bar C_1^*$,
 then
 we have
 \begin{align}
 \label{EQ_8.14}
 \pa_t\bar w_1-\Delta_x\bar w_1
 >
 |F_1(x,t)|
 \quad
 \text{for }
 x\in B_{\bar{\sf r}_1\sqrt t},\
 t\in(t_I,\infty).
 \end{align}
 By a comparison argument,
 we deduce that
 \begin{align}
 \nonumber
 |w_1(x,t)|
 &<
 \bar w_1(x,t)
 \label{EQ_8.15}
 <
 \tfrac{\bar C_1(1+|y|^2)^{-\frac{\delta_2}{2}}}{t^{1+\delta_1}(\log t)^\beta}
 +
 \tfrac{\bar C_1R^{-\delta_3}(1+|y|^2)^{-\frac{\delta_2}{2}}}{t(\log t)^{\beta'}}
 \\
 &
 \text{for }
 x\in B_{\bar{\sf r}_1\sqrt t},\
 t\in(t_I,\infty).
 \end{align}
 Applying Lemma \ref{lemma_3.8} with $\rho(t)=\bar{\sf r}_1\sqrt{t}$
 to the solution $w_1(x,t)$ of \eqref{EQ_8.12},
 we obtain
 \begin{align*}
 \nonumber
 &
 |\nabla_xw_1(x_0,t_0)|
 \\
 &<
 \tfrac{C}{|x_0|}
 \sup_{t\in(t_0-\frac{|x_0|^2}{64},t_0)}
 \sup_{x\in B(x_0,\frac{|x_0|}{4})}
 \left(
 |w_1(x,t)|
 +
 |x_0|^2
 |F_1(x,t)| 
 \right)
 \\
 &
 \text{for all }
 1<|x_0|<\tfrac{{\bar{\sf r}_1}}{2}\sqrt{t_0}
 \text{ and }
 t_0\in(t_I+\tfrac{|x_0|^2}{64},\infty)
 \end{align*}
 and
 \begin{align*}
 \nonumber
 &
 |\nabla_xw_1(x_0,t_0)|
 \\
 &<
 \tfrac{C}{|x_0|}
 \sup_{t\in(t_I,t_0)}
 \sup_{x\in B(x_0,\frac{|x_0|}{4})}
 \left(
 |w_1(x,t)|
 +
 |x_0|^2
 |F_1(x,t)| 
 \right)
 \\
 &
 \text{for all }
 1<|x_0|<\tfrac{{\bar{\sf r}_1}}{2}\sqrt{t_0}
 \text{ and }
 t_0\in(t_I,t_I+\tfrac{|x_0|^2}{64}).
 \end{align*}
 Combining these with \eqref{EQ_8.9} and \eqref{EQ_8.15},
 we obtain
 \begin{align*}
 \nonumber
 &
 |\nabla_xw_1(x_0,t_0)|
 \\
 &<
 \tfrac{C}{|x_0|^{1+\delta_2}}
 \sup_{t\in(t_0-\frac{|x_0|^2}{64},t_0)}
 \left(
 \tfrac{\lambda_{[{\sf w}]}(t)^{\delta_2}}{t^{1+\delta_1}(\log t)^\beta}
 +
 \tfrac{R^{-\delta_3}\lambda_{[{\sf w}]}(t)^{\delta_2}}{t(\log t)^{\beta'}}
 \right)
 \\
 &
 \text{for all }
 1<|x_0|<\tfrac{{\bar{\sf r}_1}}{2}\sqrt{t_0}
 \text{ and }
 t_0\in(t_I+\tfrac{|x_0|^2}{64},\infty)
 \end{align*}
 and
 \begin{align*}
 \nonumber
 &
 |\nabla_xw_1(x_0,t_0)|
 \\
 &<
 \tfrac{C}{|x_0|^{1+\delta_2}}
 \sup_{t\in(t_I,t_0)}
 \left(
 \tfrac{\lambda_{[{\sf w}]}(t)^{\delta_2}}{t^{1+\delta_1}(\log t)^\beta}
 +
 \tfrac{R^{-\delta_3}\lambda_{[{\sf w}]}(t)^{\delta_2}}{t(\log t)^{\beta'}}
 \right)
 \\
 &
 \text{for all }
 1<|x_0|<\tfrac{{\bar{\sf r}_1}}{2}\sqrt{t_0}
 \text{ and }
 t_0\in(t_I,t_I+\tfrac{|x_0|^2}{64}).
 \end{align*}
 As in \eqref{EQ_7.21},
 we verify that
 \begin{align}
 \label{EQ_8.16}
 \log\tfrac{\lambda_{[{\sf w}]}(t+\Delta t)}{\lambda_{[{\sf w}]}(t)}
 <
 \tfrac{C\Delta t}{t(\log t)^\beta}
 \quad
 \text{for }
 \Delta t\in(-t+t_I,\infty).
 \end{align}
 This implies
 \begin{align*}
 \sup_{t\in(t_0-\frac{|x_0|^2}{64},t_0)}
 \lambda_{[{\sf w}]}(t)
 &<
 C
 \lambda_{[{\sf w}]}(t_0)
 \quad
 \text{if }
 t_0\in(t_I+\tfrac{|x_0|^2}{64},\infty),
 \\
 \sup_{t\in(t_I,t_0)}
 \lambda_{[{\sf w}]}(t)
 &<
 C
 \lambda_{[{\sf w}]}(t_0)
 \quad
 \text{if }
 t_0\in(t_I,t_I+\tfrac{|x_0|^2}{64}).
 \end{align*}
 Therefore,
 it follows that
 \begin{align}
 \label{EQ_8.17}
 |\nabla_xw_1(x_0,t_0)|
 &<
 \tfrac{C}{|x_0|^{1+\delta_2}}
 \left(
 \tfrac{\lambda_{[{\sf w}]}(t_0)^{\delta_2}}{t_0^{1+\delta_1}(\log t_0)^\beta}
 +
 \tfrac{R^{-\delta_3}\lambda_{[{\sf w}]}(t_0)^{\delta_2}}{t_0(\log t_0)^{\beta'}}
 \right)
 \\
 \nonumber
 &
 \text{for all }
 1<|x_0|<\tfrac{{\bar{\sf r}_1}}{2}\sqrt{t_0}
 \text{ and }
 t_0\in(t_I,\infty).
 \end{align}
 We now construct a solution $w(x,t)$ of \eqref{equation_8.8} using the function
 $w_1(x,t)$ obtained previously.
 Let $\chi_{\bar{\sf r}_1}=\chi(\frac{4|x|}{\bar{\sf r}_1\sqrt t})$,
 and express $w(x,t)$ as
 \begin{align}
 \label{EQ_8.18}
 w(x,t)
 =
 w_1(x,t)
 \chi_{\bar{\sf r}_1}
 +
 w_{2,3}(x,t).
 \end{align}
 Then,
 $w_{2,3}(x,t)$ satisfies
 \begin{align}
 \label{EQ_8.19}
 \begin{cases}
 \dis
 \pa_t
 w_{2,3}
 =
 \Delta_x
 w_{2,3}
 +
 2\nabla_xw_1\cdot\nabla_x\chi_{\bar{\sf r}_1}
 +
 w_1\Delta_x\chi_{\bar{\sf r}_1}
 -
 w_1\pa_t\chi_{\bar{\sf r}_1}
 \\
 \qquad
 \qquad
 +
 F_2
 +
 F_3
 \quad
 \text{for } (x,t)\in\R^n\times(t_{I},\infty),
 \\
 w_{2,3}|_{t=t_{I}}=0
 \quad
 \text{for }
 x\in\R^n.
 \end{cases}
 \end{align}
 From \eqref{EQ_8.10} - \eqref{EQ_8.11}, \eqref{EQ_8.15} and \eqref{EQ_8.17},
 we can estimate the inhomogeneous terms on the right-hand side of \eqref{EQ_8.19}.
 \begin{align*}
 &
 |\nabla_xw_1\cdot\nabla_x\chi_{\bar{\sf r}_1}|
 +
 |w_1\Delta_x\chi_{\bar{\sf r}_1}|
 \\
 &<
 \tfrac{C}{|x|^{2+\delta_2}}
 \left(
 \tfrac{\lambda_{[{\sf w}]}(t)^{\delta_2}}{t^{1+\delta_1}(\log t)^\beta}
 +
 \tfrac{R^{-\delta_3}\lambda_{[{\sf w}]}(t)^{\delta_2}}{t(\log t)^{\beta'}}
 \right)
 {\bf 1}_{\frac{\bar{\sf r}_1}{4}\sqrt t<|x|<\frac{\bar{\sf r}_1}{2}\sqrt t}
 \\
 &<
 \tfrac{C}{t}
 (\tfrac{\lambda_{[{\sf w}]}(t)}{\sqrt t})^{\delta_2}
 \left(
 \tfrac{1}{t^{1+\delta_1}(\log t)^\beta}
 +
 \tfrac{R^{-\delta_3}}{t(\log t)^{\beta'}}
 \right)
 {\bf 1}_{\frac{\bar{\sf r}_1}{4}\sqrt t<|x|<\frac{\bar{\sf r}_1}{2}\sqrt t},
 \\
 &
 |w_1\pa_t\chi_{\bar{\sf r}_1}|
 \\
 &<
 \tfrac{C}{t|x|^{\delta_2}}
 \left(
 \tfrac{\lambda_{[{\sf w}]}(t)^{\delta_2}}{t^{1+\delta_1}(\log t)^\beta}
 +
 \tfrac{R^{-\delta_3}\lambda_{[{\sf w}]}(t)^{\delta_2}}{t(\log t)^{\beta'}}
 \right)
 {\bf 1}_{\frac{\bar{\sf r}_1}{4}\sqrt t<|x|<\frac{\bar{\sf r}_1}{2}\sqrt t}
 \\
 &<
 \tfrac{C}{t}
 (\tfrac{\lambda_{[{\sf w}]}(t)}{\sqrt t})^{\delta_2}
 \left(
 \tfrac{1}{t^{1+\delta_1}(\log t)^\beta}
 +
 \tfrac{R^{-\delta_3}}{t(\log t)^{\beta'}}
 \right)
 {\bf 1}_{\frac{\bar{\sf r}_1}{4}\sqrt t<|x|<\frac{\bar{\sf r}_1}{2}\sqrt t}.
 \end{align*}
 Recall that $\lambda_{[{\sf w}]}(t)<Ct^\frac{1}{4}$ (see \eqref{EQ_7.9}).
 Therefore,
 we have
 \begin{align*}
 &
 |\nabla_xw_1\cdot\nabla_x\chi_{\bar{\sf r}_1}|
 +
 |w_1\Delta_x\chi_{\bar{\sf r}_1}|
 +
 |w_1\pa_t\chi_{\bar{\sf r}_1}|
 \\
 &<
 \tfrac{C}{t^{1+\frac{\delta_2}{4}}}
 \left(
 \tfrac{1}{t^{1+\delta_1}(\log t)^\beta}
 +
 \tfrac{R^{-\delta_3}}{t(\log t)^{\beta'}}
 \right)
 {\bf 1}_{\frac{\bar{\sf r}_1}{4}\sqrt t<|x|<\frac{\bar{\sf r}_1}{2}\sqrt t}.
 \end{align*}
 Applying Lemma \ref{lemma_3.9} - Lemma \ref{lemma_3.10} to a solution $w_{2,3}(x,t)$ of \eqref{EQ_8.19},
 we obtain
 \begin{align}
 \label{EQ_8.20}
 |w_{2,3}(x,t)|
 &<
 \tfrac{C}{t^{1+\delta_1+\frac{\delta_2}{4}}(\log t)^\beta}
 {\bf 1}_{|x|<\sqrt t}
 +
 \tfrac{CR^{-\delta_3}}{t^{1+\frac{\delta_2}{4}}(\log t)^{\beta'}}
 {\bf 1}_{|x|<\sqrt t}
 \\
 \nonumber
 &\quad
 +
 \tfrac{C}{|x|^{2+2\delta_1+\frac{\delta_2}{2}}(\log|x|^2)^\beta}
 {\bf 1}_{|x|<\sqrt t}
 +
 \tfrac{CR^{-\delta_3}}{|x|^{2+\frac{\delta_2}{2}}(\log t)^{\beta'}}
 {\bf 1}_{|x|<\sqrt t}
 \\
 \nonumber
 &\quad
 +
 \tfrac{C}{t(\log t)^{2\beta}}
 {\bf 1}_{|x|<\sqrt t}
 +
 C|x|^{-2}
 (\log|x|^2)^{-2\beta}
 {\bf 1}_{|x|>\sqrt t}
 \\
 \nonumber
 &\quad
 \text{for }
 (x,t)\in\R^n\times(t_I,\infty).
 \end{align}
 Combining \eqref{EQ_8.15} and \eqref{EQ_8.20},
 we finally arrive at
 \begin{align*}
 \nonumber
 &
 |w(x,t)|
 =
 |w_1(x,t)\chi_{\bar{\sf r}_1}+w_{2,3}(x,t)|
 \\
 &<
 \tfrac{C(1+|y|^2)^{-\frac{\delta_2}{2}}}{t^{1+\delta_1}(\log t)^\beta}
 {\bf 1}_{|x|<\bar{\sf r}_1\sqrt t}
 +
 \tfrac{CR^{-\delta_3}(1+|y|^2)^{-\frac{\delta_2}{2}}}{t(\log t)^{\beta'}}
 {\bf 1}_{|x|<\bar{\sf r}_1\sqrt t}
 \\
 &\quad
 +
 \tfrac{C}{t^{1+\delta_1+\frac{\delta_2}{4}}(\log t)^\beta}
 {\bf 1}_{|x|<\sqrt t}
 +
 \tfrac{CR^{-\delta_3}}{t^{1+\frac{\delta_2}{4}}(\log t)^{\beta'}}
 {\bf 1}_{|x|<\sqrt t}
 \\
 \nonumber
 &\quad
 +
 \tfrac{C}{|x|^{2+2\delta_1+\frac{\delta_2}{2}}(\log|x|^2)^\beta}
 {\bf 1}_{|x|<\sqrt t}
 +
 \tfrac{CR^{-\delta_3}}{|x|^{2+\frac{\delta_2}{2}}(\log t)^{\beta'}}
 {\bf 1}_{|x|<\sqrt t}
 \\
 \nonumber
 &\quad
 +
 \tfrac{C}{t(\log t)^{2\beta}}
 {\bf 1}_{|x|<\sqrt t}
 +
 C|x|^{-2}
 (\log|x|^2)^{-2\beta}
 {\bf 1}_{|x|>\sqrt t}
 \\
 \nonumber
 &\quad
 \text{for }
 (x,t)\in\R^n\times(t_I,\infty).
 \end{align*}
 This ensures that
 $w(x,t)\in X_\tau$.
 Therefore, the construction is complete, and the proof of Theorem~\ref{theorem_3} is concluded.

 \section{Appendix}
 \label{section_9}
 \subsection{List of Functions $g$, $h$ and $k$}
 \label{section_9.1}
 We summarize the definitions of $g_i$, $h_i$ and $k_i$.
 \begin{enumerate}[(g1)]
 \item 
 $g_1[\lambda]
 =
 \lambda^{-\frac{n-2}{2}}
 (-{\sf Q})
 \dot\chi_1$
 
 \item
 $g_2[\lambda]
 =
 \tfrac{d}{dt}
 \{
 \tfrac{-5{\sf b}T_1}{4}
 \chi_1
 \}$
 
 \item
 $g_3[\lambda]
 =
 \theta
 \dot\chi_1$
 
 \item
 $g_4[\lambda]
 =
 \lambda^{-\frac{n-2}{2}}
 \tfrac{2(\nabla_{y}{\sf Q}\cdot\nabla_{y}\chi_1)+{\sf Q}(\Delta_{y}\chi_1)}{\lambda^2}$

 \item
 $g_5[\lambda]
 =
 \tfrac{2\nabla_{y}\frac{5{\sf b} T_1}{4}\cdot\nabla_{y}\chi_1}{\lambda^2}$

 \item
 $g_6[\lambda]
 =
 (
 \tfrac{5{\sf b} T_1}{4}
 -
 \theta
 )
 \tfrac{(\Delta_{y}\chi_1)}{\lambda^2}$
 
 \item
 $g_7[\lambda]
 =
 \tfrac{2\nabla_x\theta\cdot(-\nabla_{y}\chi_1)}{\lambda}$

 \item
 $g_8[\lambda]
 =
 \lambda^{-\frac{n+2}{2}}
 f({\sf Q})
 (\chi_1^2-\chi_1)$
 
 \item
 $g_9[\lambda]
 =
 f(\theta)
 (1-\chi_1)^2$

 \item
 $g_{10}[\lambda]
 =
 \tfrac{f'({\sf Q})}{\lambda^2}
 (\tfrac{5{\sf b} T_1}{4}-\theta)
 \chi_1
 (\chi_1-1)$

 \item
 $g_{11}[\lambda]
 =
 f'(\theta)
 \lambda^{-\frac{n-2}{2}}
 {\sf Q}
 \chi_1(1-\chi_1)$
 
 \item
 $g_{12}[\lambda]
 =
 f'(\theta)
 \tfrac{5{\sf b} T_1}{4}
 \chi_1(1-\chi_1)$
 \end{enumerate}
 \begin{enumerate}[(h1)]
 \item
 $h_1[\lambda,\epsilon]
 =
 \lambda^{-\frac{n}{2}}
 \dot
 \lambda
 (\Lambda_{y}\epsilon)
 \chi_{\text{in}}$
 
 \item
 $h_2[\lambda,\epsilon]
 =
 \lambda^{-\frac{n-2}{2}}
 (-\epsilon)
 \dot\chi_{\text{in}}$
 
 \item
 $h_3[\lambda,\epsilon]
 =
 2\lambda^{-\frac{n+2}{2}}
 (\nabla_{y}\epsilon\cdot\nabla_{y}\chi_{\text{in}})$

 \item
 $h_4[\lambda,\epsilon]
 =
 \lambda^{-\frac{n+2}{2}}
 \epsilon
 (\Delta_{y}\chi_{\text{in}})$
 \end{enumerate}
 \begin{enumerate}[(k1)]
 \item
 $k_1[w]
 =
 f'(\theta)
 w
 (1-\chi_1)$
 \end{enumerate}
 Throughout this section,
 $\chi_1$ and $\chi_\text{in}$ denote the cut off functions given by
 \begin{itemize}
 \item
 $\chi_1
 =
 \chi(\tfrac{|x|}{ \lambda(t)^\kappa(\sqrt t)^{1-\kappa} })$
 \quad
 \text{for some constant }
 $\kappa\in(0,1)$,
 \item
 $\chi_{\text{in}}
 =
 \chi(\frac{4|x|}{R\lambda(t)})$.
 \end{itemize}
 Let us use
 a notation
 \begin{align*}
 {\bf 1}_{|x|\sim r}
 =
 {\bf 1}_{r<|x|<2r}.
 \end{align*}
 Since the inner variable $y$ is defined by $y=\frac{x}{\lambda(t)}$,
 we have the equivalence:
 \begin{align}
 \label{equation_9.1}
 {\bf 1}_{|y|\sim(\frac{\sqrt t}{\lambda})^{1-\kappa}}
 \ \overset{\text{equivalent}}{\Longleftrightarrow} \
 {\bf 1}_{|x|\sim\lambda^\kappa(\sqrt t)^{1-\kappa}}
 \end{align}
 Under these assumptions,
 it holds that
 \begin{align*}
 |\dot\chi_1|
 &<
 Ct^{-1}
 {\bf 1}_{|y|\sim(\frac{\sqrt t}{\lambda})^{1-\kappa}},
 \\
 |\nabla_y\chi_1|
 &<
 C|y|^{-1}
 {\bf 1}_{|y|\sim(\frac{\sqrt t}{\lambda})^{1-\kappa}},
 \\
 |\Delta_y\chi_1|
 &<
 C|y|^{-2}
 {\bf 1}_{|y|\sim(\frac{\sqrt t}{\lambda})^{1-\kappa}},
 \\
 |\dot\chi_\text{in}|
 &<
 Ct^{-1}
 {\bf 1}_{|y|\sim\frac{R}{4}},
 \\
 |\nabla_y\chi_\text{in}|
 &<
 C|y|^{-1}
 {\bf 1}_{|y|\sim\frac{R}{4}},
 \\
 |\Delta_y\chi_\text{in}|
 &<
 C|y|^{-2}
 {\bf 1}_{|y|\sim\frac{R}{4}}.
 \end{align*}
 \begin{lem}
 \label{lemma_9.1}
 Let $n=6$, $\beta>0$ and $t_0>e$,
 and
 assume that there exists $C_1>0$ such that the following conditions hold{\rm:}
 \begin{enumerate}[\rm({\rm A}1)]
 \item ${\sf b}(t)=\theta(0,t)$
 \quad \text{for all} $t\in(t_0,\infty)$
 \item
 $|{\sf b}(t)|<\frac{C_1}{t(\log t)^\beta}$
 \quad \text{for all} $t\in(t_0,\infty)$
 \item
 $|\dot{\sf b}(t)|<\frac{C_1}{t^2(\log t)^\beta}$
 \quad \text{for all} $t\in(t_0,\infty)$
 \item
 $\frac{|\dot\lambda(t)|}{\lambda(t)}<\frac{C_1}{t(\log t)^\beta}$
 \quad \text{for all} $t\in(t_0,\infty)$
 \item
 $0<\lambda(t)<e^{C_1(\log t)^{1-\beta}}$
 \quad \text{for all} $t\in(t_0,\infty)$
 \item
 $|\theta(x,t)|<\frac{C_1}{t(\log t)^\beta}$ \quad {\rm for all} $|x|<\sqrt t$, $t>t_0$
 \item
 $|\theta(x,t)|<\frac{C_1}{|x|^2(\log|x|^2)^\beta}$ \quad {\rm for all} $|x|>\sqrt t$, $t>t_0$
 \item
 $|\nabla_x\theta(x,t)|<\frac{|x|}{t}\frac{C_1}{t(\log t)^\beta}$ \quad {\rm for all}
 $|x|<1$, $t>t_0$
 \end{enumerate}
 Then,
 for any $\kappa\in(0,\frac{1}{2})$,
 there exist $\delta_1>0$, $\delta_2>0$ and $C>0$ such that
 \begin{align*}
 \left|
 \sum_{j=1}^{12}g_j[\lambda]
 \right|
 &<
 \tfrac{1}{\lambda^2}
 \tfrac{C}{t^{1+\delta_1}(\log t)^{\beta}}
 {\bf 1}_{|y|<1}
 +
 \tfrac{1}{\lambda^2}
 \tfrac{C|y|^{-(2+\delta_2)}}{t^{1+\delta_1}(\log t)^\beta}
 {\bf 1}_{1<|y|<2(\frac{\sqrt t}{\lambda})^{1-\kappa}}
 \\
 &\quad
 +
 \tfrac{C}{t^2(\log t)^{2\beta}} 
 {\bf 1}_{|x|<\sqrt t}
 +
 \tfrac{C}{|x|^4(\log|x|^2)^{2\beta}}
 {\bf 1}_{|x|>\sqrt t}
 \\
 &\text{for all }
 (y,t)\in\R^n\times(t_0,\infty).
 \end{align*}
 All the constants $\delta_1$, $\delta_2$ and $C$ depend only on $\beta$, $\kappa$ and $C_1$
 {\rm(}as defined in {\rm (A1) - (A8)}{\rm)}.
 \end{lem}
 \begin{proof}
 Note that
 \begin{itemize}
 \item $f(u)=|u|u$ and $f'(u)=|u|$ when $n=6$,
 \item $T_1(y)=\frac{4}{5}+O(|y|^{-2})$ as $|y|\to\infty$ when $n=6$ (see \eqref{e_5.5}),
 \item $|\nabla_yT_1(y)|<C|y|^{-3}$ as $|y|\to\infty$ when $n=6$ (see \eqref{e_5.6}),
 \end{itemize}
 Using the above three facts and assumptions (A1) - (A8),
 we verify that 
 \begin{align*}
 |g_1|
 &<
 Ct^{-1}
 \lambda^{-2}
 |y|^{-4}
 {\bf 1}_{|y|\sim(\frac{\sqrt t}{\lambda})^{1-\kappa}},
 \\
 |g_2|
 &<
 \tfrac{C}{t^2(\log t)^\beta}
 {\bf 1}_{|y|<2(\frac{\sqrt t}{\lambda})^{1-\kappa}},
 \\
 |g_3|
 &<
 \tfrac{C}{t^2(\log t)^\beta}
 {\bf 1}_{|y|\sim(\frac{\sqrt t}{\lambda})^{1-\kappa}},
 \\
 |g_4|
 &<
 C\lambda^{-4}
 |y|^{-6}
 {\bf 1}_{|y|\sim(\frac{\sqrt t}{\lambda})^{1-\kappa}},
 \end{align*}
 \begin{align*}
 |g_5|
 &<
 \tfrac{C}{t(\log t)^\beta}
 \lambda^{-2}
 |y|^{-4}
 {\bf 1}_{|y|\sim(\frac{\sqrt t}{\lambda})^{1-\kappa}},
 \\
 |g_6|
 &<
 \tfrac{C}{t(\log t)^\beta}
 \lambda^{-2}
 |y|^{-4}
 {\bf 1}_{|y|\sim(\frac{\sqrt t}{\lambda})^{1-\kappa}}
 \\
 &\quad
 +
 \tfrac{C}{t^2(\log t)^\beta}
 {\bf 1}_{|y|\sim(\frac{\sqrt t}{\lambda})^{1-\kappa}},
 \\
 |g_7|
 &<
 \tfrac{C}{t^2(\log t)^\beta}
 {\bf 1}_{|y|\sim(\frac{\sqrt t}{\lambda})^{1-\kappa}},
 \\
 |g_8|
 &<
 C
 \lambda^{-4}
 |y|^{-8}
 {\bf 1}_{|y|\sim(\frac{\sqrt t}{\lambda})^{1-\kappa}},
 \end{align*}
 \begin{align*}
 |g_9|
 &<
 \tfrac{C}{t^2(\log t)^{2\beta}}
 {\bf 1}_{\lambda^\kappa(\sqrt t)^{1-\kappa}<|x|<\sqrt t}
 \\
 &\quad
 +
 \tfrac{C}{|x|^4(\log|x|^2)^{2\beta}}
 {\bf 1}_{|x|>\sqrt t},
 \\
 |g_{10}|
 &<
 \tfrac{C}{t(\log t)^\beta}
 \lambda^{-2}
 |y|^{-6}
 {\bf 1}_{|y|\sim(\frac{\sqrt t}{\lambda})^{1-\kappa}}
 \\
 &\quad
 +
 \tfrac{C}{t^2(\log t)^\beta}
 |y|^{-2}
 {\bf 1}_{|y|\sim(\frac{\sqrt t}{\lambda})^{1-\kappa}},
 \\
 |g_{11}|
 &<
 \tfrac{C}{t(\log t)^\beta}
 \lambda^{-2}
 |y|^{-4}
 {\bf 1}_{|y|\sim(\frac{\sqrt t}{\lambda})^{1-\kappa}},
 \\
 |g_{12}|
 &<
 \tfrac{C}{t^2(\log t)^{2\beta}}
 {\bf 1}_{|y|\sim(\frac{\sqrt t}{\lambda})^{1-\kappa}}.
 \end{align*}
 To estimate $g_i$,
 it suffices to derive the following bounds:
 \begin{itemize}
 \item Estimate for $g_1$:
 \begin{align*}
 &t^{-1}
 \lambda^{-2}
 |y|^{-4}
 {\bf 1}_{|y|\sim(\frac{\sqrt t}{\lambda})^{1-\kappa}}
 \\
 &=
 \lambda^{-2}
 \cdot
 t^{-1}
 |y|^{-2+\delta_1}
 |y|^{-(2+\delta_1)}
 {\bf 1}_{|y|\sim(\frac{\sqrt t}{\lambda})^{1-\kappa}}
 \\
 &<
 \lambda^{-2}
 \cdot
 (\tfrac{\lambda}{\sqrt t})^{(1-\kappa)(2-\delta)}
 \cdot
 \tfrac{1}{t}
 |y|^{-(2+\delta)}
 {\bf 1}_{|y|\sim(\frac{\sqrt t}{\lambda})^{1-\kappa}}.
 \end{align*}
 \item estimate for $g_2$ in $|y|<1$
 \begin{align*}
 \tfrac{1}{t^2(\log t)^\beta}
 &=
 \lambda^{-2}
 \cdot
 (\tfrac{\lambda}{\sqrt t})^2
 \cdot
 \tfrac{1}{t(\log t)^\beta},
 \end{align*}
 \item Estimate for $g_2$ in $|y|>1$:
 \begin{align*}
 &
 \tfrac{1}{t^2(\log t)^\beta}
 {\bf 1}_{|y|<2(\frac{\sqrt t}{\lambda})^{1-\kappa}}
 \\
 &=
 \lambda^{-2}
 \cdot
 \tfrac{\lambda^2}{t}
 |y|^{2+\delta}
 \cdot
 \tfrac{1}{t(\log t)^\beta}
 |y|^{-(2+\delta)}
 {\bf 1}_{|y|<2(\frac{\sqrt t}{\lambda})^{1-\kappa}}
 \\
 &<
 \lambda^{-2}
 \cdot
 (\tfrac{\lambda}{\sqrt t})^{2-(1-\kappa)(2+\delta)}
 \cdot
 \tfrac{1}{t(\log t)^\beta}
 |y|^{-(2+\delta)}
 {\bf 1}_{|y|<2(\frac{\sqrt t}{\lambda})^{1-\kappa}}.
 \end{align*}
 \item Estimate for $g_3$, the last term of $g_6$, $g_7$, the last term of $g_{10}$:
 \begin{align*}
 &
 \tfrac{1}{t^2(\log t)^\beta}
 {\bf 1}_{|y|\sim(\frac{\sqrt t}{\lambda})^{1-\kappa}}
 \\
 &=
 \lambda^{-2}
 \cdot
 \tfrac{\lambda^2}{t}
 |y|^{2+\delta}
 \cdot
 \tfrac{1}{t(\log t)^\beta}
 |y|^{-(2+\delta)}
 {\bf 1}_{|y|\sim(\frac{\sqrt t}{\lambda})^{1-\kappa}}
 \\
 &<
 \lambda^{-2}
 \cdot
 (\tfrac{\lambda}{\sqrt t})^{2-(1-\kappa)(2+\delta)}
 \cdot
 \tfrac{1}{t(\log t)^\beta}
 |y|^{-(2+\delta)}
 {\bf 1}_{|y|\sim(\frac{\sqrt t}{\lambda})^{1-\kappa}}.
 \end{align*}
 \item Estimate for $g_4$ and $g_8$:
 \begin{align*}
 &
 \lambda^{-4}
 |y|^{-6}
 {\bf 1}_{|y|\sim(\frac{\sqrt t}{\lambda})^{1-\kappa}}
 \\
 &=
 \lambda^{-2}
 \cdot
 \tfrac{t}{\lambda^2}
 |y|^{-4+\delta}
 \cdot
 \tfrac{1}{t}
 |y|^{-(2+\delta)}
 {\bf 1}_{|y|\sim(\frac{\sqrt t}{\lambda})^{1-\kappa}}
 \\
 &<
 \lambda^{-2}
 \cdot
 (\tfrac{\lambda}{\sqrt t})^{-2+(1-\kappa)(4-\delta)}
 \cdot
 \tfrac{1}{t}
 |y|^{-(2+\delta)}
 {\bf 1}_{|y|\sim(\frac{\sqrt t}{\lambda})^{1-\kappa}}.
 \end{align*}
 \item Estimate for $g_5$, the first term of $g_6$, the first term of $g_{10}$ and $g_{11}$:
 \begin{align*}
 &
 \tfrac{1}{t(\log t)^\beta}
 \lambda^{-2}
 |y|^{-4}
 {\bf 1}_{|y|\sim(\frac{\sqrt t}{\lambda})^{1-\kappa}}
 \\
 &=
 \lambda^{-2}
 \cdot
 |y|^{-2+\delta}
 \cdot
 \tfrac{1}{t(\log t)^\beta}
 |y|^{-(2+\delta)}
 {\bf 1}_{|y|\sim(\frac{\sqrt t}{\lambda})^{1-\kappa}}
 \\
 &<
 \lambda^{-2}
 \cdot
 (\tfrac{\lambda}{\sqrt t})^{(1-\kappa)(2-\delta)}
 \cdot
 \tfrac{1}{t(\log t)^\beta}
 |y|^{-(2+\delta)}
 {\bf 1}_{|y|\sim(\frac{\sqrt t}{\lambda})^{1-\kappa}}.
 \end{align*}
 \end{itemize}
 We choose parameters $\kappa>0$ and $\delta>0$ so that
 \begin{align}
 \label{equation_9.2}
 \begin{cases}
 (1-\kappa)(2-\delta)>0,
 \\
 2-(1-\kappa)(2+\delta)>0,
 \\
 -2+(1-\kappa)(4-\delta)>0.
 \end{cases}
 \end{align}
 For any given $\kappa\in(0,\frac{1}{2})$,
 there exists $\delta_*>0$ such that if $\delta\in(0,\delta_*)$,
 then the conditions in \eqref{equation_9.2} holds.
 From assumption (A5),
 we have
 $\frac{\lambda}{\sqrt t}<Ct^{-\frac{1}{4}}$
 for some constant $C>0$ depending only on $\beta$ and $C_1$.
 Using this inequality together with the estimates established above,
 we obtain the required bounds.
 The proof is complete.
 \end{proof}

 \begin{lem}
 \label{lemma_9.2}
 Let $n=6$, $\beta'>0$ and $t_0>e$.
 Assume {\rm (A4) - (A5)}.
 Suppose further that
 there exists $C_2>0$ such that
 a parameter $R$ and a function $\epsilon(y,t)$ satisfy the following conditions{\rm:}
 \begin{enumerate}[\rm({\rm B}1)]
 \setlength{\itemsep}{2mm}
 \item
 $1<R<2\log\log t_0$
 
 \item
 $\tfrac{|\epsilon(y,t)|}{\lambda(t)^2}
 <
 \tfrac{C_2R^3}{t(\log t)^{\beta'}}
 {\bf 1}_{|y|<1}
 +
 \tfrac{C_2}{t(\log t)^{\beta'}}
 \tfrac{1}{|y|^2}
 {\bf 1}_{1<|y|<R}
 +
 \tfrac{C_2}{t(\log t)^{\beta'}}
 \tfrac{R^3}{|y|^\frac{7}{2}}
 {\bf 1}_{1<|y|<R}$
 \\
 \quad {\rm for all} $(y,t)\in B_R\times(t_0,\infty)$
 
 \item
 $\tfrac{|\nabla_y\epsilon(y,t)|}{\lambda(t)^2}
 <
 \tfrac{C_2R^3}{t(\log t)^{\beta'}}
 {\bf 1}_{|y|<1}
 +
 \tfrac{C_2}{t(\log t)^{\beta'}}
 \tfrac{1}{|y|^3}
 {\bf 1}_{1<|y|<R}
 +
 \tfrac{C_2}{t(\log t)^{\beta'}}
 \tfrac{R^3}{|y|^\frac{9}{2}}
 {\bf 1}_{1<|y|<R}$
 \\
 \quad {\rm for all} $(y,t)\in B_R\times(t_0,\infty)$
 \end{enumerate}
 Then,
 there exist constants $\delta_1,\delta_2,\delta_3>0$ and $C>0$ such that
 \begin{align*}
 \left|
 \sum_{j=1}^4
 h_j[\epsilon]
 \right|
 &<
 \tfrac{1}{\lambda(t)^2}
 \tfrac{C}{t^{1+\delta_1}(\log t)^{\beta'}}
 {\bf 1}_{|y|<1}
 +
 \tfrac{1}{\lambda(t)^2}
 \tfrac{C|y|^{-(2+\delta_2)}}{t^{1+\delta_1}(\log t)^{\beta'}}
 {\bf 1}_{1<|y|<\frac{R}{2}}
 \\
 &\quad
 +
 \tfrac{1}{\lambda(t)^2}
 \tfrac{CR^{-\delta_3}|y|^{-(2+\delta_2)}}{t(\log t)^{\beta'}}
 {\bf 1}_{1<|y|<\frac{R}{2}}
 \\
 &\text{for all }
 (y,t)\in \R^n\times(t_0,\infty).
 \end{align*}
 The constants $\delta_1$, $\delta_2$, $\delta_3$ and $C$ depend only on $\beta'$ and $C_2$,
 as given in conditions {\rm (B1) - (B3)}.
 \end{lem}
 \begin{proof}
 From (B2) - (B3) and (A4),
 we see that
 \begin{align*}
 |h_1[\epsilon]|
 &<
 \tfrac{C}{t(\log t)^\beta}
 \tfrac{|\Lambda_y\epsilon|}{\lambda^2}
 {\bf 1}_{|y|<\frac{R}{2}}
 \\
 &<
 \tfrac{CR^3}
 {t^2(\log t)^{\beta+\beta'}}
 {\bf 1}_{|y|<1}
 \\
 &\quad
 +
 \tfrac{C}{t^2(\log t)^{\beta+\beta'}}
 |y|^{-2}
 {\bf 1}_{1<|y|<\frac{R}{2}}
 \\
 &\quad
 +
 \tfrac{C}{t^2(\log t)^{\beta+\beta'}}
 R^3
 |y|^{-\frac{7}{2}}
 {\bf 1}_{1<|y|<\frac{R}{2}}
 \\
 &=
 \lambda^{-2}
 \cdot
 \tfrac{\lambda^2R^3}{t}
 \cdot
 \tfrac{C}
 {t(\log t)^{\beta+\beta'}}
 {\bf 1}_{|y|<1}
 \\
 &\quad
 +
 \lambda^{-2}
 \cdot
 \tfrac{\lambda^2}{t}
 \cdot
 \tfrac{C}{t(\log t)^{\beta+\beta'}}
 |y|^{-2}
 {\bf 1}_{1<|y|<\frac{R}{2}}
 \\
 &\quad
 +
 \lambda^{-2}
 \cdot
 \tfrac{\lambda^2R^3}{t}
 \cdot
 \tfrac{C}{t(\log t)^{\beta+\beta'}}
 |y|^{-\frac{7}{2}}
 {\bf 1}_{1<|y|<\frac{R}{2}},
 \\
 |h_2[\epsilon]|
 &<
 \lambda^{-2}
 \cdot
 \tfrac{\lambda^2}{t}
 \cdot
 \tfrac{C}{t(\log t)^{\beta'}}
 |y|^{-2}
 {\bf 1}_{|y|\sim\frac{R}{4}}
 \\
 &\quad
 +
 \lambda^{-2}
 \cdot
 \tfrac{\lambda^2R^3}{t}
 \cdot
 \tfrac{C}{t^2(\log t)^{\beta'}}
 |y|^{-\frac{7}{2}}
 {\bf 1}_{|y|\sim\frac{R}{4}},
 \end{align*}
 \begin{align*}
 |h_3[\epsilon]|
 &<
 \lambda^{-2}
 \cdot
 \tfrac{C}{t(\log t)^{\beta'}}
 |y|^{-4}
 {\bf 1}_{|y|\sim\frac{R}{4}}
 \\
 &\quad
 +
 \lambda^{-2}
 \cdot
 \tfrac{C}{t(\log t)^{\beta'}}
 R^3
 |y|^{-\frac{11}{2}}
 {\bf 1}_{|y|\sim\frac{R}{4}}
 \\
 &<
 \lambda^{-2}
 \cdot
 \tfrac{CR^{-1}}{t(\log t)^{\beta'}}
 |y|^{-3}
 {\bf 1}_{|y|\sim\frac{R}{4}}
 \\
 &\quad
 +
 \lambda^{-2}
 \cdot
 \tfrac{CR^{-\frac{1}{4}}}{t(\log t)^{\beta'}}
 |y|^{-\frac{9}{4}}
 {\bf 1}_{|y|\sim\frac{R}{4}},
 \\
 |h_4[\epsilon]|
 &<
 \lambda^{-2}
 \cdot
 \tfrac{CR^{-1}}{t(\log t)^{\beta'}}
 |y|^{-3}
 {\bf 1}_{|y|\sim\frac{R}{4}}
 \\
 &\quad
 +
 \lambda^{-2}
 \cdot
 \tfrac{CR^{-\frac{1}{4}}}{t(\log t)^{\beta'}}
 |y|^{-\frac{9}{4}}
 {\bf 1}_{|y|\sim\frac{R}{4}}.
 \end{align*}
 From (B1) and (A5),
 we easily see that
 \begin{itemize}
 \item 
 $\tfrac{\lambda^2R^3}{t}<t^{-\frac{3}{4}}$ for $t\in(t_0,\infty)$,
 
 \item
 $|y|{\bf 1}_{|y|<\frac{R}{2}}
 <R{\bf 1}_{|y|<\frac{R}{2}}<t^\frac{1}{32}{\bf 1}_{|y|<\frac{R}{2}}$.
 \end{itemize}
 Combining these estimates with the pointwise estimates for $h_j[\epsilon]$,
 we obtain the desired inequalities.
 The proof is completed.
 \end{proof}

 \begin{lem}
 \label{lemma_9.3}
 Let $n=6$, $\beta'>0$ and $t_0>e$.
 Assume {\rm (A6) - (A7)}, and suppose further that
 a function $w(x,t)$ satisfies the following pointwise bound for all
 $(x,t)\in\R^n\times(t_0,\infty)$.
 \begin{enumerate}[\rm({\rm D}1)]
 \item
 $|w(x,t)|
 <
 \begin{cases}
 \dis
 t^{-1}
 (\log t)^{-\beta'}
 & \text{for } |x|< \sqrt{t}, t\in(t_0,\infty)
 \\
 |x|^{-2}
 (\log|x|^2)^{-\beta'}
 & \text{for } |x|> \sqrt{t}, t\in(t_0,\infty)
 \end{cases}$
 \end{enumerate}
 Then,
 there exists a constant $C$
 such that
 \begin{align*}
 |k_1[w]|
 &<
 \tfrac{C}{t^2(\log t)^{\beta+\beta'}}
 {\bf 1}_{|x|<\sqrt t}
 +
 |x|^{-4}
 (\log|x|^2)^{-(\beta+\beta')}
 {\bf 1}_{|x|>\sqrt t}
 \\
 &
 \text{for }
 (x,t)\in\R^n\times(t_0,\infty).
 \end{align*}
 \end{lem}
 \begin{proof}
 Since the proof is straightforward, we omit the details.
 \end{proof}
 Finally,
 we provide an estimate for the nonlinear term $N[\lambda,\epsilon,w]$.
 Recall the definition of $N[\lambda,\epsilon,w]$ given in \eqref{equation_6.4}:
 \begin{align*}
 N[\lambda,\epsilon,w]
 &=
 f
 (
 \lambda^{-2}
 {\sf Q}
 \chi_1
 +
 \tfrac{5{\sf b}}{4}
 T_1
 \chi_1
 +
 \theta
 \chi_1^c
 +
 \lambda^{-2}
 \epsilon
 \chi_\text{in}
 +
 w
 )
 \\
 \nonumber
 &\quad
 -
 f(
 \lambda^{-2}{\sf Q}
 \chi_1
 )
 -
 f'(
 \lambda^{-2}{\sf Q}
 \chi_1
 )
 (
 \tfrac{5{\sf b}}{4}
 T_1
 \chi_1
 +
 \theta
 \chi_1^c
 +
 \lambda^{-2}
 \epsilon
 \chi_\text{in}
 +
 w
 )
 \\
 &\quad
 -
 f(\theta\chi_1^c)
 -
 f'(\theta\chi_1^c)
 (
 \lambda^{-2}
 {\sf Q}
 \chi_1
 +
 \tfrac{5{\sf b}}{4}
 T_1
 \chi_1
 +
 \lambda^{-2}
 \epsilon
 \chi_\text{in}
 +
 w
 ),
 \end{align*}
 where $\chi_1^c=1-\chi_1$.
 \begin{lem}
 \label{lemma_9.4}
 Let $n=6$, $\beta'>\beta>0$ and $t_0>e$.
 Assume {\rm (A1) - (A8)}, {\rm (B1) - (B3)} and {\rm (D1)}.
 Then,
 for any given $\kappa\in(0,1)$,
 there exist $\delta_1>0$, $\delta_2>0$ and $C>0$ such that
 \begin{align*}
 \left|
 N[\lambda,\epsilon,w]
 \right|
 &<
 \tfrac{1}{\lambda^2}
 \tfrac{C|y|^{-(2+\delta_2)}}{t^{1+\delta_1}(\log t)^\beta}
 {\bf 1}_{(\frac{\sqrt t}{\lambda})^{1-\kappa}<|y|<2(\frac{\sqrt t}{\lambda})^{1-\kappa}}
 \\
 &\quad
 +
 \tfrac{C}{t^2(\log t)^{2\beta}}
 {\bf 1}_{|x|<\sqrt t}
 +
 \tfrac{C}{|x|^4(\log|x|^2)^{2\beta}}
 {\bf 1}_{|x|>\sqrt t}
 \\
 &\text{for }
 (x,t)\in\R^n\times(t_0,\infty).
 \end{align*}
 All the constants $\delta_1$, $\delta_2$ and $C$ depend only on
 $\beta$, $\beta'$, $\kappa$, $C_1$ from {\rm (A1) - (A8)},
 and $C_2$ from {\rm (B1) - (B3)}.
 \end{lem}
 \begin{proof}
 Note that, for $n=6$, we have $f(u)=|u|u$.
 From the definition of $\chi_1$ (see below \eqref{equation_6.2}),
 it is clear that
 \begin{align*}
 \chi_1=
 \begin{cases}
 1 & \text{for } |y|<(\frac{\sqrt t}{\lambda})^{1-\kappa},
 \\
 0 & \text{for } |y|>2(\frac{\sqrt t}{\lambda})^{1-\kappa},
 \end{cases}
 \quad
 \chi_\text{in}=
 \begin{cases}
 1 & \text{for } |y|<\frac{R}{4},
 \\
 0 & \text{for } |y|>\frac{R}{2}.
 \end{cases}
 \end{align*}
 Using these definitions,
 the nonlinear term $N$ can be estimated as follows.
 \\
 (i) Region $|y|<1$, $t\in(t_0,\infty)$:
 \begin{align*}
 |N[\lambda,\epsilon,w]|
 &=
 |
 f
 (
 \lambda^{-2}
 {\sf Q}
 +
 \tfrac{5{\sf b} T_1}{4}
 +
 \lambda^{-2}
 \epsilon\chi_{\text{in}}
 +
 w
 )
 \\
 &\qquad
 -
 f(
 \lambda^{-2}
 {\sf Q}
 )
 -
 f(
 \lambda^{-2}
 {\sf Q}
 )
 (
 \tfrac{5{\sf b} T_1}{4}
 +
 \lambda^{-2}
 \epsilon\chi_{\text{in}}
 +
 w
 )
 |
 \\
 &<
 C
 (
 \tfrac{5{\sf b} T_1}{4}
 +
 \lambda^{-2}
 \epsilon\chi_{\text{in}}
 +
 w
 )^2
 \\
 &<
 \tfrac{C}{t^2(\log t)^{2\beta}}
 +
 \tfrac{CR^6}{t^2(\log t)^{2\beta'}}
 +
 \tfrac{C}{t^2(\log t)^{2\beta'}}
 \\
 &<
 \tfrac{C}{t^2(\log t)^{2\beta}},
 \quad \text{since } R<2\log\log t_0 \text{ and }\beta' > \beta.
 \end{align*}
 \\
 (ii) Region $1<|y|<\frac{R}{2}$, $t\in(t_0,\infty)$:
 \begin{align*}
 |N[\lambda,\epsilon,w]|
 &=
 |
 f
 (
 \lambda^{-2}
 {\sf Q}
 +
 \tfrac{5{\sf b} T_1}{4}
 +
 \lambda^{-2}
 \epsilon\chi_{\text{in}}
 +
 w
 )
 \\
 &\qquad
 -
 f(
 \lambda^{-2}
 {\sf Q}
 )
 -
 f'(
 \lambda^{-2}
 {\sf Q}
 )
 (
 \tfrac{5{\sf b} T_1}{4}
 +
 \lambda^{-2}
 \epsilon\chi_{\text{in}}
 +
 w
 )
 |
 \\
 &<
 C(
 \tfrac{5{\sf b} T_1}{4}
 +
 \lambda^{-2}
 \epsilon\chi_{\text{in}}
 +
 w
 )^2
 \\
 &<
 \tfrac{C}{t^2(\log t)^{2\beta}}
 +
 \tfrac{C|y|^{-4}}{t^2(\log t)^{2\beta'}}
 +
 \tfrac{CR^6|y|^{-7}}{t^2(\log t)^{2\beta'}}
 +
 \tfrac{C}{t^2(\log t)^{2\beta'}}
 \\
 &<
 \tfrac{C}{t^2(\log t)^{2\beta}},
 \quad \text{since } R<2\log\log t_0 \text{ and }\beta' > \beta.
 \end{align*}
 (iii) Region $\frac{R}{2}<|y|<(\frac{\sqrt t}{\lambda})^{1-\kappa}$, $t\in(t_0,\infty)$:
 \begin{align*}
 |N[\lambda,\epsilon,w]|
 &=
 |
 f(
 \lambda^{-2}
 {\sf Q}
 +
 \tfrac{5{\sf b} T_1}{4}
 +
 w
 )
 \\
 &\quad
 -
 f(
 \lambda^{-2}
 {\sf Q}
 )
 -
 f'(
 \lambda^{-2}
 {\sf Q}
 )
 (
 \tfrac{5{\sf b} T_1}{4}
 +
 w
 )
 |
 \\
 &<
 C(
 \tfrac{5{\sf b} T_1}{4}
 +
 w
 )^2
 \\
 &<
 \tfrac{C}{t^2(\log t)^{2\beta}}
 +
 \tfrac{C}{t^2(\log t)^{2\beta'}}
 \\
 &<
 \tfrac{C}{t^2(\log t)^{2\beta}},
 \quad \text{since } \beta' > \beta.
 \end{align*}
 (iv) Region $(\frac{\sqrt t}{\lambda})^{1-\kappa}<|y|<2(\frac{\sqrt t}{\lambda})^{1-\kappa}$, $t\in(t_0,\infty)$:
 \begin{align*}
 &
 |N[\lambda,\epsilon,w]|
 \\
 &=
 |
 f
 (
 \lambda^{-2}
 {\sf Q}
 \chi_1
 +
 \tfrac{5{\sf b}}{4}
 T_1
 \chi_1
 +
 \theta
 \chi_1^c
 +
 w
 )
 \\
 \nonumber
 &\quad
 -
 f(
 \lambda^{-2}{\sf Q}
 \chi_1
 \bigr)
 -
 f'\bigl(
 \lambda^{-2}{\sf Q}
 \chi_1
 )
 (
 \tfrac{5{\sf b}}{4}
 T_1
 \chi_1
 +
 \theta
 \chi_1^c
 +
 w
 )
 \\
 &\quad
 -
 f(\theta\chi_1^c)
 -
 f'(\theta\chi_1^c)
 (
 \lambda^{-2}
 {\sf Q}
 \chi_1
 +
 \tfrac{5{\sf b}}{4}
 T_1
 \chi_1
 +
 w
 )
 |
 \\
 &<
 |
 f(
 \lambda^{-2}{\sf Q}
 \chi_1
 )
 |
 +
 |
 f'(
 \lambda^{-2}{\sf Q}
 \chi_1
 )
 (
 \tfrac{5{\sf b}}{4}
 T_1
 \chi_1
 +
 \theta
 \chi_1^c
 +
 w
 )
 |
 \\
 &\quad
 +
 C(
 \lambda^{-2}
 {\sf Q}
 \chi_1
 +
 \tfrac{5{\sf b}}{4}
 T_1
 \chi_1
 +
 w
 )^2
 \\
 &<
 C\lambda^{-4}
 |y|^{-8}
 +
 \lambda^{-2}
 \tfrac{C|y|^{-4}}{t(\log t)^\beta}
 +
 \lambda^{-2}
 \tfrac{C|y|^{-4}}{t(\log t)^\beta}
 \\
 &\quad
 +
 \lambda^{-2}
 \tfrac{C|y|^{-4}}{t(\log t)^{\beta'}}
 +
 C
 \lambda^{-4}
 |y|^{-8}
 +
 \tfrac{C}{t^2(\log t)^{2\beta}}
 +
 \tfrac{C}{t^2(\log t)^{2\beta'}}
\\
 &<
 C\lambda^{-4}
 |y|^{-8}
 +
 \lambda^{-2}
 \tfrac{C|y|^{-4}}{t(\log t)^\beta}
 +
 \tfrac{C}{t^2(\log t)^{2\beta}}.
 \end{align*}
 (v) Region $|y|>2(\frac{\sqrt t}{\lambda})^{1-\kappa}$ and $|x|<\sqrt t$, $t\in(t_0,\infty)$:
 \begin{align*}
 |N[\lambda,\epsilon,w]|
 &=
 |
 f
 \bigl(
 \theta
 +
 w
 \bigr)
 -
 f\bigl(\theta\bigr)
 -
 f'\bigl(\theta\bigr)
 w
 |
 \\
 &<
 C
 w^2
 <
 \tfrac{C}{t^2(\log t)^{2\beta'}}
 {\bf 1}_{|x|<\sqrt t}
 \\
 &<
 \tfrac{C}{t^2(\log t)^{2\beta}}
 {\bf 1}_{|x|<\sqrt t},
 \quad
 \text{since }
 \beta'>\beta.
 \end{align*}
 (vi) Region $|x|>\sqrt t$, $t\in(t_0,\infty)$:
 \begin{align*}
 |N[\lambda,\epsilon,w]|
 &=
 |
 f(
 \theta
 +
 w
 )
 -
 f\bigl(\theta\bigr)
 -
 f'\bigl(\theta\bigr)
 w
 |
 \\
 &<
 C
 w^2
 <
 C
 |x|^{-4}
 (\log|x|^2)^{-2\beta'}
 {\bf 1}_{|x|>\sqrt t}. 
 \end{align*}
 Therefore,
 the pointwise estimates for $N[\lambda,\epsilon,w]$
 in all regions have been established.
 This completes the proof.
 \end{proof}

 \subsection{Proof for \eqref{e_5.22} - \eqref{e_5.25}}
 \label{section_9.2}
 In this subsection,
 we provide the proof for \eqref{e_5.22} - \eqref{e_5.25}.
 For convenience,
 we rewrite them as follows.
 For even $j\in\N$,
 \begin{align}
 \label{EQ_9.3}
 \log[\tfrac{\lambda(t_j^-)}{\lambda(t_1^-)}]
 &<
 -
 q_1(\log t_j^-)^{1-\beta}
 +
 3q_1(\log t_{j-1}^+)^{1-\beta}
 +
 \int_{t_1^+}^{t_j^-}
 |D(s)|ds,
 \\
 \label{EQ_9.4}
 \log[\tfrac{\lambda(t_j^-)}{\lambda(t_1^-)}]
 &>
 -
 q_1
 (\log t_j^-)^{1-\beta}
 -
 q_1(\log t_{j-1}^+)^{1-\beta}
 -
 \int_{t_1^+}^{t_j^-}
 |D(s)|ds.
 \end{align}
 For odd $j\in\N$,
 \begin{align}
 \label{EQ_9.5}
 \log[\tfrac{\lambda(t_j^-)}{\lambda(t_1^-)}]
 &<
 q_1(\log t_j^-)^{1-\beta}
 +
 q_1(\log t_{j-1}^+)^{1-\beta}
 +
 \int_{t_1^+}^{t_j^-}
 |D(s)|ds,
 \\
 \label{EQ_9.6}
 \log[\tfrac{\lambda(t_j^-)}{\lambda(t_1^-)}]
 &>
 q_1(\log t_j^-)^{1-\beta}
 -
 3q_1
 (\log t_{j-1}^+)^{1-\beta}
 -
 \int_{t_1^+}^{t_j^-}
 |D(s)|ds.
 \end{align}
 \begin{proof}
 The proofs of \eqref{EQ_9.3} - \eqref{EQ_9.6} are carried out by induction.
 Since the arguments are similar,
 we present only the proof for \eqref{EQ_9.3}. 
 Recall the definition of $D(t)$:
 \begin{align*}
 D(t)
 =
 \begin{cases}
 D_j(t)
 &
 \text{if }
 t\in(t_j^+,t_{j+1}^-),
 \\[2mm]
 0
 &
 \text{if }
 t\in\bigcup_{j\in\N}(t_j^-,t_j^+).
 \end{cases}
 \end{align*}
 Due to \eqref{e_5.21} and \eqref{e_5.18},
 for \( j = 2 \),
 we have
 \begin{align*}
 &
 \log[\tfrac{\lambda(t_2^-)}{\lambda(t_1^-)}]
 =
 \log[\tfrac{\lambda(t_2^-)}{\lambda(t_1^+)}]
 +
 \log[\tfrac{\lambda(t_1^+)}{\lambda(t_1^-)}]
 \\
 &<
 -
 q_1(\log t_2^-)^{1-\beta}
 +
 q_1(\log t_1^+)^{1-\beta}
 +
 \int_{t_1^+}^{t_2^-}
 D_1(s)
 ds
 \\
 &\quad
 +
 2q_1
 (\log t_1^+)^{1-\beta}
 -
 2q_1
 (\log t_1^-)^{1-\beta}
 \\
 &=
 -
 q_1(\log t_2^-)^{1-\beta}
 +
 3q_1(\log t_1^+)^{1-\beta}
 -
 2q_1(\log t_1^-)^{1-\beta}
 \\
 &\quad
 +
 \int_{t_1^+}^{t_2^-}
 D_1(s)
 ds.
 \end{align*}
 Thus,
 \eqref{EQ_9.3} holds for \(j = 2\).
 Suppose that \eqref{EQ_9.3} holds for a certain even $j=J\in\N$.
 Then,
 we have
 \begin{align*}
 &
 \log[\tfrac{\lambda(t_{J+2}^-)}{\lambda(t_1^-)}]
 \\
 &=
 \log[\tfrac{\lambda(t_{J+2}^-)}{\lambda(t_{J+1}^+)}]
 +
 \log[\tfrac{\lambda(t_{J+1}^+)}{\lambda(t_{J+1}^-)}]
 +
 \log[\tfrac{\lambda(t_{J+1}^-)}{\lambda(t_J^+)}]
 \\
 &\quad
 +
 \log[\tfrac{\lambda(t_J^+)}{\lambda(t_J^-)}]
 +
 \log[\tfrac{\lambda(t_J^-)}{\lambda(t_1^-)}]
 \\
 &<
 -
 q_1(\log t_{J+2}^-)^{1-\beta}
 +
 q_1(\log t_{J+1}^+)^{1-\beta}
 +
 \int_{t_{J+1}^+}^{t_{J+2}^-}
 D_{J+1}(s)
 ds
 \\
 &\quad
 +
 2q_1
 (\log t_{J+1}^+)^{1-\beta}
 -
 2q_1
 (\log t_{J+1}^-)^{1-\beta}
 \\
 &\quad
 +
 q_1(\log t_{J+1}^-)^{1-\beta}
 -
 q_1(\log t_{J}^+)^{1-\beta}
 +
 \int_{t_J^+}^{t_{J+1}^-}
 D_J(s)
 ds
 \\
 &\quad
 +
 2q_1
 (\log t_J^+)^{1-\beta}
 -
 2q_1
 (\log t_J^-)^{1-\beta}
 \\
 &\quad
 -
 q_1
 (\log t_J^-)^{1-\beta}
 +
 3q_1
 (\log t_{J-1}^+)^{1-\beta}
 +
 \int_{t_1^+}^{t_{J}^-}
 |D(s)|
 ds
 \\
 &<
 -
 q_1(\log t_{J+2}^-)^{1-\beta}
 +
 3q_1(\log t_{J+1}^+)^{1-\beta}
 \\
 &\quad
 -
 q_1
 (\log t_{J+1}^-)^{1-\beta}
 +
 q_1(\log t_{J}^+)^{1-\beta}
 \\
 &\quad
 -
 3q_1
 (\log t_J^-)^{1-\beta}
 +
 3q_1(\log t_{J-1}^+)^{1-\beta}
 +
 \int_{t_1^+}^{t_{J+2}^-}
 |D(s)|
 ds
 \\
 &<
 -
 q_1(\log t_{J+2}^-)^{1-\beta}
 +
 3q_1(\log t_{J+1}^+)^{1-\beta}
 +
 \int_{t_1^+}^{t_{J+2}^-}
 |D(s)|
 ds.
 \end{align*}
 This completes the inductive step, and therefore the proof of \eqref{EQ_9.3}.
 \end{proof}

 \newpage
 \subsection{List of Notations}
 We here collect notations for convenience.
 \begin{table}[h]
 \begin{tabular}{llcccc}
 \hline
 $f(u)$
 & nonlinear term defined by $f(u)=|u|^{p-1}u$
 \
 ($p=\frac{n+2}{n-2}$, $n=6$)
 \\
 $\nabla_x$, $\nabla_y$
 &
 standard nabla in each variable
 \\
 $\Delta_x$, $\Delta_y$
 &
 standard laplacian in each variable
 \\
 $\Lambda_y$
 & linear operator related to the scaling
 $u(x)\mapsto\lambda^\frac{2}{p-1}u(\lambda x)$
 \\
 &
 defined by $\Lambda_y=\frac{n-2}{2}+y\cdot\nabla_y$
 \
 (below \eqref{e_5.2} p. \pageref{e_5.2})
 \\
 $H_y$
 & linearized operator around ${\sf Q}(y)$
 \\
 &
 define by $H_y=\Delta_y+f'({\sf Q}(y))$
 \
 (below \eqref{e_5.2} p. \pageref{e_5.2})
 \\ \hline
 ${\sf Q}(y)$
 & positive solution of $-\Delta_y{\sf Q}={\sf Q}^p$
 \\
 $T_1(y)$
 & solution of $H_yT_1=-\Lambda_y{\sf Q}$
 \
 (\eqref{e_5.3} p. \pageref{e_5.3})
 \\
 $\theta(x,t)$
 & specific solution of
 $\pa_t\theta=\Delta_x\theta$ 
 \
 (see Section \ref{section_4})
 \\
 $\lambda(t)$ &
 scaling function depending on $t$
 \\
 ${\sf b}(t)$
 &
 value of $\theta(x,t)$ at the origin:
 \
 (beginning of Section \ref{section_6})
 \\ \hline
 \end{tabular}
 \begin{tabular}{llcccc}
 $\chi$ &
 standard cut off function defined by
 $\chi(\zeta)=\begin{cases}
 1 &\text{if } \zeta<1 \\
 0 &\text{if } \zeta>2
 \end{cases}$
 \\
 $\chi_1$ &
 cut off function defined by
 $\chi_1=\chi(\frac{|x|}{\lambda^\kappa(\sqrt t)^{1-\kappa}})$,
 where $\kappa\in(0,\frac{1}{2})$
 \\
 &
 is arbitrarily chosen
 \
 (see below \ref{equation_6.2} p. \pageref{equation_6.2})
 \\
 $\chi_1^c$
 &
 cut off function defined by
 $\chi_1^c
 =
 1-\chi_1$
 (beginning of Section \ref{section_6.1})
 \\
 $\chi_\text{in}$ &
 cut off function defined by
 $\chi_\text{in}=\chi(\frac{4|x|}{R\lambda})$
 \
 (below \eqref{equation_6.5} p. \pageref{equation_6.5})
 \\ \hline
 \end{tabular}
 \begin{tabular}{llcccc}
 $t_I$
 & initial time of our equation: $u_t=\Delta u+|u|^{p-1}u$
 \\
 $R$
 & large positive constant determining the size of the inner region
 \\
 &
 taken as $R=\log|\log t_I|$
 \\
 $\beta$
 &
 parameter describing the decay of the initial data:
 $\beta\in(\frac{1}{2},1)$
 \\
 &
 (statement of Theorems \ref{theorem_1} - Theorems \ref{theorem_2} and Section \ref{section_4})
 \\
 $\beta'$
 &
 parameter for defining a subset in the fixed point argument:
 \\
 &
 $\beta'\in(1,\beta+1)$
 \\ \hline
 \end{tabular}
 \\
 \begin{tabular}{llcccc}
 $B(x_0,R)=\{x\in\R^n;|x-x_0|<R\}$
 \\
 $B_R=\{x\in\R^n;|x|<R\}$
 \\
 $B_{R\setminus r}=\{x\in\R^n;r<|x|<R\}$
 \\
 $\pa B_{R\setminus r}=\{x\in\R^n;|x|=r\}\cup\{x\in\R^n;|x|=R\}$
 \\ \hline
 \end{tabular}
 \end{table}

\section*{Acknowledgement}
The author is partly supported by
Grant-in-Aid for Young Scientists (B) No. 26800065.


\end{document}